\documentclass[smallcondensed]{svjour3}       % onecolumn (second format)

\usepackage[colorlinks=true]{hyperref}
\usepackage{lineno}
\modulolinenumbers[5]

\usepackage{graphicx}
\usepackage{amsmath} 
\usepackage{amsfonts,verbatim,mathdots}
\usepackage{amssymb}
\usepackage{dsfont}
\usepackage{units}
\usepackage{array,multirow,makecell}

\usepackage{amsthm} % 

\usepackage{mathrsfs}
\usepackage{stmaryrd}
\usepackage{cite} %: replace with \usepackage[numbers,super]{natbib} in the following

\usepackage{rotating}
\usepackage{ifthen}
\usepackage{booktabs}
\usepackage{xcolor}
\usepackage{subcaption}
\RequirePackage{caption, float}
\usepackage{graphicx}
\usepackage{float}
\usepackage{microtype}
\usepackage{bigints}

\newcommand{\f}{{f}}
\newcommand{\n}{{n}}

\newcommand{\Gx}{{\Gamma}}

\newcommand{\N}{{\mathcal{N}}} 

\newcommand{\Gl}{ { \Gamma^\star_{\rm linear}   } } 
\newcommand{\muslashmu}{{\ensuremath{(\mu/\mu_{w}, \lambda)}}}

\newcommand{\markupdraftm}[2]{% {#1: {color|display} command}{#2: desired color or text}
    \ifthenelse{\equal{#1}{display}}{#2}{}%                 % display only in draft version
    \ifthenelse{\equal{#1}{color}}{\color{#2}}{}%           % colored only in draft (for \new command)
}
\newcommand{\newcolored}[3][]{{\markupdraft{color}{#2}#3}%  % kept in the final print
    \ifthenelse{\equal{#1}{}}{}{\markupdraft{display}{{\color{yellow!70!black}[#1]}}}} 

\newcommand{\del}[2][]{{\markupdraft{display}{{\color{yellow!90!black}[{\tiny removed:}{\small "#2"\ifthenelse{\equal{#1}{}}{}{[#1]}]}}}}} % (to be) removed
\newcommand{\new}[2][]{\newcolored[#1]{blue!80!black}{#2}}%  % kept in the final print
\newcommand{\nnew}[2][]{\newcolored[#1]{red!90!black}{#2}}%  % kept in the final print
\newcommand{\markupdraft}{\markupdraftm}% allows to change the command back and forth

\newcommand{\indraftonly}[1]{{#1}}  % draft mode, outcomment below renewcommand for final version
\renewcommand{\indraftonly}[1]{}\renewcommand{\markupdraft}[2]{}  % final mode text command

\renewcommand{\del}[2][]{}

\newcommand{\anne}[1]{\indraftonly{{\color{magenta}Anne: #1}}}
\newcommand{\niko}[2][]{\indraftonly{\color{green!60!black}~#2$_{-\!\mathrm{Niko}}$#1}}

\newcommand{\ba}{\begin{align}}
\newcommand{\ea}{\end{align}}
\newcommand{\baStar}{\begin{eqnarray*}}
\newcommand{\eaStar}{\end{eqnarray*}}

\newcommand{\todo}[1]{\indraftonly{\color{red} {\bf TODO} #1}}
\newcommand{\done}[1]{\indraftonly{\color{green} {\bf DONE} #1}}

\newcommand{\level}{\mathcal{L}}
\newcommand{\normalized}{$\sigma$-normalized}

     \newcommand{\foncfast}[4]{#2 \in #1 \mapsto #3 \in #4}
  \newcommand{\foncsmall}[2]{#1 \mapsto #2}
   \newcommand{\foncfleche}[2]{#1 \rightarrow #2}

\newcommand{\dsp}{\displaystyle}
\newcommand{\acco}[1]{\left\{#1\right\}} 
\newcommand{\pare}[1]{\left(#1\right)}
\newcommand{\croc}[1]{\left[#1\right]}
\newcommand{\abs}[1]{\left\lvert#1\right\rvert}

\def\R{{\mathbb R}}
\def\diff{{\mathrm d}}
\def\ZZ{{\mathcal Z}}

\journalname{Journal of Global Optimization}

\InputIfFileExists{toggles.tex}{}{} % file may contain
\begin{document}

%\begin{frontmatter}

%\title{Elsevier \LaTeX\ template\tnoteref{mytitlenote}}

\title{ 
Global linear convergence of Evolution Strategies with recombination on scaling-invariant functions}

\titlerunning{Global linear convergence of Evolution Strategies with recombination}

%%\subtitle{Using  the  LaTex Template}
%
\author{Cheikh Toure \and  Anne Auger \and Nikolaus Hansen} %\fnref{myfootnote}}

\institute{
Inria and CMAP, Ecole Polytechnique, IP Paris, France \\
cheikh.toure@polytechnique.edu\\firstname.lastname@inria.fr}

%\fntext[myfootnote]{Inria and CMAP, Ecole Polytechnique, IP Paris, France\\
%cheikh.toure@polytechnique.edu, firstname.lastname@inria.fr}

%\institute{
%Inria and CMAP, Ecole Polytechnique \\
%France\\
%firstname.lastname@inria.fr\\cheikh.toure@polytechnique.edu}
%

\date{Received: date / Accepted: date}

%%% Group authors per affiliation:

\maketitle

\begin{abstract}
Evolution Strategies (ESs) are stochastic derivative-free optimization algorithms whose most prominent representative, the CMA-ES algorithm, is widely used to solve difficult numerical optimization problems. We provide the first rigorous investigation of the linear convergence of step-size adaptive ESs involving a population and recombination, two ingredients crucially important in practice to be robust to local irregularities or multimodality.

We investigate the convergence of step-size adaptive ESs with weighted recombination on composites of strictly increasing functions with continuously differentiable scaling-invariant functions with a global optimum. 
This function class includes functions with non-convex sublevel sets and discontinuous functions. We prove the existence of a constant $r$ such that the logarithm of the distance to the optimum divided by the number of iterations converges to $r$. The constant is given as an expectation with respect to the stationary distribution of a Markov chain---its sign allows to infer linear convergence or divergence of the ES and is found numerically.

Our main condition for convergence is the increase of the expected log step-size on linear functions. In contrast to previous results, our condition is equivalent to the almost sure geometric divergence of the step-size on linear functions.

%Our methodology relies on investigating the stability of a Markov chain associated to the algorithm. Our stability study is crucially based on recent developments connecting the stability of deterministic control models to the stability of associated Markov chains.


\end{abstract}
%\MSC[2010] 00-01\sep  99-00
%\end{keyword}

%\end{frontmatter}

%\linenumbers : to put for option ``review'' in the environment

%\tableofcontents
%\newpage

\section{Introduction}

Evolution Strategies (ES\nnew{s}) are stochastic numerical optimization algorithms introduced in the 70's \cite{schwefel1977numerische, rech1973a, rechenberg1994evolutionsstrategie, schw1995a}. They aim at optimizing an objective function $\f: \R^{\n} \to \R$ in a so-called zero-order black-box scenario where gradients are not available and only \emph{comparisons} between $f$-values of candidate solutions are used to update the state of the algorithm. ESs sample candidate solutions from a multivariate normal distribution parametrized by a mean vector and a covariance matrix. The mean vector represents the incumbent or current favorite solution while the covariance matrix determines the geometric shape of the sampling probability distribution. In adaptive ES\nnew{s}, not only the mean vector but also \nnew{a step-size or} the covariance matrix is adapted in each iteration. Covariance matrices can be seen as encoding a metric such that Evolution Strategies that adapt a full covariance matrix are variable metric algorithms~\cite{suttorp2009efficient}.

In \new{the domain of Evolutionary Computation}, the covariance-matrix-adaptation ES (CMA-ES)~\cite{hansen2001completely, hansen2016cma} is nowadays recognized as  state-of-the-art to solve difficult numerical optimization problems that can typically be non-convex, non-linear, ill-conditioned, non-separable, rugged or multi-modal\footnote{\new{%
The \href{https://pypi.org/project/cmaes}{\texttt{cmaes}} and the
\href{https://pypi.org/project/cma}{\texttt{pycma}} Python modules that implement the algorithm are downloaded more than \href{https://pepy.tech/project/cmaes}{300,000} and \href{https://pepy.tech/project/cma}{30,000} times per week, respectively, from \href{https://pypi.org}{PyPI} as of \nnew{September}\del{March} 2022.}
\new{Both modules implement the main ideas of CMA-ES \cite{hansen2001completely} and
% different enhancement and options
further enhancements published over the years, notably the rank-$\mu$ update \cite{hansen2003reducing}, a better setting for step-size damping and the weights \cite{hansen2004evaluating}, an active covariance matrix update \cite{jastrebski2006improving},
% box-constrained handling \cite{}, large-scale option \cite{},
and restart mechanisms with increasing population size \cite{auger2005restart, hansen2009benchmarking}.
% mirroring option for small population size \cite{}.
}
}\cite{garcia2017since,hansen2010comparing,bouter2021achieving,glasmachers2022convergence}\cite[Fig.\,20]{rios2013derivative}.
Other relevant algorithms to solve ill-structured, non-convex, multi-modal, non-differentiable problems are also \nnew{often} population based like Estimation of Distribution algorithms notably AMaLGaM \cite{bosman2013benchmarking}, Differential Evolution \cite{feoktistov2006differential,storn1997differential}, and Particle Swarm Optimization (PSO) \cite{kennedy1995particle}. PSO methods however exploit separability and are inefficient to solve non-separable ill-conditioned problems \cite{hansen2011impacts}.
\new{The CMA-ES algorithm is based upon several maximum likelihood updates \cite{hansen2014principled}, can be interpreted as a natural gradient descent \cite{akimoto2010bidirectional, glasmachers2010exponential, ollivier2017information} and has been tightly linked to the EM-algorithm \cite{akimoto2012theoretical}.}
Adaptation of the full covariance matrix is crucial to solve \new{general} ill-conditioned, non-separable problems. Up to a multiplicative factor that converges to zero, the covariance matrix \new{in CMA-ES} becomes on strictly convex quadratic objective functions close to the inverse Hessian of the function~\cite{hansen2006cma}.

The CMA-ES algorithm follows a \muslashmu-ES algorithmic scheme where from the offspring population of $\lambda$ candidate solutions sampled at each iteration, the $\mu \approx\lambda/2$ best solutions---the new parent population---are recombined as a weighted sum to define the new mean vector of the multivariate normal distribution.
On a unimodal spherical function, the optimal step-size, i.e.\  the standard deviation that should be used to sample each coordinate of the candidate solutions, depends monotonously on $\mu$~\cite{rechenberg1994evolutionsstrategie}.
Hence, increasing the population size makes the search less local while preserving a close-to-optimal convergence rate per function evaluation as long as $\lambda$ remains moderately large~\cite{arnold2005optimal, arnold2006weighted, hansen2015evolution}. This remarkable theoretical property implies robustness and partly explains why on many multi-modal test functions increasing $\lambda$ empirically increases the probability to converge to the global optimum~\cite{hansen2004evaluating}.
The robustness when increasing $\lambda$ and the inherent parallel nature of ESs are two key features behind their success for tackling difficult black-box optimization problems.

Convergence is a central question in optimization. For comparison-based algorithms like ES\nnew{s}, linear convergence (where the distance to the optimum decreases geometrically) is the fastest possible convergence \cite{teytaud2006general, jamieson2012query}. 
Gradient methods also converge linearly on strongly convex functions~\cite[Theorem 2.1.15]{nesterov2003introductory}.
We have ample empirical evidence that adaptive ESs converge linearly on wide classes of functions~\cite{ros2008simple, hansen2011impacts, hansen2015evolution, igel2006computational}. Yet,  establishing proofs is known to be difficult.
So far, linear convergence could be proven only for step-size adaptive algorithms where the covariance matrix equals a scalar times the identity~\cite{auger2005convergence, auger2013linear, jagerskupper2003analysis, jagerskupper2007algorithmic, jagerskupper2005rigorous, jagerskupper20061+} or a scalar times a covariance matrix with eigenvalues upper bounded and bounded away from zero~\cite{akimoto2020}. In addition, these proofs require the parent population size to be one.

In this context, we analyze here for the first time the linear convergence of a step-size adaptive ES with a parent population size greater than one and recombination, following a \muslashmu-ES framework.
As a second novelty, we model the step-size update by a generic function and thereby also encompass the step-size updates in the CMA-ES algorithm~\cite{hansen2016cma} (however with a specific parameter setting which leads to a reduced state-space) and in the xNES algorithm~\cite{glasmachers2010exponential}.

Our proofs hold on composites of strictly increasing functions with either continuously differentiable scaling-invariant functions with a unique argmin or nontrivial linear functions. This class of functions includes discontinuous functions, functions with infinite many critical points, and functions with non-convex sublevel sets. It does not include functions with more than one (local or global) optimum.

In this paper, we use a methodology\del{ formalized \del{in} \cite{auger2016linear}}\del{ and previously used in~\cite{auger2005convergence,auger2013linear}} based on analyzing the stochastic process defined as the difference between the mean vector and a reference point (often the optimum of the problem), normalized by the step-size \cite{auger2016linear}.
This construct is a viable model of the underlying (translation and scale-invariant) algorithm when optimizing scaling-invariant functions, in which case the stochastic process is also a Markov chain and here referred to as \emph{\normalized\ Markov chain}. This chain is \emph{homogeneous} as a consequence of three crucial invariance properties of the ES algorithms:
translation invariance, scale invariance, and invariance to strictly increasing transformations of the objective function.
Proving \emph{stability} of the \normalized\ Markov chain ($\varphi$-irreducibility, Harris recurrence, positivity) is key to obtain almost sure \emph{linear behavior} of the algorithm~\cite{auger2016linear}. The sign and value of the convergence or divergence rate can however only be obtained from elementary Monte Carlo simulations.
\new{The technically challenging part in the proof methodology\del{we apply} is the stability analysis.}
\new{It was not carried out by Auger and Hansen \cite{auger2016linear} who presented the methodology and some algorithm classes that can be \nnew{addressed}\del{handled} by the methodology but assumed\del{ the technical} stability \nnew{of the algorithms} without proof.}
\new{We prove in the following the stability for some algorithms belonging to the \muslashmu-ES framework and thus formally prove linear behavior of these algorithms.
}

\paragraph*{Relation to previous works:}~ 
\new{In contrast to our study, most \del{Most}theoretical analyses of linear convergence concern the so-called (1+1)-ES where a single candidate solution is sampled ($\lambda = 1$) and the new mean is the best among the current mean and the sampled solution and in addition the one-fifth success rule is used to adapt the step-size~\cite{rech1973a, kern2004learning}. J\"agersk\"upper established lower-bounds and upper-bounds \new{on the hitting time to reduce the distance to the optimum} related to linear convergence on spherical functions~\cite{jagerskupper2003analysis, jagerskupper2007algorithmic} and on some convex-quadratic functions \cite{jagerskupper2005rigorous, jagerskupper20061+}. \new{Remarkably, these studies derive the dependency of the hitting time bounds on dimension and condition number, an aspect which is not covered with our approach.}
The underlying methodology used for the proofs \del{is hidden within the algorithm analysis and} was later unveiled as connected to drift analysis where an overall Lyapunov function of the state of the algorithm (mean and step-size) is used to prove upper and lower bounds on the hitting time of an epsilon neighborhood of the optimum~\cite{akimoto2018drift}. 
%This Lyapunov function is shown to satisfy drift conditions from which upper and lower bounds of the hitting time of an epsilon neighborhood of the optimum can be derived. \todo{revise}
With this drift analysis, Akimoto et al.~\cite{akimoto2018drift} provide lower and upper bounds on  the hitting time of an $\epsilon$-ball pertaining to linear convergence (coming as well with dependency in the dimension) for the 
%provide a simple analysis of the hitting time pertaining to linear convergence of 
the (1+1)-ES with one-fifth success rule on spherical functions. The analysis was later generalized for classes of functions including strongly convex functions with Lipschitz gradient as well as positively homogeneous functions~\cite{morinaga2019generalized, akimoto2020}.
% \new{In contrast to our analysis, the studies on the sphere function \cite{} provide a dependency of the convergence rate on the dimension of the problem and in the dimension and condition number for a certain class of convex-quadratic functions \cite{}. }

\new{Using the same methodology as in this paper}, the linear convergence of the $\pare{1+1}$-ES with step-size adapted via the one-fifth success rule is proven on increasing transformations of $C^{1}$ positively homogeneous functions $p$ with a unique global argmin and upper bounds on the degree of $p$ and on the norm of the gradient $\|\nabla p \|$ \cite{auger2013linear}. \del{The methodology in~\cite{auger2013linear} is similar to ours, as it consists in applying an LLN to ergodic Markov chains. }
 
 \new{While most theoretical studies of linear convergence concern a (1+1)-ES\del{ algorithm framework}, the $(1,\lambda)$-ES with self-adaptation has been analyzed on the sphere function \cite{auger2005convergence} and more recently an ODE method has been developed and \new{applied to}\del{exemplified for} a $(\mu \slash \mu,\lambda)$-ES \new{with a specific step-size adaptation} concluding linear convergence on the sphere function when the learning rate is small enough \cite{akimoto2022ode}}. \new{Our analysis holds for wider classes of functions and does not impose a small learning rate. However it does not allow to obtain the sign of the convergence rate.}
 
A few studies attempt to analyze ES\nnew{s} with\del{ a} covariance matrix adaptation:
% Diouane et al.~\cite{diouane2015globally} prove the convergence (but not linear convergence) of 
\new{A variant of CMA-ES\del{ where the algorithm is}, modified to ensure a sufficient decrease, globally converges (but not provably linearly)
\cite{diouane2015globally}.}
%and the convergence proof relies on this  modification;
% an abstract covariance adaptation is included in the 
% linear convergence analysis presented by Akimoto et al.\ \cite{akimoto2020}
Provided the eigenvalues \new{of the covariance matrix} stay upper bounded and bounded away from zero (which is not the case in the affine-invarian\new{t} CMA-ES),
%(hence the affine-invariant update of the original algorithm is not included).,
a (1+1)-CMA-ES with any covariance matrix update and proper step-size adaptation converges linearly \cite{akimoto2020}.
When convergence occurs on a twice continuously differentiable function for CMA-ES without step-size adaptation, the limit point is a local (or global) optimum~\cite{akimoto2010theoretical}.} \\

This paper is organized as follows.
We present in Section~\ref{cb-saes-framework} the algorithm framework, the assumptions on the algorithm and the class of objective functions where the convergence analysis is carried out. In Section~\ref{main-results} we present the main \new{proof idea to prove a linear behavior and present the ensuing proof structure}.
In Section~\ref{stability-markov-process}, we present \new{different} Markov chain notions and tools needed for our analysis. In Section~\ref{normalized-chain-section}, we establish  different stability properties on the \normalized\ Markov chain. We \new{state and prove} the main results in Section~\ref{consequences-section}.\del{In Section~\ref{previous-works},\done{fix reference} we highlight previous works related to this paper.}
\new{Notations are summarized in Table~\ref{notations}.}

\begin{table}\caption{Notations}\label{notations}\new{
\vspace{-1ex}\hspace{-0.5ex}%
\begin{tabular}{r@{ }p{0.88\textwidth}}
% Symbol & Meaning \\\hline
$\croc{\cdot}_i$ & is the $i^{\rm th}$ vector of a sequence of vectors\\
$\| \cdot \|$ & is the Euclidean norm \\
$\|\cdot\|_{\infty}$ & is the infinity norm on a space of bounded functions\\
% $\|\cdot\|_{h}$ & denotes for a positive function $h$ the norm on signed measures on $\mathcal{B}(\ZZ)$ defined for all signed measures $\nu$ as\\
$\|\nu\|_{h}$&${} = \sup_{|g| \leq h} | \mathbb{E}_{\nu}(g) |$
is for a positive function $h$ the norm of the signed measure $\nu$\\

$\pi_{1} \times \pi_{2}$ & is the product measure from two measure spaces
$\pare{\ZZ_{i}, \mathcal{B}(\ZZ_{i}), \pi_{i}}$,
$i=1,2$,
%$\pare{\ZZ_{2}, \mathcal{B}(\ZZ_{2}), \pi_{2}}$
on the product measurable space $\pare{\ZZ_{1}\times \ZZ_{2}, \, \mathcal{B}(\ZZ_{1}) \otimes \mathcal{B}(\ZZ_{2} ) }$ where $\otimes$ is the tensor product\\

$A^{c}$ & is the complement of a set $A$ \\
$A^{\top}$ & is the transpose of a matrix $A$\\
$\mathbf{B}\pare{x, \rho}$ & ${}=\acco{ y\in \R^{\n}; \|x-y\| < \rho}$
is the open ball around
$x\in\R^{\n}$ with radius $\rho > 0$ and $\overline{\mathbf{B}\pare{x, \rho}}$ is its closure\\
$\mathcal{B}(\ZZ)$ & is the Borel sigma-field of the topological space $\ZZ$\\

$\mathbb{E}_{\nu}(g)$ & ${} = \int g(z) \,\nu(\diff z)$ for any real-valued function $g$ and a signed measure $\nu$\\

$\mathcal{L}_{\f,z}$ & ${}=\acco{y\in\R^{\n} \, ; \f(y) = \f(z) }$ is the level set
for an objective function $\f: \R^{\n} \to \R$ and an element $z\in \R^{\n}$\\

$\mathbb{N}$ & is the set of non-negative integers\\
$\N$ & is the standard normal distribution\\
$\N_{m}$ & is the standard multivariate normal distribution in dimension $m$\\
$\N\pare{x, C}$ & is the multivariate normal distribution with mean $x \in\R^{m}$ and covariance matrix $C$\\

% nontrivial & means non-zero for linear functions\\

$p_{\N_{m}}$ & is the probability density function of $\N_{m}$\\
$\R_{+}$ & is the set of non-negative real numbers\\

$u $&${}= (u^{1},\dots,u^m) \in\R^{pm}$
where $u^{i} \in \R^{p}$ for $i = 1, \dots, m$ 
and $p \in  \mathbb{N}\setminus\acco{0}$
and we write $u = (u^{1}) = u^{1}$ if $m = 1$\\

$w^\top u$&$={} \sum_{i=1}^{m} w_{i} u^{i}$ for $w \in \R^{m}$ and $u \in \R^{pm}$

\end{tabular}
}
\end{table}

\del{%\vspace{5cm}
\paragraph*{Notation}
The set of non-negative real numbers (resp.\ non negative integers) is denoted $\R_{+}$ (resp.\ $\mathbb{N}$), $\| \cdot \|$ denotes the Euclidean norm.
For $x\in\R^{\n}$ and $\rho > 0$, 
$\mathbf{B}\pare{x, \rho} = \acco{ y\in \R^{\n}; \|x-y\| < \rho}$ and $\overline{\mathbf{B}\pare{x, \rho}}$ is its closure. 
For a set $A$, $A^{c}$ denotes its complement. For a topological space $\ZZ$, we denote its Borel sigma-field by $\mathcal{B}(\ZZ)$.
For a signed measure $\nu$, we denote for any real-valued function $g$, $\mathbb{E}_{\nu}(g) = \dsp\int g(z) \,\nu(\diff z).$
For a positive function $h,$ we denote by $\|\cdot\|_{h}$ the norm on signed measures on $\mathcal{B}(\ZZ)$ defined for all signed measure $\nu$ as
$\|\nu\|_{h} = \sup_{|g| \leq h} | \mathbb{E}_{\nu}(g) |.$
If $\pare{\ZZ_{1}, \mathcal{B}(\ZZ_{1}), \pi_{1}}$ and $\pare{\ZZ_{2}, \mathcal{B}(\ZZ_{2}), \pi_{2}}$ are two measure spaces, $\pi_{1} \times \pi_{2}$ denotes the product measure on the product measurable space $\pare{\ZZ_{1}\times \ZZ_{2}, \, \mathcal{B}(\ZZ_{1}) \otimes \mathcal{B}(\ZZ_{2} ) }$ where $\otimes$ is the tensor product.
We denote $\N$ the standard normal distribution.
The multivariate normal distribution with mean $x \in\R^{m} $ and a $m \times m$ covariance matrix $C$  is denoted $\N\pare{x, C}$.
%If $x\in\R^{m}$ and $C$ is a $m \times m$ covariance matrix, we denote by $\N\pare{x, C}$ the multivariate normal distribution with mean $x$ and covariance matrix $C$. 
If $C$ is the identity matrix, $\N_{m}$ denotes the standard multivariate normal distribution in dimension $m$ and $p_{\N_{m}}$ is its probability density function. 
%We denote by $p_{\N_{m}}$ its probability density function.
%
%We denote $\|\cdot\|_{\infty}$ 
The infinity norm on a space of bounded functions is $\|\cdot\|_{\infty}$.
For a matrix $T$, $T^{\top}$ is its transpose. 
For $p \in  \mathbb{N}\setminus\acco{0}$, we denote an element $u$ of $\R^{pm}$ as $u = (u^{1},\dots,u^m)$ where $u^{i} \in \R^{p}$ for $i = 1, \dots, m$. If $m = 1$, we write that $u = (u^{1}) = u^{1}$.
For $w \in \R^{m}$ and $u \in \R^{pm}$, we denote $\sum_{i=1}^{m} w_{i} u^{i}$ as $w^\top u$.
For an objective function $\f: \R^{\n} \to \R$ and an element $z\in \R^{\n}$, we denote by $\mathcal{L}_{\f,z}$ the level set $\acco{y\in\R^{\n} \, ; \f(y) = \f(z) }$.
We refer to a non-zero linear function as a nontrivial linear function.
}

\section{Algorithm framework and class of functions studied}
\label{cb-saes-framework}

We present in this section the step-size adaptive algorithm framework analyzed, the assumptions on the algorithm and the function class considered as well as preliminary results.
In the following, we consider an abstract measurable space $\pare{\Omega, \mathcal{F}}$ and a probability measure $P$ so that $\pare{\Omega, \mathcal{F}, P}$ is a measure space.

\subsection{The \muslashmu-ES algorithm framework}
\label{algorithm-framework}

We introduce 
%the algorithm framework studied in this work, specifically, 
step-size adaptive ESs with recombination, referred to as step-size adaptive $(\mu/\mu_{w},\lambda)$-ES. Given a positive integer $n$ and a function $\f: \R^{\n} \to \R$ to be minimized, the sequence of states of the algorithm is represented by $\{(X_k, \sigma_k) \,; k \in \mathbb{N} \}$ where at iteration $k,$ $X_{k}\in \R^n$ is the incumbent (the favorite solution considered as current estimate of the optimum) and the positive scalar $\sigma_{k}$ is the step-size. 
%The incumbent is also considered as a current estimate of the optimum. 
We fix positive integers $\lambda$ and $\mu$ such that $\mu \leq \lambda$.

Let $\left(X_{0}, \sigma_{0} \right) \in \R^{\n} \times \pare{0, \infty}$ and $U=\{ U_{k+1} = (U_{k+1}^1,\ldots,U_{k+1}^\lambda)\, ; k \in \mathbb{N} \}$ be a sequence of independent and identically distributed (i.i.d.) random inputs independent from $\left(X_{0}, \sigma_{0} \right)$, where for all $k \in \mathbb{N}$,  $U_{k+1}= (U_{k+1}^1,\ldots,U_{k+1}^\lambda)$ is  composed of $\lambda$ independent random vectors following a standard multivariate normal distribution $\N_{n}$. Given $(X_k,\sigma_k)$ for $k \in \mathbb{N}$, we consider the following iterative update.
First, we define $\lambda$ candidate solutions as
\begin{equation}\label{eq:offspring}
X_{k+1}^i = X_k + \sigma_k \; U_{k+1}^i \quad \text{for $i=1,\dots, \lambda$}.
\end{equation}
Second, we evaluate the candidate solutions on $f$.
We then denote an $f$-sorted permutation of $\pare{X_{k+1}^{1}, \dots, X_{k+1}^{\lambda} }$ as $\pare{X_{k+1}^{1:\lambda}, \dots, X_{k+1}^{\lambda:\lambda} }$ such that
\begin{align}
\f(X_{k+1}^{1:\lambda}) \leq \dots \leq \f(X_{k+1}^{\lambda:\lambda})
\label{ranking}
\end{align}
and thereby define the indices $i\!:\!\lambda$.
To break possible ties, we require that $i\!:\!\lambda < j\!:\!\lambda$ if $\f(X_{k+1}^{i}) = f(X_{k+1}^{j})$ and $i < j$. The sorting indices $i\!:\!\lambda$ are also used for the $\sigma$-normalized difference vectors $U_{k+1}^i$ in that
%
%\begin{align*}
$
U_{k+1}^{i:\lambda} = \frac{X_{k+1}^{i:\lambda} - X_k}{\sigma_k} %
$.
%\end{align*}
%
Accordingly, we define the \emph{selection function} $\alpha_{\f}$ of $z \in \R^{\n}$ and $u = (u^{1},\dots,u^{\lambda}) \in \R^{\n \lambda}$ to yield the sorted sequence of the difference vectors as
\begin{align}
\alpha_{f}(z, u) = (u^{1:\lambda}, \dots, u^{\mu: \lambda}) \in  \R^{\n \mu}  ,
\label{alphaDefinition}
\end{align}
with $\f(z + u^{1:\lambda}) \leq \dots \leq \f(z + u^{\lambda:\lambda})$ and the above tie breaking.
For $\lambda = 2$ and $\mu = 1$, the selection function has the simple expression $\alpha_{f}(z, (u^{1}, u^{2})) = (u^{1}-u^{2}) \mathds{1}_{\acco{ \f(z + u^1) \leq \f(z + u^2)} } + u^{2}.$
By definition,  for $k \in \mathbb{N}$, $\alpha_{f}(X_{k} , \sigma_{k} U_{k+1}) =   \pare{ \sigma_{k} U_{k+1}^{1:\lambda}, \dots, \sigma_{k} U_{k+1}^{\mu:\lambda} }$
so that 
\begin{align}
\frac{\alpha_{f}(X_{k} , \sigma_{k} U_{k+1})} {\sigma_{k}} = \pare{ U_{k+1}^{1:\lambda}, \dots, U_{k+1}^{\mu:\lambda} }.
\label{selection-division}
\end{align}
However, $\alpha_{f}$ is not a homogeneous function in general, because the indices $i\!:\!\lambda$ in \eqref{selection-division} depend on $f$ and hence on $\alpha_f$ and hence on $\sigma_k$.

The update of the state of the algorithm uses the objective function only through the above selection function which is invariant to strictly increasing transformations of the objective function. Indeed, the selection is determined through the ranking of candidate solutions in~\eqref{ranking} which is the same when on $g \circ f$ or $f$ given that $g$ is strictly increasing. We talk about comparison-based algorithms. Formally:
\begin{lemma}
%Let $g$ be a  function. Define 
Let $\f = \varphi\circ g$ where $g:\R^n \to \R$ and $\varphi$ is strictly increasing. Then $\alpha_{f} = \alpha_{g}$.
\label{selection-function-increasing-transformation}
\end{lemma}

To update the mean vector $X_k$, we consider a weighted average of the $\mu \leq \lambda$ best solutions $\sum_{i=1}^{\mu} w_{i} X_{k+1}^{i:\lambda} $ where $w=\pare{w_{1},\dots,w_{\mu}}$ is a non-zero vector. \del{With}\nnew{When} only positive weights summing to one \nnew{are used}, this weighted average is situated in the convex hull of the $\mu$ best points.
The next incumbent $X_{k+1}$ is constructed by combining $X_{k}$ and $\sum_{i=1}^{\mu} w_{i} X_{k+1}^{i:\lambda}$  
\begin{align}
X_{k+1} & = \pare{1 - \sum_{i=1}^{\mu}w_{i} } X_{k} + \sum_{i=1}^{\mu} w_{i} X_{k+1}^{i:\lambda}
%& = X_{k} + \sum_{i=1}^{\mu} w_{i} \pare{X_{k+1}^{i:\lambda} - X_{k}} 
=  X_{k} + \sigma_{k}\sum_{i=1}^{\mu} w_{i} U_{k+1}^{i:\lambda} \enspace .
\label{incumbent}
\end{align}
Positive weights with small indices move the new mean vector towards the better solutions, hence these weights should generally be large.
In ESs, the weights are always non-increasing in $i$.
With the notable exception of Natural Evolution Strategies (\hspace{1sp}\cite{glasmachers2010exponential} and related works), all weights are positive.
In practice, $\sum_{i=1}^{\mu}w_{i}$ is often set to $1$ such that the new mean vector is the weighted average of the $\mu$ best solutions. Proposition~\ref{prop:weight-condition} describes (generally weak) explicit conditions for the weights under which our results hold.
We write the step-size update in an abstract manner as
\begin{equation}
\sigma_{k+1} = \sigma_{k} \, \Gx\pare{ U_{k+1}^{1:\lambda}, \dots, U_{k+1}^{\mu:\lambda} }   \label{step-size}
\end{equation}
where $\Gx: \foncfleche{ \R^{\n \mu} }{\R_{+}\backslash\acco{0}}$ is a measurable function. This generic step-size update is by construction scale-invariant, which is key for our analysis \cite[Proposition~2.9]{auger2016linear}.
 The update of the mean vector and of the step-size are both functions of the $f$-sorted sampled vectors $(U_{k+1}^{1:\lambda}, \dots, U_{k+1}^{\mu:\lambda})$.

Using~\eqref{selection-division}, we rewrite the algorithm framework \eqref{incumbent} and~\eqref{step-size} for all $k$ as:
\begin{align}
X_{k+1} &= X_{k} + \sum_{i=1}^{\mu} w_{i} \croc{\alpha_{f}(X_{k} , \sigma_{k} U_{k+1})}_{i} = X_{k} +  w^{\top} \alpha_{f}(X_{k} , \sigma_{k} U_{k+1})   \label{x-with-alpha}
\\
\sigma_{k+1} &= \sigma_{k} \, \Gx \left( \frac{\alpha_{f}(X_{k} , \sigma_{k} U_{k+1})} {\sigma_{k}} \right) \label{sigma-with-alpha}
\end{align}%
with $U=\{U_{k+1}\, ; k \in \mathbb{N} \}$ the sequence of identically distributed random inputs and $w \in \R^{\mu}\setminus\acco{0}$. \nnew{In \eqref{x-with-alpha}, we use the notation $[~]_i$ to denote the $i^{\rm th}$ vector\del{of a sequence\del{the vector}} of the $\mu$ vectors \nnew{composing} $\alpha_{f}(X_{k} , \sigma_{k} U_{k+1})$.}

\subsection{Algorithms encompassed\label{sec:alg-encompassed}}

The generic update in \eqref{step-size} or equivalently \eqref{sigma-with-alpha} encompasses the step-size update of the cumulative step-size adaptation evolution strategy ($\pare{\mu/\mu_w, \lambda}$-CSA-ES)
\cite{auger2016linear,hansen2001completely}
with cumulation factor set to $1$
where for $d_\sigma > 0$, $w \in \R^{\mu} \setminus\acco{0}$ and $u = (u^{1},\dots,u^{\mu}) \in \R^{\n \mu}$,
\begin{equation}\label{eq:CSA1}
\Gx_{\rm CSA1}^0(u^{1},\dots,u^{\mu}) = \exp\left(\frac{1}{d_{\sigma}} \left(  \frac{ \| \sum_{i=1}^{\mu} w_{i} u^{i} \|    }  {\|w\| \, \mathbb{E}\croc{ \|  \N_{\n}  \|   }   } - 1 \right)   \right)
\enspace.
\end{equation}
The acronym CSA1 emphasizes that we only consider a particular case here:
in the original CSA algorithm, 
the sum in \eqref{eq:CSA1} is an exponentially fading average of these sums from the past iterations with a smoothing factor of $1 - c_{\sigma}$.
Equation~\eqref{eq:CSA1} only holds when
the cumulation factor $c_{\sigma}$ is equal to $1$,
whereas in practice, $1/c_{\sigma}$ is between $\sqrt\n/2$ and $\n+2$ (see~\cite{hansen2016cma} for more details). The damping parameter $d_{\sigma} \approx 1$ scales the change magnitude of $\log(\sigma_{k})$.

Equation~\eqref{eq:CSA1} increases the step-size if and only if the length of $\sum _{i=1}^{\mu} w_{i} U_{k+1}^{i:\lambda}$ is larger than the expected length of $\sum _{i=1}^{\mu} w_{i} U_{k+1}^{i}$ \nnew{under random selection} which is equal to $\| w \| \,\mathbb{E}\croc{\| \N_{\n} \|} $.
Since the function $\Gx_{\rm CSA1}^0$ is not continuously differentiable (an assumption needed in our analysis) we consider a version of the $\pare{\mu/\mu_w, \lambda}$-CSA1-ES~\cite{arnold2002random} that compares the square length of $\sum _{i=1}^{\mu} w_{i} U_{k+1}^{i:\lambda}$ to the expected square length of $\sum _{i=1}^{\mu} w_{i} U_{k+1}^{i}$ which is $n \| w\|^{2}$. Hence, \nnew{the step-size update we consider and that satisfies our assumptions is defined}\del{we analyze} for $d_\sigma > 0,$ $w \in \R^{\mu} \setminus\acco{0}$ and $u = (u^{1},\dots,u^{\mu}) \in \R^{\n \mu}$\del{:} \nnew{as}
 \begin{align}
 \Gx_{\rm CSA1}(u^{1},\dots,u^{\mu}) =  \exp\left( \frac{1}{2 d_{\sigma}\n }\left(  \frac{ \| \sum_{i=1}^{\mu} w_{i} u^{i} \|^{2} }{ \| w \|^{2} } - \n \right)   \right).
 \label{step-size-csa}
\end{align}

Another step-size update encompassed with~\eqref{selection-division} is given by the Exponential Natural Evolution Strategy (xNES)~\cite{glasmachers2010exponential, schaul2012natural, auger2016linear, ollivier2017information} and defined for $d_\sigma > 0,$ $w \in \R^{\mu} \setminus\acco{0}$ and $u = (u^{1},\dots,u^{\mu}) \in \R^{\n \mu}$ as
 \begin{align}
 \Gx_{\rm xNES}(u^{1},\dots,u^{\mu}) = \exp\left(  \frac{1}{2 d_{\sigma} n}\left(  \sum_{i=1}^{\mu}
\frac{w_{i}}{\sum_{j=1}^\mu |w_j|} \left( \| u^{i} \|^{2} - n \right)   \right) \right).
 \label{step-size-xnes}
 \end{align}
 Both equations \eqref{step-size-csa} and \eqref{step-size-xnes} correlate the step-size increment with the vector lengths of the $\mu$ best solutions.
While \eqref{step-size-csa} takes the squared norm of the weighted sum of the vectors, \eqref{step-size-xnes} takes the weighted sum of squared norms.
Hence, correlations between the directions $u^i$ affect only \eqref{step-size-csa}.
Both equations are offset to become unbiased such that $\log\circ\,\Gamma$ is zero in expectation when $u^i\sim \N_{\n}$ for all $1\le i\le\lambda$, are i.i.d.\ random vectors.

\subsection{Assumptions on the algorithm framework}

\newcommand{\algframe}{\eqref{x-with-alpha} and \eqref{sigma-with-alpha}}
We pose some assumptions on the algorithm \algframe\ starting with assumptions on the step-size update function $\Gx$.
\begin{itemize}
\item[A1.] The function $\Gx: \foncfleche{ \R^{\n \mu} }{\R_{+}\backslash\acco{0}}$ is continuously differentiable ($C^{1}$).

\item[A2.]  $\Gx$ is \emph{invariant under rotation} in the following sense:
for all $n \times n$ orthogonal matrices $T$,  for all $u = \pare{u_{1},\dots,u_{\mu}} \in \R^{\n\mu}$, 
$ \Gamma(Tu_1,\ldots,Tu_\mu)=\Gamma(u)$.

\item[A3.]   The function $\Gamma$ is lower-bounded by a constant $m_{\Gx} > 0$, that is for all $x \in \R^{\n\mu}$, $\Gx(x) \geq m_\Gx$.
\item[A4.]  $\log\circ \,\Gx$ is $\mathcal{N}_{\n\mu}$-integrable, that is, $\dsp\int \abs{\log(\Gx(u))} p_{\mathcal{N}_{\n\mu}} (u) \diff u  < \infty$.
\end{itemize}

We can easily verify that Assumptions A1--A4 are satisfied for the $\pare{\mu/\mu_w, \lambda}$-CSA1 and $\pare{\mu/\mu_w, \lambda}$-xNES updates given in \eqref{step-size-csa} and \eqref{step-size-xnes}. More precisely, the following lemma holds.

\begin{lemma}
The step-size update function $\Gx_{\rm CSA1}$ defined in \eqref{step-size-csa} satisfies Assumptions A1$-$A4. Endowed with non-negative weights $w_{i} \geq 0$ for all $i=1,\dots, \mu$, 
the step-size update function $\Gx_{\rm xNES}$ defined in~\eqref{step-size-xnes} satisfies Assumptions A1$-$A4.
\end{lemma}
\begin{proof}
A1 and A4 are immediate to verify.
For A2, the invariance under rotation comes from the norm-preserving property of orthogonal matrices. For all $u = \pare{u^{1},\dots,u^{\mu}} \in \R^{\n\mu},$ $\Gx_{\rm CSA1}(u) \geq \exp\left( -\frac{1}{2d_{\sigma}} \right)$ such that $\Gx_{\rm CSA1}$ satisfies A3. Similarly $\Gx_{\rm xNES}(u)
= \exp\left( -\frac{1}{2 d_{\sigma}}\frac{\sum_{i=1}^{\mu}w_{i}}{\sum_{j=1}^{\mu}|w_{j}| } + \frac{1 }{2 d_{\sigma} n} \sum_{i=1}^{\mu} \frac{w_{i}}{\sum_{j=1}^{\mu}|w_{j}| } \| u^{i} \|^{2}   \right)$. Since all the weights are non-negative, $\frac{1}{2 d_{\sigma} n} \sum_{i=1}^{\mu} w_{i} \| u^{i} \|^{2} \geq 0$. And then $-\frac{1}{2 d_{\sigma} }\sum_{i=1}^{\mu}w_{i} + \frac{1 }{2 d_{\sigma} n} \sum_{i=1}^{\mu} w_{i} \| u^{i} \|^{2}  \geq -\frac{ 1}{2 d_{\sigma}}\sum_{i=1}^{\mu}w_{i}$. Therefore $\Gx_{\rm xNES}(u) \geq \exp\left( -\frac{ 1}{2 d_{\sigma}} \right)$ which does not depend on $u$, such that $\Gx_{\rm xNES}$ satisfies A3.  
\end{proof}

Assumptions A1--A4 are also satisfied for a constant function $\Gx$ equal to a positive number. When the positive number is greater than $1$, our main condition for a linear behavior is satisfied, as we will see later on. Yet, the step-size of this algorithm clearly diverges geometrically.

We formalize now the assumption on the source distribution used to sample candidate solutions, as it was already specified when defining the algorithm framework.

\begin{itemize}
\item [A5.]  $U =\{ U_{k+1} = \pare{U_{k+1}^{1},\dots,U_{k+1}^{\lambda}}\in \R^{\n\lambda} \, ; k \in \mathbb{N} \}$, see e.g.~\eqref{eq:offspring}, is an i.i.d.\ sequence that is also independent from $\pare{X_{0}, \sigma_{0}}$, and for all natural integer $k,$ $U_{k+1}$ is an independent sample of $\lambda$ standard multivariate normal distributions on $\mathbb{R}^{\n}$ at time $k+1$.
\end{itemize}

The last assumption is natural as ESs use predominantly Gaussian distributions\footnote{In Evolution Strategies, Gaussian distributions are mainly used for convenience: they are the natural choice to generate rotationally invariant random vectors. Several attempts have been made to replace Gaussian distributions by Cauchy distributions \cite{kappler1996evolutionary,yao1997fast,schaul2011high}.
Yet, their implementations are typically not rotational invariant and steep performance gains are observed either in low dimensions or crucially based on the implicit exploitation of separability~\cite{hansen2006heavy}. }.
Yet, we can replace the multivariate normal distribution by a distribution with finite first and second moments and a probability density function of the form $x \mapsto \frac{1}{\sigma^{n} } g\pare{ \frac{ \|x\|^{2} }{\sigma^{2}} }$
where  $\sigma > 0$ and $g: \R_{+} \to \R_{+}$ is $C^{1}$, non-increasing and submultiplicative in that there exists $K > 0$ such that for $t \in \R_{+}$ and $s \in \R_{+}$, $g(t + s) \leq K g(t) g(s)$ (such that Proposition~\ref{log-integrable} holds).

\subsection{Assumptions on the objective function}
\label{subclass-scaling}
%We introduce in this section the assumptions needed on the objective function to prove the linear behavior of a step-size adaptive $\pare{\mu/\mu_{w}, \lambda}$-ES. 
Our main assumption on $f$ to analyze the linear behavior of a step-size adaptive $\pare{\mu/\mu_{w}, \lambda}$-ES is that it is scaling-invariant. We remind that $\f$ is scaling-invariant~\cite{auger2016linear} with respect to a reference point $x^{\star}$ if for all $\rho > 0$, $x, y \in \R^{n}$
\begin{align}
\f(x^{\star} + x) \leq \f(x^{\star} + y) \iff \f\pare{x^{\star}  + \rho \, x } \leq  \f\pare{x^{\star}  + \rho \, y } .\label{scaling-invariant}
\end{align}
More precisely, we pose one of the following assumptions on $\f$:
\begin{itemize}
\item[F1.] The function $\f$ satisfies $\f = \varphi \circ g $ where $\varphi$ is a strictly increasing function and g is a $C^{1}$ scaling-invariant function with respect to $x^{\star}$ and has a unique global argmin (that is $x^{\star}$).
\item[F2.] The function  $\f$ satisfies $\f = \varphi \circ g $ where $\varphi$ is a strictly increasing function and $g$ is a nontrivial linear function.
\end{itemize}

Assumption F1 is our core assumption for studying convergence: we assume scaling invariance and continuous differentiability not on $f$ but on $g$ where $\f=\varphi \circ g$ such that the function $f$ can be discontinuous (we can include jumps in the function via the function $\varphi$).
Because ESs are comparison-based algorithms and thus the selection function is identical on $f$ or $g \,\circ f$ (see Lemma~\ref{selection-function-increasing-transformation}), our analysis is invariant if we carry it out on $f$ or $g \circ f$.
Strictly increasing transformations of strictly convex quadratic functions satisfy F1. Functions with non-convex sublevel sets can satisfy F1 (see Figure~\ref{fig:scalingInvariance}). More generally, strictly increasing transformations of $C^{1}$ positively homogeneous functions with a unique global argmin satisfy F1.
 Recall that a function $p$ is positively homogeneous with degree $\alpha > 0$ and with respect to $x^{\star}$ if for all $x, y \in \R^{\n}$, for all $\rho > 0$, 
\begin{align}
p(\rho (x-x^{\star})) = \rho^{\alpha} p(x - x^{\star}) \enspace .
\label{def-positively-homogeneous}
\end{align}

\begin{figure*}~\\[-1.5ex]
\centering
\newcommand{\figwidth}{0.3}
        \includegraphics[width=\figwidth\textwidth]{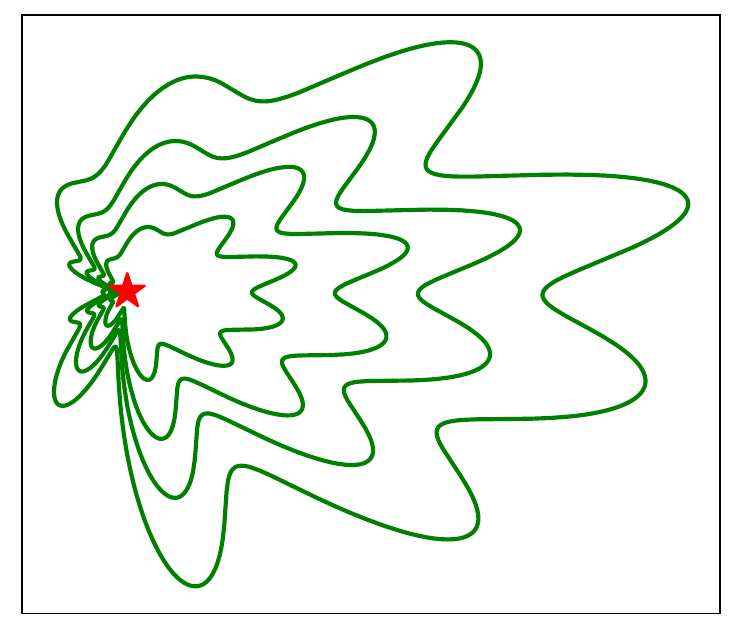}
        \vspace*{-1.5ex}
        \caption{\label{fig:scalingInvariance}
         Level sets of scaling-invariant functions with respect to the red star $x^{\star}$. A randomly generated scaling-invariant function from a ``smoothly'' randomly perturbed sphere function.}
\end{figure*}

\subsection{Preliminary results}

If $\f$ is scaling-invariant with respect to $x^{\star}$, the composite of the selection function $\alpha_{\f}$ with the translation $(z, u) \mapsto (x^{\star} + z, u)$ is positively homogeneous with degree $1$. If in addition $\f$ is a measurable function with Lebesgue negligible level sets,
then\del{~\cite[Proposition 5.2]{chotard2019verifiable} gives} the explicit expression of the probability density function of $\alpha_{f}(x^{\star} + z, U_{1})$ \new{is known \cite[Proposition 5.2]{chotard2019verifiable}} where $U_{1}$ follows the distribution of $\N_{n\lambda}$. These results are formalized in the next lemma.

\begin{lemma}
If $\f$ is a scaling-invariant function with respect to $x^{\star}$, then the function $(z, u) \mapsto \alpha_{f}(x^{\star} + z, u)$ is positively homogeneous with degree $1$. In other words, for all $z \in \R^{\n}$, $\sigma > 0$ and $u = \pare{u^{1},\dots,u^{\lambda}} \in \R^{\n\lambda}$,
$ \alpha_f\pare{x^{\star} + \sigma z, \sigma u } = \sigma \alpha_{f}\pare{x^{\star} + z , u }.$

If in addition $\f$ is a measurable function with Lebesgue negligible level sets and $U_{1} = \pare{ U_{1}^{1},\dots, U_{1}^{\lambda}}$ is distributed according to $\N_{n\lambda}$, then for all $z \in \R^{n}$, the probability density function $p_{z}^{\f}$ of $\alpha_{f}(x^{\star} + z, U_{1})$ exists and for all $u = (u^{1},\dots,u^{\mu}) \in \R^{\n \mu},$
\begin{align}
 p_{z}^{\f}(u) = \frac{\lambda !}{(\lambda - \mu)!} (1-Q_{z}^{\f}(u^{\mu}))^{\lambda-\mu} \prod_{i=1}^{\mu-1} \mathds{1}_{\f(x^{\star} + z+u^i) < \f(x^{\star} + z+u^{i+1})} \prod_{i=1}^{\mu} p_{\N_{\n}}(u^{i})
 \label{mupositive}
 \end{align}
where $Q_{z}^{\f}(w) = P\left( f\left(x^{\star} + z + \N_{\n}\right) \leq f\left(x^{\star} + z + w \right) \right).$
\label{lem:SI-alpha}
\end{lemma}
\begin{proof}
We have that $\f(x^{\star} + z + u^{1:\lambda}) \leq \dots \leq \f(x^{\star} + z +  u^{\lambda:\lambda})$ if and only if $\f(x^{\star} + \sigma (z + u^{1:\lambda})) \leq \dots \leq \f(x^{\star} + \sigma(z + u^{\lambda:\lambda}))$. Therefore $ \alpha_f\pare{x^{\star} + \sigma z, \sigma u } = \sigma \pare{u^{1:\lambda}, \dots, u^{\mu:\lambda}  } = \sigma \alpha_{f}\pare{x^{\star} + z , u }$. 
Equation~\eqref{mupositive} \new{holds}\del{is given by~\cite[Proposition 5.2]{chotard2019verifiable}} whenever $\f$ has Lebesgue negligible level sets \new{\cite[Proposition 5.2]{chotard2019verifiable}}.
\end{proof}

On a linear function $f$, the selection function $\alpha_{\f}$ defined in \eqref{alphaDefinition} is independent of the current state of the algorithm and is positively homogeneous with degree $1$. 
%This result is underlying previous results  \cite{auger2005convergence, chotard2012cumulative}. 
We provide  a simple formalism and proof of this result while it is already known and underlying previous works \cite{auger2005convergence, chotard2012cumulative}.
\begin{lemma}
If $\f$ is an increasing transformation of a linear function, then for all $x \in \R^{n}$ the function $\alpha_{\f}\pare{x, \cdot}$ does not depend on $x$ and is positively homogeneous with degree $1$. In other words, for $x \in \R^{\n}$, $\sigma > 0$ and $u = \pare{u^{1},\dots,u^{\lambda}} \in \R^{\n\lambda}$,
$ \alpha_{\f}\pare{x, \sigma u } = \sigma \alpha_{f}\pare{0 , u }.$
\label{alpha-linearity}
\end{lemma}
\begin{proof}
By linearity $\f(x + \sigma u^{1:\lambda}) \leq \dots \leq \f(x + \sigma u^{\lambda:\lambda})$ if and only if $ \f( u^{1:\lambda}) \leq \dots \leq \f( u^{\lambda:\lambda})$. Therefore $ \alpha_{\f}\pare{x, \sigma u } = \sigma \pare{u^{1:\lambda}, \dots, u^{\mu:\lambda}  } =\sigma \alpha_{\f}\pare{0, u }$.
\end{proof}

Let $l^{\star}$ be the linear function defined for all $x \in \R^{n}$ as $l^{\star}(x) = x_{1}$ and $U_{1} = \pare{ U_{1}^{1}, \dots, U_{1}^{\lambda} }$ where $U_{1}^{1}, \dots, U_{1}^{\lambda}$ are i.i.d.\ with law $\N_{\n}$. Define the step-size change $\Gl$ as
\begin{align}
\Gl = \Gx \pare{ \alpha_{l^{\star}}(0, U_{1}) } .
\label{step-size-notation}
\end{align}

We prove in the next proposition that for all nontrivial linear functions, the step-size multiplicative factor of the algorithm \algframe\ has at all iterations the distribution of $\Gl$. 
This result derives from the rotation invariance of the function $\Gx$ (see Assumption A2) and of the probability density function $p_{\N_{\n\mu}}: u \mapsto \frac{1}{ (2\pi)^{n\mu/2} } \exp\pare{-\|u \|^{2}/2}$. The details of the proof are in Appendix~\ref{proof-linear-invariance}. 
\begin{proposition}(Invariance of the step-size multiplicative factor on linear functions)
\label{linear-invariance}
Let $f$ be an increasing transformation of a nontrivial linear function, i.e.\ satisfy F2. Assume that \nnew{the sequence }$\{U_{k+1} \,; k \in \mathbb{N} \}$ satisfies Assumption A5 and that $\Gx$ satisfies Assumption A2, i.e.\ $\Gx$ is invariant under rotation.
Then for all $z \in \R^{\n}$ and all natural integer $k$, the step-size multiplicative factor $\Gx \pare{ \alpha_{f}(z, U_{k+1} )}$ has the law of the step-size change $\Gl$ defined in~\eqref{step-size-notation}.
\end{proposition}
The proposition shows that on any (nontrivial) linear function the step-size change factor is independent of $X_k$, $Z_k$ and even $\sigma_k$.
We can now state the result which is at the origin of the methodology used in this paper, namely that on scaling-invariant functions, $\{Z_k = (X_k - x^{\star})/\sigma_k \, ; k \in \mathbb{N} \}$ is a homogeneous Markov chain. 
%(We specify later on why the stability of this chain is key for the linear convergence of $\{(X_k,\sigma_k) \, ; k \in \mathbb{N} \}$.) 
For this, we introduce the following function
 \begin{align}
F_{w}  (z, v)= \frac{z + \sum_{i=1}^{\mu}w_{i} v_{i} }{ \Gx(v) } \,\,\mbox{for all} \,\, (z, v) \in \R^{\n} \times \R^{\n \mu},
\label{homogeneousFunction}
\end{align}
which allows to write $Z_{k+1}$ as a deterministic function of $Z_k$ and $U_{k+1}$.
The following proposition establishes conditions under which $\{Z_k ; k \in \mathbb{N}\}$ is a homogeneous Markov chain that is defined with~\eqref{homogeneousFunction}, independently of $\{ \left(X_k, \sigma_{k} \right) ; k \in \mathbb{N}\}$. We refer to $\{ Z_{k} ; k \in \mathbb{N} \}$ as the \normalized\ chain. This is a particular case \new{from}\del{of
\cite[Proposition 4.1]{auger2016linear} where} a more abstract algorithm framework \new{\cite[Proposition 4.1]{auger2016linear}}.\del{is assumed.}

\begin{proposition}
\label{prop:MC} 
Let $\f$ be a scaling invariant function with respect to $x^{\star}$. Define the sequence $\{(X_k,\sigma_k) ; k \in \mathbb{N} \}$ as in~\eqref{incumbent} and~\eqref{step-size}. Then $\{Z_k = (X_k - x^{\star})/\sigma_k \, ; k \in \mathbb{N} \}$ is a homogeneous Markov chain and for all natural integer $k$, the following equation holds
\begin{align}
Z_{k+1} = F_{w} \pare{Z_{k}, \alpha_{\f}\pare{x^{\star} + Z_{k}, \,U_{k+1}}},
\label{normalized-process}
\end{align}
where $\alpha_f$ is defined in~\eqref{alphaDefinition}, $F_{w}$ is defined in \eqref{homogeneousFunction}
and $\{U_{k+1}\, ; k \in \mathbb{N}\}$ is the sequence of random inputs used to sample the candidate solutions in \eqref{eq:offspring} corresponding to the random input in \eqref{x-with-alpha} and \eqref{sigma-with-alpha}.\label{markov-chain-normalized-chain}
\end{proposition}
\begin{proof}
The definition of the selection function $\alpha_f$ allows to write~\eqref{incumbent} and~\eqref{step-size} as~\eqref{x-with-alpha} and~\eqref{sigma-with-alpha}. 
We have $Z_{k+1} = \frac{X_{k+1} - x^{\star}}{\sigma_{k+1}} = 
\frac{X_{k} - x^{\star} + \sum_{i=1}^{\mu} w_{i} \croc{\alpha_{f}(X_{k} , \sigma_{k} U_{k+1})}_{i}}
{\sigma_{k} \, \Gx \pare{ \frac{\alpha_{f}(X_{k} , \sigma_{k} U_{k+1})}{ \sigma_{k} }  } } = \frac{Z_{k} + \sum_{i=1}^{\mu} w_{i} \frac{\croc{\alpha_{f}(X_{k} , \sigma_{k} U_{k+1})}_{i}}{\sigma_{k}} }
{\Gx \pare{ \frac{\alpha_{f}(X_{k} ,\, \sigma_{k} U_{k+1})}{ \sigma_{k} }  } }.$
By Lemma~\ref{lem:SI-alpha}, $\frac{\alpha_{f}(X_{k} , \,\sigma_{k} U_{k+1})}{ \sigma_{k} } = \frac{\alpha_{f}(x^{\star} + X_{k} - x^{\star}, \,\sigma_{k} U_{k+1})}{ \sigma_{k} }  = \alpha_{f}(x^{\star} + \frac{X_{k} - x^{\star} }{\sigma_{k}}, U_{k+1})$. Then
$Z_{k+1} = F_{w} \pare{Z_{k}, \alpha_{\f}\pare{x^{\star} + Z_{k}, U_{k+1}}} $ and $\{Z_k ; k \in \mathbb{N}\}$ is a homogeneous Markov chain.  
\end{proof}

Three invariances are key to obtain that $\{Z_k = (X_k - x^{\star})/\sigma_k \, ; k \in \mathbb{N} \}$ is a homogeneous Markov chain: invariance to strictly increasing transformations (stemming from the comparison-based property of ESs), translation invariance, and scale invariance~\cite[Proposition 4.1]{auger2016linear}. The last two invariances are satisfied with the update we assume for mean and step-size.

\section{\del{Main results}\new{Methodology and overview of the rest of the analysis}}
\label{main-results}

%\new{The convergence analysis of the algorithm presented in Section~\ref{cb-saes-framework} relies on establishing a number of results on the stability of Markov chains underlying the algorithm. In order to understand the link between the stability analysis and the convergence of the algorithm we present the main proof ideas 
%different steps of our anlaysis, we motivate in this section the rest of the mathematical study of the paper by presenting an overview of the main convergence results and their proof idea. We conclude the section by giving the organization of the rest of the paper that naturally stem from the proof methodology presented here.}
%\new{The convergence analysis of the algorithm presented in Section~\ref{cb-saes-framework} relies on the stability analysis of an underlying Markov chain. We present in this section the reason for this.
%To understand the link between the stability study and the convergence of the original algorithm, we present in this section the main proof}  
  
\new{We present in this section the main idea behind the proof methodology used in this paper, namely how the stability study of an underlying Markov chain leads to convergence (or divergence) of the original algorithm. From there, we sketch the different steps of the analysis and present an overview of the structure of the rest of the mathematical analysis.}

\del{We present our main results that express the global linear convergence of the algorithm presented in Section~\ref{cb-saes-framework}. L}\new{We aim at proving l}inear convergence \new{that} can be visualized by looking at the distance to the optimum: after an adaptation phase, we observe that the log distance to the optimum diverges to minus infinity with a graph that resembles a straight line with random perturbations. The step-size converges to zero at the same linear rate (see Figure~\ref{fig:convergence}),\del{. We call this constant}\done{new: which constant?} the \new{so-called} convergence rate of the algorithm. Formally, in case of convergence, there exists ${\rm r} >0$ such that
\begin{align}\label{eq:linear-behavior}
\lim_{k \to \infty} \frac1k \log \frac{ \| X_k - x^\star \| }{\| X_0 - x^\star \|} 
= \lim_{k \to \infty} \frac1k \log \frac{\sigma_k }{\sigma_0} = - {\rm r} 
\end{align}
where $x^\star$ is the optimum of the function. 

\new{
We consider a scaling invariant function with respect to $x^*$. From Proposition~\ref{prop:MC}, we know that $\{Z_k = (X_k - x^{\star})/\sigma_k \, ; k \in \mathbb{N} \}$ is a homogeneous Markov chain where $\{(X_k,\sigma_k) \, ; k \in \mathbb{N}  \}$ is the sequence of states of the step-size adaptive $\pare{\mu/\mu_{w}, \lambda}$-ES defined in~\eqref{incumbent} and~\eqref{sigma-with-alpha} (see Proposition~\ref{prop:MC}). 
We use this  Markov chain to 
%
%
% In order to prove \eqref{eq:linear-behavior}, we will use the Markov chain $\{Z_k = (X_k - x^{\star})/\sigma_k \, ; k \in \mathbb{N} \}$ and 
 write the log progress in the following way:
\begin{align}
\log \frac{\|X_{k+1} - x^{\star}  \|}{\|X_{k} - x^{\star}  \|} & = \log \frac{\|Z_{k+1}\|}{\|Z_{k}\|} + \log\frac{\sigma_{k+1}}{\sigma_{k}} \label{one-step-relation}\\
					      &= \log\frac{\|Z_{k+1}\|}{\|Z_{k}\|} + \log\left(\Gx\pare{\alpha_{\f}\pare{x^{\star} + Z_{k}, U_{k+1}} } \right) \nonumber
\end{align}
where $\Gx$ and $\alpha_\f$ are defined in~\eqref{step-size} and in~\eqref{alphaDefinition}. This  equation can now be used to express the term whose limit we need to investigate:
 \begin{align}
 \frac1k\log \frac{\|X_{k} - x^{\star} \|}{\|X_{0} - x^{\star} \|} &  = \frac1k\sum_{t=0}^{k-1}\log \frac{\|X_{t+1} - x^{\star} \|}{\|X_{t} - x^{\star} \|} \\ 
 & = \frac1k\sum_{t=0}^{k-1}\log \frac{\|Z_{t+1}\|}{\|Z_{t}\|} + \frac1k\sum_{t=0}^{k-1}\log(\Gx(\alpha_{\f}(x^{\star} + Z_{t}, U_{t+1}) )) \label{apply-lln} \enspace .
 \end{align}

%we exploit the property that on}
%\del{If} $\f$\del{ is} scaling-invariant with respect to $x^{\star}$, $\{Z_k = (X_k - x^{\star})/\sigma_k \, ; k \in \mathbb{N} \}$ is a homogeneous Markov chain where $\{(X_k,\sigma_k) \, ; k \in \mathbb{N}  \}$ is the sequence of states of the step-size adaptive $\pare{\mu/\mu_{w}, \lambda}$-ES defined in~\eqref{incumbent} and~\eqref{sigma-with-alpha} (see Proposition~\ref{prop:MC}). \new{More precisely, s}\del{ S}ince $X_{k} - x^{\star} = \sigma_{k} Z_{k}$, it follows that 
%\begin{align}
%\log \frac{\|X_{k+1} - x^{\star}  \|}{\|X_{k} - x^{\star}  \|} & = \log \frac{\|Z_{k+1}\|}{\|Z_{k}\|} + \log\frac{\sigma_{k+1}}{\sigma_{k}} \label{one-step-relation}\\
%					      &= \log\frac{\|Z_{k+1}\|}{\|Z_{k}\|} + \log\left(\Gx\pare{\alpha_{\f}\pare{x^{\star} + Z_{k}, U_{k+1}} } \right) \nonumber
%\end{align}
%where $\Gx$ and $\alpha_\f$ are defined in~\eqref{step-size} and in~\eqref{alphaDefinition}. 

This latter equation suggests that if we can apply a law of large numbers to $\acco{ Z_k \, ; k \in \mathbb{N}}$ and $\acco{ (Z_{k}, U_{k+1}) \,; k \in \mathbb{N}}$, the right-hand side of~\eqref{apply-lln} converges when $k$ goes to infinity to $\dsp \int \mathbb{E}_{U_{1}\sim\N_{\n\lambda}}\croc{ \log\pare{\Gx \pare{ \alpha_{f}(x^{\star} + z, U_{1})} } } \pi(dz) = \mathbb{E}_{\pi}( \mathcal{R}_{\f} )$ where $\mathcal{R}_{\f}$ is defined \new{as the expected change of the logarithm of the step-size for any state $z \in \R^{\n}$ of the \normalized\ chain as
\begin{align}
\mathcal{R}_{\f}(z) = \mathbb{E}_{U_{1}\sim\N_{\n\lambda}}\croc{ \log\pare{\Gx \pare{ \alpha_{f}(x^{\star} + z, U_{1})} } } \enspace,
\label{expected-step-size-f}
\end{align}}\del{in~\eqref{expected-step-size-f}}and $\pi$ is the invariant measure of $\{ Z_k \, ; k \in \mathbb{N}\}$. From there, we obtain the almost sure convergence of $\dsp\frac1k\log \frac{\|X_{k} - x^{\star}\|}{\|X_{0} - x^{\star}\|}$ towards $\mathbb{E}_{\pi}( \mathcal{R}_{\f} )$ expressed in \eqref{almost-sure-linear-convergence-main} characterizing the asymptotic linear behavior of the algorithm.
\new{A similar equation can be established to prove the convergence of $1/k \log(\sigma_k/\sigma_0)$. Convergence of the expected log-progress can also be deduced from stability properties of $\{Z_k ; k \geq 0 \}$.}

The idea to apply a\del{n} \nnew{Law of Large Numbers (LLN)} to the chain $\{Z_k ; k \geq 0 \}$ to prove the asymptotic linear behavior of the underlying algorithm is the key behind the asymptotic almost sure linear behavior proof we provide.
\new{This seminal idea}\del{It} was introduced \del{by Bienven\"ue and Fran\c cois} \new{for}\del{\cite{bienvenue2003global} on a} self-adaptive ES on the sphere function \cite{bienvenue2003global} \new{and} exploited\del{ by Auger} to prove their linear behavior \cite{auger2005convergence} and\del{ later on} generalized to a wider class of algorithms and functions~\cite{auger2016linear}. 
%We therefore see that we need to investigate under which conditions $\{ Z_k \, ; k \in \mathbb{N}\}$ and $\{ (Z_{k}, U_{k+1}) \,; k \in \mathbb{N}\}$ satisfy an LLN.

%It shapes the rest of the paper. 
\new{Hence, in order to obtain a proof of the linear behavior of the studied algorithm following the idea sketched above, we need to investigate now the stability of the chain $\{ Z_k \, ; k \in \mathbb{N}\}$ (and in turn $\{ (Z_{k}, U_{k+1}) \,; k \in \mathbb{N}\}$). In particular, we need to prove that it
%Yet proving that the chain $\{ Z_k \, ; k \in \mathbb{N}\}$ (and in turn $\{ (Z_{k}, U_{k+1}) \,; k \in \mathbb{N}\}$) 
satisfies the mathematical properties referred to informally as stability properties (following a terminology by Meyn and Tweedie \cite{meyn2012markov}) such that an LLN can be applied. It is  not a  trivial task and it will occupy a large part of the rest of the paper. While establishing stability properties to obtain an LLN we will prove stronger properties that will allow to state convergence of the expected log progress and a Central Limit Theorem.
The outline of the remaining mathematical analysis and the proof structure is  as follows:}
\new{
\begin{itemize}
\item In Section~\ref{stability-markov-process}, we introduce  different notions related to Markov chains, notably the stability properties that we will prove like $\varphi$-irreducibility, aperiodicity, positivity, Harris-recurrence and geometric ergodicity. We also introduce the different practical tools to prove that a Markov chain satisfies those properties.
% together with practical tools to prove that a Markov chain satisfy those properties.
\item In Section~\ref{normalized-chain-section}, we establish those stability properties for the Markov chain $\{ Z_k ; k \geq 0 \}$ associated to a step-size adaptive \muslashmu-ES under the appropriate conditions on the objective functions and the step-size adaptation mechanism.
\item In Section~\ref{consequences-section}, we use those properties to prove the linear behavior of the studied algorithms. In addition to the asymptotic almost sure linear behavior stemming from the LLN, we establish convergence in terms of expected log progress and a Central Limit Theorem. Our conditions for linear convergence are expressed for an abstract step-size update. We investigate how those conditions translate to the case of the CSA and xNES step-size updates. 
%look like in the case 
\end{itemize}
}}
\begin{figure*}~\\[-1.5ex]
\newcommand{\figwidth}{0.24}
	\centering
	\includegraphics[width=\figwidth\textwidth]{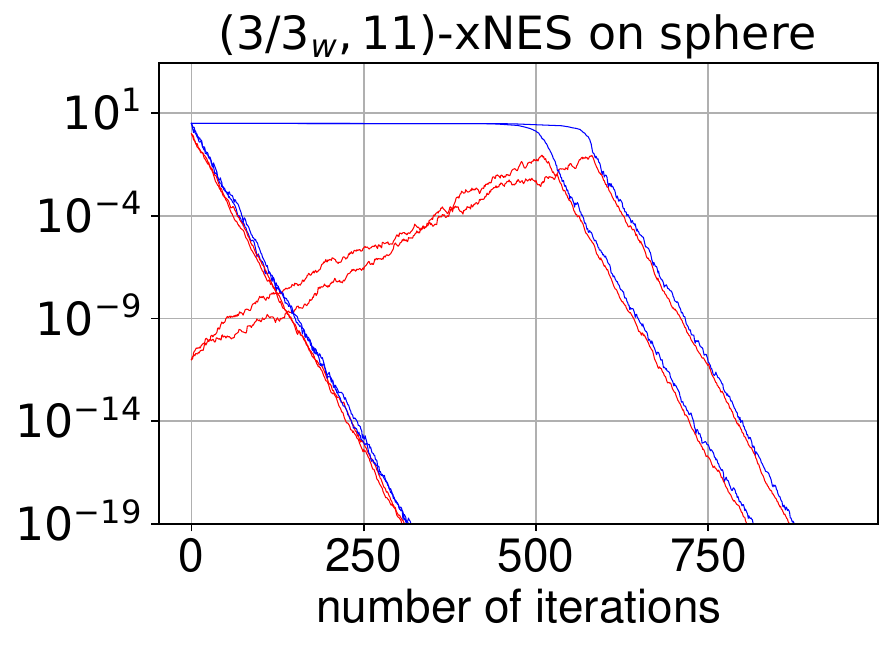}
	\includegraphics[width=\figwidth\textwidth]{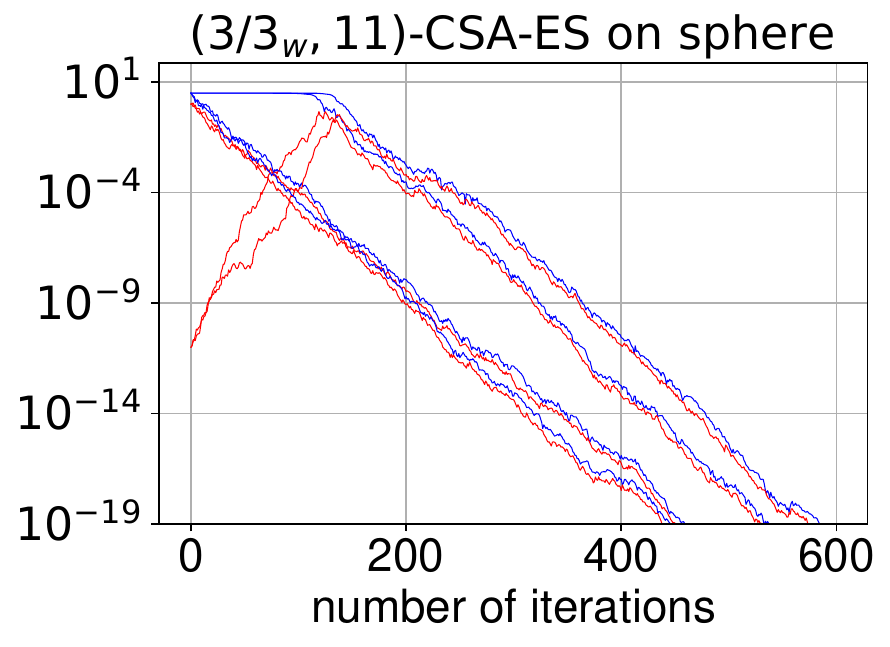}
	\includegraphics[width=\figwidth\textwidth]{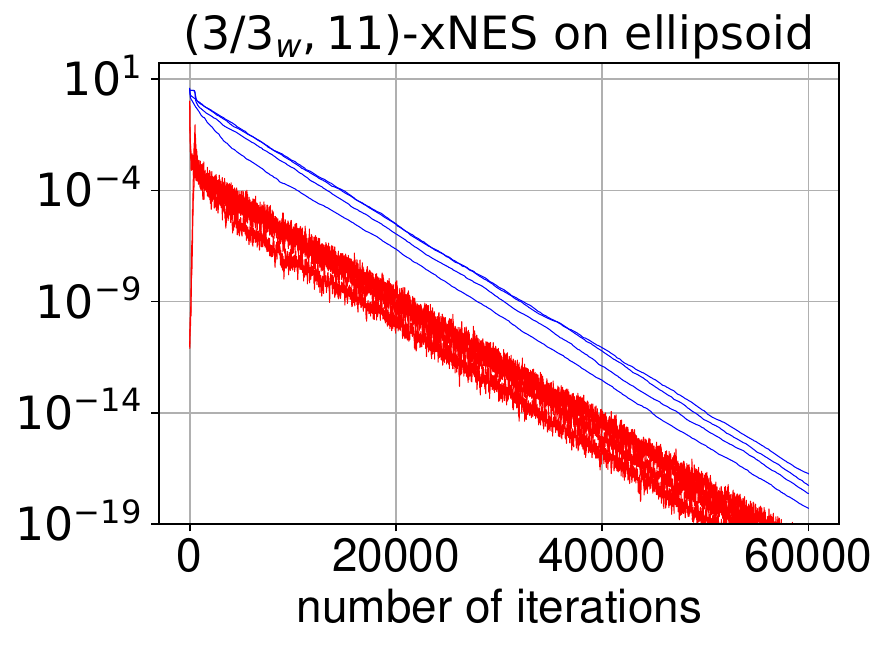}
	\includegraphics[width=\figwidth\textwidth]{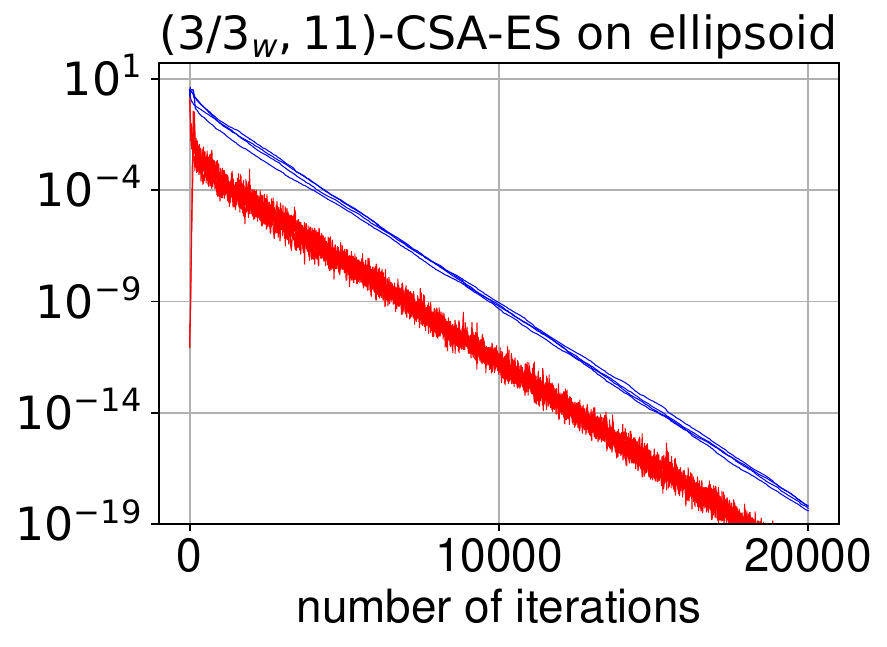}
	\vspace*{-1.5ex}
	\caption{\label{fig:convergence} Four independent runs of $\pare{\mu/\mu_w, \lambda}$-xNES and $\pare{\mu/\mu_w, \lambda}$-CSA1-ES (without cumulation) as presented in Section~\ref{algorithm-framework} on the functions $x \mapsto \| x\|^{2}$ (first two figures) and $x \mapsto \sum_{i=1}^{n}10^{3\frac{i-1}{n-1}} x_{i}^{2}$ (last two figures). Illustration of $\| X_{k }\|$ in blue and $\sigma_{k}$ in red where $k$ is the number of iterations, $\mu = 3$, $\lambda = 11$ and $w_i=1/\mu$.
Initializations:
$\sigma_{0} $ equals to $10^{-11}$ in two runs and $1$ in the two other runs, $X_{0}$ is the all-ones vector in dimension $10$.}
\end{figure*}
\section{\new{Reminders on Markov chains and various tools}}
\label{stability-markov-process}
\label{irreducible-Tchain-aperiodicity}

%We sketch in this section the main steps of our methodology and introduce definitions and tools stemming from Markov chain theory that are needed in the rest of the paper.

\del{
Similarly the proof idea for \eqref{expected-log-progress-main} goes as follows.  %
 If we take the expectation under $Z_{0} = z$ in~\eqref{sigma-with-alpha}, then 
 $
 \mathbb{E}_{z}\croc{\log\frac{\sigma_{k+1}}{\sigma_{k}}} = \mathbb{E}_{z}\croc{\log\left(\Gx\pare{\alpha_{\f}\pare{x^{\star} + Z_{k}, U_{k+1}} } \right)} =  \int P^{k}(z, \diff y) \mathcal{R}_{\f}(y).
 $
 With~\eqref{one-step-relation} we have that 
\begin{align}
 & \mathbb{E}_{z} \croc{\log \frac{\|X_{k+1} - x^{\star} \|}{\|X_{k} - x^{\star} \|}} = \mathbb{E}_{z}\croc{\log \frac{\|Z_{k+1}\|}{\|Z_{k}\|} } +  \mathbb{E}_{z}\croc{\log\frac{\sigma_{k+1}}{\sigma_{k}} } \label{expected-step}\\
 &=   \dsp\int P^{k+1}(z, \diff y) \log(\|y\|) - \dsp\int P^{k}(z, \diff y) \log(\|y\|) 
 %\nonumber\\&
 +   \dsp\int P^{k}(z, \diff y) \mathcal{R}_{\f}(y) \label{expected-step-plus}.
 \end{align}
  
 If $P^k(z,.)$ converges to $\pi$ assuming all the limits can be taken, then the right-hand side converges to $\mathbb{E}_\pi(\mathcal{R}_{\f})$ as the two first integrals cancel each other such that 
 $
\dsp\lim_{k\to\infty} \mathbb{E}_{\frac{x- x^{\star}}{\sigma}}\croc{ \log\frac{ \| X_{k+1}- x^{\star} \|   }{ \| X_{k}- x^{\star} \|   } } = \dsp\lim_{k\to\infty} \mathbb{E}_{\frac{x- x^{\star}}{\sigma}} \croc{  \log\frac{ \sigma_{k+1} }{ \sigma_{k} }}  = \mathbb{E}_{\pi}(\mathcal{R}_{\f}) ,
 $
i.e.\ ~\eqref{expected-log-progress-main} is satisfied. To prove that $P^k(z,.)$ converges to $\pi$ we prove the $h$-ergodicity of $\{Z_k \, ; k \in \mathbb{N}\}$, a notion formally defined later on.\\}

We consider a Markov chain $\{Z_k \, ; k \in \mathbb{N}\}$ on a measure space $\pare{\ZZ, \mathcal{B}(\ZZ), P}$ where $\ZZ$ in an open subset of $\R^{\n}$,  for all $k \in \mathbb{N}$ its $k$-step transition kernel as
$
P^{k}(z, A) = P\pare{Z_{k} \in A | Z_{0} = z}
$ 
for $ z \in \ZZ, \, A \in \mathcal{B}(\ZZ)$. We also denote $P(z, A)$ and $P_{z}(A)$ as $P^{1}(z, A)$.
We remind different stability notions investigated later on to prove in particular that $\{ Z_k \, ; k \in \mathbb{N}\}$ satisfies an LLN, a central limit theorem and that for some $z\in\ZZ$, $P^k(z, \cdot)$ converges to a stationary distribution.
%Following the terminology of \cite{meyn2012markov}, we refer in an informal way to those properties as stability properties.
We additionally present different tools to be able to verify that a Markov chain satisfies those various properties.

\subsection{Stability properties and practical drift conditions}
%The first stability property to be verified is the so-called $\varphi$-irreducibility. 
If there exists a nontrivial measure $\varphi$ on $\pare{\ZZ, \mathcal{B}(\ZZ)}$ such that for all $A \in \mathcal{B}(\ZZ),$
$ \varphi(A) > 0$ implies $\sum_{k=1}^{\infty} P^{k}(z, A) > 0 \text{ for all $z \in \ZZ$}, $
then the chain is called $\varphi$-irreducible. A $\varphi$-irreducible Markov chain is Harris recurrent if for all $A \in \mathcal{B}(\ZZ)$ with $\varphi(A) > 0$ and for all $z \in \ZZ,$
 $P_{z}\pare{ \eta_{A} = \infty } = 1,$ where $\eta_{A} = \sum_{k=1}^{\infty} \mathds{1}_{Z_{k} \in A}$  is the \textit{occupation time} of $A.$  
 
A $\sigma$-finite measure $\pi$ on $\pare{\ZZ, \mathcal{B}(\ZZ)}$ is an invariant measure for $\{ Z_k; k \in \mathbb{N}\} $ if for all $A \in \mathcal{B}(\ZZ),$ $\pi(A) = \dsp\int_{\ZZ} \pi(dz) P(z, A)$.
A Harris recurrent chain admits a unique (up to constant multiples) invariant measure $\pi$ (see \cite[Theorem 10.0.1]{meyn2012markov}). A $\varphi$-irreducible Markov chain admitting an invariant probability measure $\pi$ is said positive. A positive Harris-recurrent chain satisfies an LLN as reminded below.
\begin{theorem}{\cite[Theorem 17.0.1]{meyn2012markov}}
 If $\{Z_k ; k \in \mathbb{N}\}$ is a positive and Harris recurrent chain with invariant probability measure $\pi$, then the LLN holds for any $\pi$-integrable function $g$, i.e.\  
for any $g$ with $\mathbb{E}_{\pi}(\abs{g}) < \infty$, $\lim_{k\to\infty} \frac1k \sum_{t=0}^{k-1} g(Z_{t}) = \mathbb{E}_{\pi}(g)$.
\label{lln-origin}
\end{theorem}
%We prove positivity and Harris-recurrence using Foster-Lyapunov drift conditions. Before introducing those conditions 
We will need the notion of aperiodicity.
 Assume that $d$ is a positive integer and $\{Z_k \, ; k \in \mathbb{N}\}$ is a $\varphi$-irreducible Markov chain defined on $\pare{\ZZ, \mathcal{B}(\ZZ)}$. Let $\pare{D_{i}}_{i=1,\dots,d} \in \mathcal{B}(\ZZ)^{d}$ be a sequence of disjoint sets. Then $\pare{D_{i}}_{i=1,\dots,d} $ is called a $d$-cycle if 
\begin{itemize}
\item[(i)] $P(z, D_{i+1}) = 1$ for all $z \in D_{i}$ and $i = 0, \dots, d-1$ (mod  $d$),
\item[(ii)] $\Lambda\pare{ \pare{ \bigcup_{i=1}^{d} D_{i} }^{c} } = 0$ for all irreducibility measure $\Lambda$ of $\{Z_k \, ; k \in \mathbb{N}\}$.
\end{itemize}  
If $\{Z_k ; k \in \mathbb{N}\}$ is $\varphi$-irreducible, there exists a $d$-cycle where $d$ is a positive integer \cite[Theorem 5.4.4]{meyn2012markov}. The largest $d$ for which there exists a $d$-cycle is called the period of $\{Z_k \, ; k \in \mathbb{N}\}$. We then say that a $\varphi$-irreducible Markov chain $\{Z_k \, ; k \in \mathbb{N}\}$ on $\pare{\ZZ, \mathcal{B}(\ZZ)}$ is aperiodic if it has a period of $1$. 

A set $C \in \mathcal{B}(\ZZ)$ is called {\it small} if there exists a positive integer $k$ and a nontrivial measure $\nu_{k}$ on $\mathcal{B}(\ZZ)$ such that
$
P^{k}(z, A) \geq \nu_{k}(A) \text{ for all } z \in C, \, A \in  \mathcal{B}(\ZZ).
$
We then say that $C$ is a $\nu_{k}$-small set~\cite{meyn2012markov}. 

Given an extended-valued, non-negative and measurable function  $V:$ $\foncfleche{\ZZ}{\R_{+}\cup \acco{\infty}}$ (called potential function), the drift operator is defined 
for all $z \in \ZZ$ as 
$
\Delta V(z) = \mathbb{E}\croc{ V(Z_{1}) | Z_{0} = z} - V(z) = \dsp\int_{\ZZ} V(y) P(z, \diff y)- V(z) .
$
A $\varphi$-irreducible, aperiodic Markov chain $\{Z_k \, ; k \in \mathbb{N}\}$ defined on $\pare{\ZZ, \mathcal{B}(\ZZ)}$ satisfies a geometric drift condition if there exist $0 < \gamma < 1,$ $b \in \R,$ a small set $C$ and a potential function $V$ greater than $1$, finite at some $z_{0} \in \ZZ$ such that for all $z \in \ZZ:$
 $
 \Delta V(z) \leq (\gamma - 1) V(z) + b \mathds{1}_{C}(z) ,
 $
or equivalently if 
$\mathbb{E}\croc{ V(Z_{1}) | Z_{0} = z} \leq \gamma V(z) + b \mathds{1}_{C}(z).$
The function $V$ is called a geometric drift function and if $\acco{y \in \ZZ \, ; V(y) < \infty} = \ZZ$, we say that $\{Z_k \, ; k \in \mathbb{N}\}$ is $V$-geometrically ergodic.

If a $\varphi$-irreducible and aperiodic Markov chain  is $V$-geometrically ergodic, then it is positive and Harris recurrent \cite[Theorem 13.0.1 and Theorem 9.1.8]{meyn2012markov}. 
We prove a geometric drift condition in Section~\ref{geometric-dirft-condition-for-main}, this in turn implies positivity and Harris-recurrence\del{ property}.

From a geometric drift condition follows a stronger result than an LLN, namely a central limit theorem.

\begin{theorem}{\cite[Theorem 17.0.1 and Theorem 16.0.1]{meyn2012markov}}
 Let $\{Z_k \, ; k \in \mathbb{N}\}$ be a $\varphi$-irreducible aperiodic Markov chain on $\pare{\ZZ, \mathcal{B}(\ZZ)}$ that is $V$-geometrically ergodic with invariant probability measure $\pi$. For any function $g$ on $\ZZ$ that satisfies $g^{2} \leq V$, the central limit theorem holds for $\{Z_k \, ; k \in \mathbb{N}\}$ in the following sense.
Define $\bar{g} = g - \mathbb{E}_{\pi}(g)$ and for all positive integer $t,$ define $S_{t}( \bar{g} ) = \sum_{k=0}^{t-1} \bar{g}(Z_{k})$. Then the constant $\gamma^{2} = \mathbb{E}_{ \pi } [ (\bar{g}(Z_{0}) )^{2} ] + 2 \sum_{k=1}^{\infty}  \mathbb{E}_{ \pi } [\bar{g}(Z_{0}) \bar{g}(Z_{k}) ]$
is well defined, non-negative, finite and $\dsp\lim_{t\to\infty}\frac{1}{t}  \mathbb{E}_{ \pi } [ (S_{t}( \bar{g}) )^{2} ] = \gamma^{2}$. Moreover if $\gamma^{2} > 0$ then $\frac{1}{\sqrt{t \gamma^{2}}} S_{t}(\bar{g})$ converges in distribution to $\N(0, 1)$ when $t$ goes to $\infty$, else if $\gamma^{2} = 0$ then $\frac{1}{\sqrt{t} } S_{t}(\bar{g}) = 0$ $a.s.$ 
\label{clt-generic}
\end{theorem}

For a measurable function $h \geq 1$ on $\ZZ$,
%~\cite[Theorem 14.0.1]{meyn2012markov} states that
a $\varphi$-irreducible aperiodic Markov chain $\{Z_k \, ; k \in \mathbb{N}\}$ defined on $\pare{\ZZ, \mathcal{B}(\ZZ)}$ is positive Harris recurrent with invariant probability measure $\pi$ such that $h$ is $\pi$-integrable if and only if there exist $b \in \R,$ a small set $C$ and an extended-valued non-negative function $V \neq \infty$ such that
 \begin{align}
 \Delta V(z) \leq -h(z) + b \mathds{1}_{C}(z)
 \label{h-ergodic-drift}
 \end{align}
for all $z \in \ZZ$ \cite[Theorem 14.0.1]{meyn2012markov}.
Recall that for a measurable function $h \geq 1$, we say that a general Markov chain $\{Z_k \, ; k \in \mathbb{N}\}$ is $h$-ergodic if there exists a probability measure $\pi$ such that $\dsp \lim_{k \to \infty} \| P^{k}(z, \cdot) - \pi \|_{h} = 0$ for any initial condition $z$. The probability measure $\pi$ is then called the invariant probability measure of $\{Z_k \, ; k \in \mathbb{N}\}$. If $h = 1$, we say that $\{Z_k \, ; k \in \mathbb{N}\}$ is ergodic.

\del{With~\cite[Theorem 14.0.1]{meyn2012markov},}%
A $\varphi$-irreducible aperiodic Markov chain on $\ZZ$ that satisfies~\eqref{h-ergodic-drift} is $h$-ergodic if in addition $\acco{y \in \ZZ \, ; V(y) < \infty} = \ZZ$ \cite[Theorem 14.0.1]{meyn2012markov}.
 
Prior to establishing a drift condition, we need to identify small sets. Using the notion of T-chain defined below, compact sets are small sets\del{as a direct consequence of \cite[Theorem 5.5.7 and Theorem 6.2.5]{meyn2012markov} stating that}
\new{because}
for a $\varphi$-irreducible aperiodic T-chain, every compact set is a small set \cite[Theorem 5.5.7 and Theorem 6.2.5]{meyn2012markov}.

The T-chain property calls for the notion of kernel: a kernel $K$ is a function on $\pare{\ZZ, \mathcal{B}(\ZZ)}$ such that for all $A \in \mathcal{B}(\ZZ),$ $K(., A)$ is a measurable function and for all $z \in \ZZ,$ $K(z, .)$ is a signed measure.
A non-negative kernel $K$ satisfying $K(z, \ZZ) \leq 1$ for all $z \in \ZZ$ is called substochastic.
A substochastic kernel $K$ satisfying $K(z, \ZZ) = 1$ for all $z \in \ZZ$ is a transition probability kernel. 
Let $b$ be a probability distribution on $\mathbb{N}$ and denote by $K_{b}$ the probability transition kernel defined as 
$ K_{b}(z, A) = \sum_{k=0}^{\infty} b(k) P^{k}(z, A) \text{ for all $z \in \ZZ,$ \,$ A \in \mathcal{B}(\ZZ)$}.$
If $T$ is a substochastic transition kernel such that $T(., A)$ is lower semi-continuous for all $A \in \mathcal{B}(\ZZ)$ and 
$
K_{b}(z, A) \geq T(z, A) \text{ for all $z \in \ZZ,$ \,$ A \in \mathcal{B}(\ZZ)$},
$
then $T$ is called a continuous component of $K_{b}.$
If a Markov chain $\{Z_k \, ; k \in \mathbb{N}\}$ admits a probability distribution $b$ on $\mathbb{N}$ such that $K_{b}$ has a continuous component $T$ that satisfies $T(z, \ZZ) > 0$ for all $z \in \ZZ,$ then $\{Z_k \, ; k \in \mathbb{N}\}$ is called a $T$-chain. 

\subsection{Generalized law of large numbers}
\label{sec:GLLN}

To apply an LLN for the convergence of the term $\frac1k\sum_{t=0}^{k-1}\log(\Gx(\alpha_{\f}(x^{\star} + Z_{t}, U_{t+1}) ))$ in~\eqref{apply-lln}, we proceed in two steps. First we prove that if $\{Z_k \, ; k \in \mathbb{N}\}$ is defined as $Z_{k+1} = G(Z_{k}, U_{k+1} )$ where $G: \ZZ \times \R^{m} \to \ZZ$ is a measurable function and $\acco{U_{k+1} \, ; k \in \mathbb{N} }$ is a sequence of i.i.d.\ random vectors, then the ergodic properties of $\{Z_k \, ; k \in \mathbb{N}\}$ are transferred to $\{W_k = (Z_{k}, U_{k+2}) \, ; k \in \mathbb{N}\}$. Afterwards we apply a generalized LLN recalled in the following theorem.
%%% THEOREM
\begin{theorem}[{\!\cite[Theorem 1]{jensen2007law}}]
\label{general-lln}
Assume that $\{Z_k \, ; k \in \mathbb{N}\}$ is a homogeneous Markov chain on an abstract measurable space $\pare{\mathbf{E}, \mathcal{E}}$ that is ergodic with invariant probability measure $\pi$. 
For all measurable function $g: \mathbf{E}^{\infty} \to \R$ such that for all $s\in \mathbb{N}$, $\mathbb{E}_{\pi }(\abs{g(Z_{s}, Z_{s+1}, \dots)} ) < \infty$ and for any initial distribution $\Lambda$, the generalized LLN holds as follows
%\begin{align*} 
$
\lim_{k\to\infty} \frac1k \sum_{t=0}^{k-1} g\pare{Z_{t}, Z_{t+1}, \dots } = \mathbb{E}_{\pi }(g(Z_{s}, Z_{s+1}, \dots)) \enspace P_{\Lambda} \enspace a.s. \enspace 
$
%\end{align*}
where $P_{\Lambda}$ is the distribution of the process $\acco{Z_{k} \, ; k \in \mathbb{N}}$ on $\pare{ \mathbf{E}^{\infty}, \mathcal{E}^{\infty} }$.
\end{theorem}
Theorem~\ref{general-lln} generalizes\del{~\cite[Theorems 3.5.7 and 3.5.8]{stout1974almost}. Indeed in~\cite[Theorems 3.5.7 and 3.5.8]{stout1974almost}, the generalized LLN holds only if}
\new{the case where} the initial state is distributed under the invariant measure \cite[Theorems 3.5.7 and 3.5.8]{stout1974almost} 
\new{to an arbitrary}\del{, whereas in~\cite[Theorem 1]{jensen2007law}, the} initial distribution.\del{ considered can be anything.}

If we have the generalized LLN for a chain $\{(Z_k, U_{k+2}) \, ; k \in \mathbb{N}\}$ on $\R^{\n} \times \R^{m}$, then an LLN for the chain $\{(Z_k, U_{k+1}) \, ; k \in \mathbb{N}\}$ is directly implied. We formalize this statement in the next corollary. 
\begin{corollary} 
\label{coro-lln}
Assume that $\{W_{k} = (Z_k, U_{k+2}) \, ; k \in \mathbb{N}\}$ is a homogeneous Markov chain on $\R^{\n} \times \R^{m}$ that is ergodic with invariant probability measure $\pi$. Then the LLN holds for $\{(Z_k, U_{k+1}) \, ; k \in \mathbb{N}\}$ in the following sense. Define the function $T: \pare{\R^{\n} \times \R^{m}}^{2} \to \R^{\n} \times \R^{m}$ as $T( (z_1, u_3), (z_2, u_4) ) = (z_2, u_3)$.
If $g:  \R^{\n} \times \R^{m} \to \R$ is such that for all $s \in \mathbb{N}$, $\mathbb{E}_{\pi}(\abs{g\circ T}(W_{s}, W_{s+1})) < \infty$, then $\lim_{k\to\infty} \frac1k \sum_{t=0}^{k-1} g(Z_{t}, U_{t+1}) = \mathbb{E}_{\pi} \croc{\left(g\circ T\right) \left( W_{s}, W_{s+1} \right) } $.
\end{corollary}
\begin{proof}
We have $\lim_{k\to\infty} \frac1k \sum_{t=0}^{k-1} \left(g\circ T\right)\left(W_{t}, W_{t+1}\right) = \mathbb{E}_{\pi} \croc{\left(g\circ T\right) \left( W_{s}, W_{s+1} \right) }$ thanks to Theorem~\ref{general-lln}. For $t\in\mathbb{N}$, $\left(g\circ T\right)\left(W_{t}, W_{t+1}\right) = g\left( Z_{t+1}, U_{t+2} \right)$. Therefore \\
$\dsp\mathbb{E}_{\pi} \croc{\left(g\circ T\right) \left( W_{s}, W_{s+1} \right) } = \lim_{k\to\infty} \frac1k \sum_{t=0}^{k-1} g\left(Z_{t+1}, U_{t+2}\right) 
= \lim_{k\to\infty} \frac1k \sum_{t=0}^{k-1} g(Z_{t}, U_{t+1})  .$
\end{proof}

We formulate now that for a Markov chain following a non-linear state space model of the form $Z_{k+1} = G(Z_{k}, U_{k+1} )$ with $G$ measurable and  $\acco{U_{k+1}  \, ; k \in \mathbb{N}}$ i.i.d., then $\varphi$-irreducibility, aperiodicity and $V$-geometric ergodicity of $Z_k$ are transferred to $\{W_k = (Z_{k}, U_{k+2}) \, ; k \in \mathbb{N}\}$. We provide a proof of this result in Appendix~\ref{proof-general-stability} for the sake of completeness.
\begin{proposition}
Let $\{Z_k \, ; k \in \mathbb{N}\}$ be a Markov chain on $(\ZZ, \mathcal{B}(\ZZ))$ defined as $Z_{k+1} = G(Z_{k}, U_{k+1} )$ where $G: \ZZ \times \R^{m} \to \ZZ$ is a measurable function and $\acco{U_{k+1}  \, ; k \in \mathbb{N}}$ is a sequence of i.i.d.\ random vectors with probability measure $\Psi$. 
Consider $\{W_k = (Z_{k}, U_{k+2}) \, ; k \geq 0\}$, then it is a Markov chain on $\mathcal{B}(\ZZ) \otimes \mathcal{B}(\R^{m})$ which inherits properties of $\{Z_k \, ; k \in \mathbb{N}\}$ in the following sense:
\begin{itemize}
\item If $\varphi$ (resp. $\pi$) is an irreducibility (resp. invariant) measure of $\{Z_k \, ; k \in \mathbb{N}\}$, then $\varphi \times \Psi$ (resp. $\pi \times \Psi$) is an irreducibility (resp. invariant) measure of $\{W_k \, ; k \in \mathbb{N} \}$.
\item The set of integers $d$ such that there exists a $d$-cycle for $\{Z_k \, ; k \in \mathbb{N}\}$ is equal to the set of integers $d$ such that there exists a $d$--cycle for $\{W_k \, ; k \in \mathbb{N}\} $. In particular $\{Z_k \, ; k \in \mathbb{N}\}$ and $\{W_k \, ; k \in \mathbb{N} \}$ have the same period. Therefore $\{Z_k \, ; k \in \mathbb{N}\}$ is aperiodic if and only if $\{W_k \, ; k \in \mathbb{N} \}$ is aperiodic.
\item If $C$ is a small set for $\{Z_k \, ; k \in \mathbb{N}\}$, then $C \times \R^m$ is a small set for $\{ W_k \,; k \in \mathbb{N} \}$.
\item If $\{Z_k \, ; k \in \mathbb{N}\}$ satisfies a drift condition 
\begin{align}
\Delta V(z) \leq - \beta h(z) + b \mathds{1}_{C}(z) \enspace \mbox{ for all } z \in \ZZ,
\label{generic-drift}
\end{align} 
where $V$ is a potential function, $0 < \beta < 1$, $h \geq 0$ is a measurable function and $C \subset \ZZ$ is a measurable set, 
then $\{ W_k \,; k \in \mathbb{N} \}$ satisfies the following drift condition for all $(z, u) \in \ZZ \times \R^{m}:$
 %\begin{align}
 $
 \Delta \widetilde{V}(z, u) \leq - \beta \widetilde{h}(z, u) + b \mathds{1}_{C\times \R^{m}}(z, u),
 \label{generic-drift-couple}
$
% \end{align}
  where $\widetilde{V}: (z,u) \mapsto V(z)$ and $\widetilde{h}: (z,u) \mapsto h(z) $.
 \end{itemize}
\label{general-stability}
\end{proposition}

Remark that the drift condition in \eqref{generic-drift} includes the geometric drift condition by taking $h=V$, the drift condition for $h$-ergodicity by dividing the equation by $\beta$ and assuming that $h \geq 1$, for positivity and Harris recurrence by taking $h=1/\beta$, and for Harris recurrence by taking $h=0$. This is obtained assuming that $V$ and $C$ satisfy the proper assumptions for the drift to hold.

\subsection{$\varphi$-irreducibility, aperiodicity and $T$-chain property via deterministic control models}
\label{deterministic-control-model}
%Proving that a Markov chain is a $\varphi$-irreducible aperiodic T-chain can sometimes be immediate. 
%In our case however, 
For the Markov chain considered, it is difficult to establish $\varphi$-irreducible, aperiodicity and the T-chain property ``by hand". We thus resort to tools connecting those properties to stability properties of the underlying control model \cite[Chapter~13]{meyn2012markov} \cite{chotard2019verifiable}.
%We remind here the different notions needed and refer to \cite{chotard2019verifiable} for more details. 
Assume that $\ZZ$ is an open subset of $\R^{\n}$. We consider a Markov chain that takes the following form
 \begin{align}
Z_{k+1} = F\pare{Z_{k}, \alpha\pare{Z_{k},\, U_{k+1}}},
\label{nonlinear-alpha}
\end{align}
where $Z_{0} \in \ZZ$ and for all natural integer $k,$ $F: \ZZ \times \R^{\n\mu} \to \ZZ$ and $\alpha: \ZZ \times \R^{\n\lambda} \to \R^{\n\mu}$ are measurable functions,  $U = \acco{U_{k+1} \in \R^{\n\lambda} \, ; \, k \in \mathbb{N}}$ is a sequence of i.i.d.\ random vectors. We consider the following assumptions on the model:
\begin{itemize}
\item[B1.] $\pare{Z_{0}, U }$ are random variables on a probability space $\pare{ \Omega, \mathcal{F}, P_{Z_{0}} }.$
\item[B2.] $Z_{0}$ is independent of $U.$
\item[B3.] $U$\ is an independent and identically distributed process.
\item[B4.] For all $z \in \ZZ,$ the random variable $\alpha(z, U_{1})$ admits a probability density function denoted by $p_{z}$, such that the function $\foncsmall{(z,u)}{p_{z}(u)}$ is lower semi-continuous.
\item[B5.] The function $F: \ZZ \times \R^{\n\mu} \to \ZZ$  is $C^{1}.$
\end{itemize}

We recall the deterministic control model related to~\eqref{nonlinear-alpha} denoted by  CM($F$) \cite{chotard2019verifiable}. It is based on the notion of extended transition map function~\cite{meyn1991asymptotic}, defined recursively for all $z \in \ZZ$ as $S_{z}^{0} = z$, and for all $k \in \mathbb{N}\backslash{\acco{0}}$, $S_{z}^{k}: \R^{\n\mu k} \to \ZZ$ such that for all $\bf{w}$ $= \pare{ w_{1},\dots, w_{k} } \in \R^{\n\mu k}$, 
$
S_{z}^{k}(\textbf{w}) = F\pare{S_{z}^{k-1}\pare{w_{1}, \dots, w_{k-1}}, w_{k} }.
$
Assume in the following that Assumptions B1$-$B4 are satisfied and that $F$ is continuous.

Let us define the process $W$ for all $k \in \mathbb{N}\backslash\acco{0}$ and  $z \in \ZZ$  as
$W_{1} = \alpha\pare{ z, \,U_{1} }$ and
$ W_{k} = \alpha\pare{ S_{z}^{k-1}(W_{1},\dots, W_{k-1}), \,U_{k} }$.
Then the probability density function of $( W_{1}, W_{2},\dots, W_{k} )$ denoted by $p_{z}^{k}$ is what is called the extended probability function. It is defined inductively for all $k \in \mathbb{N}\backslash{\acco{0}}$, $\bf{w}$ $= \pare{w_{1}, \dots, w_{k} } \in \R^{\n\mu k}$ by $p_{z}^{1}(w_{1}) = p_{z}(w_{1})$ and
$p_{z}^{k}(\bf w) =$ $p_{z}^{k-1}\pare{ w_{1},\dots, w_{k-1}} p_{S_{z}^{k-1}(w_{1},\dots, w_{k-1}) }(w_{k})$.
For all $k \in \mathbb{N}\backslash\acco{0}$ and for all $z \in \ZZ$, the control sets are finally defined as 
$\mathcal{O}_{z}^{k} = \acco{ \bf{w}\text{ $\in$ $\R^{\n\mu k}$ }  \, ; \text{ $p_{z}^{k}(\bf{w})$ } \text{$ > 0$} }.$
The control sets are open sets since $F$ is continuous and the functions $(z, \bf w)$ $\mapsto$ $p_{z}^{k}(\bf{w})$ are lower semi-continuous  (see~\cite{chotard2019verifiable} for more details).

The deterministic control model CM($F$) is defined recursively for all $k \in \mathbb{N},$ $z \in \ZZ$ and $\pare{w_{1}, \dots, w_{k+1} } \in \mathcal{O}_{z}^{k+1}$ as 
$
S_{z}^{k+1}\pare{w_{1}, \dots, w_{k+1} } = F\pare{S_{z}^{k}\pare{w_{1}, \dots, w_{k}}, w_{k+1} }.
$

For $z \in \ZZ$, $A \in \mathcal{B}(\ZZ)$ and $k \in  \mathbb{N}\backslash{\acco{0}}$, we say that $\bf w$ $\in \R^{\n\mu k}$ is a $k$-steps path from $z$ to $A$ if $\bf w$ $\in \mathcal{O}_{z}^{k}$ and $S_{z}^{k}(\bf w)$ $\in A$.
We introduce for $z \in \ZZ$ and $k \in \mathbb{N}$ the set of all states reachable from $z$ in $k$ steps by  CM($F$), denoted by $A_{+}^{k}(z)$ and defined as $A_{+}^{0}(z) = \acco{z} $ and $A_{+}^{k}(z) = \acco{ S_{z}^{k}(\text{$\bf w$}) \, ; \,\text{$\bf w$} \in \mathcal{O}_{z}^{k} }$.

A point $z \in \ZZ$ is a steadily attracting state if for all $y \in \ZZ$, there exists a sequence $\acco{ y_{k} \in A_{+}^{k}(y)| \, k \in \mathbb{N}\setminus\acco{0} }$ that converges to $z$. 

The controllability matrix is defined for $k \in \mathbb{N}\backslash\acco{0}$, $z \in \ZZ$ and $\bf w$ $\in \R^{\n\mu k}$ as the Jacobian matrix of $\foncsmall{\pare{w_{1},\dots,w_{k} }}{ S_{z}^{k}\pare{w_{1},\dots,w_{k} } }$ and denoted by $C_{z}^{k}(\bf w)$. Namely,
$ C_{z}^{k}(\text{\bf{w}}) = \begin{bmatrix}
 \frac{\partial S_{z}^{k} }{\partial w_{1} }(\text{\bf{w}})|   \dots |  \frac{\partial S_{z}^{k} }{\partial w_{k} }(\text{\bf{w}})
\end{bmatrix}.$

If $F$ is $C^{1}$, the existence of a steadily attracting state $z$ and a full-rank condition on a controllability matrix of $z$ imply that a Markov chain following~\eqref{nonlinear-alpha} is a $\varphi$-irreducible aperiodic $T$-chain, as reminded in the next theorem.

\begin{theorem}{\cite[Theorem 4.4: Practical condition to be a $\varphi$-irreducible aperiodic T-chain.]{chotard2019verifiable}} Consider a Markov chain $\{Z_k \, ; k \in \mathbb{N}\}$ following the model \eqref{nonlinear-alpha} for which the conditions B1$-$B5 are satisfied. If there exist a steadily attracting state $z\in\ZZ$, $k \in \mathbb{N}\backslash\acco{0}$ and $\text{\bf w} \in \mathcal{O}_{z}^{k}$ such that $\text{\rm rank}\pare{C_{z}^{k}(\text{$\bf w$})} = n $, then $\{Z_k \, ; k \in \mathbb{N}\}$ is a $\varphi$-irreducible aperiodic T-chain, and every compact set is a small set.
\label{condition-upstream}
\end{theorem}

The next lemma allows to loosen the full-rank condition stated above if the control set $\mathcal{O}_{z}^{k}$ is dense in $\R^{\n \mu k}$.
\begin{lemma} Consider a Markov chain $\{Z_k \, ; k \in \mathbb{N}\}$ following the model \eqref{nonlinear-alpha} for which the conditions B1$-$B5 are satisfied. Assume that there exist a positive integer $k$ and $z \in \ZZ$ such that the control set $\mathcal{O}_{z}^{k}$ is dense in $\R^{\n\mu k}$. If there exists $\bf\widetilde{w}$ $\in \R^{\n\mu k}$ such that $\text{\rm rank}(C_{z}^{k}(\text{$\bf \widetilde{w}$})) = n$, then the rank condition in Theorem~\ref{condition-upstream} is satisfied, i.e.\  there exists $\bf w$ $\in \mathcal{O}_{z}^{k}$ such that $\text{\rm rank}(C_{z}^{k}(\text{$\bf w$})) = \n.$
\label{rankCondition}
\end{lemma}

\begin{proof}
The function $\foncsmall{w}{S_{z}^{k}(w)}$ is $C^{1}$\del{ thanks to}~\cite[Lemma 6.1]{chotard2019verifiable}. Since the set of full rank matrices is open, there exists an open neighborhood $\mathcal{V}_{ \text{$\bf \widetilde{w}$} }$ of $\bf \widetilde{w}$ such that  for all $w \in \mathcal{V}_{ \text{$\bf \widetilde{w}$} },$ $\text{\rm rank}(C_{z}^{k}(w)) = \n.$ By density of $\mathcal{O}_{z}^{k},$ the non-empty set $\mathcal{V}_{ \text{$\bf \widetilde{w}$} } \cap \mathcal{O}_{z}^{k}$ contains an element $\bf w$.  
\end{proof}

If $z$ is steadily attracting,  there exists under mild assumptions an open set outside of a ball centered at $z$, with positive measure with respect to the invariant probability measure of a chain following the model~\eqref{nonlinear-alpha} as stated next.
% We state this result in the next lemma.

\begin{lemma}
Consider a Markov chain $\{Z_k \, ; k \in \mathbb{N}\}$ on $\R^{n}$ following the model \eqref{nonlinear-alpha} for which the conditions B1$-$B5 are satisfied. Assume that there exist a steadily attracting state $z\in\R^{n}$ such that $\mathcal{O}_{z}^{1}$ is dense in $\R^{n}$ and $w \in \mathcal{O}_{z}^{1}$ with $\text{\rm rank}\pare{C_{z}^{1}(w)} = n$.
Assume also that $\{Z_k \, ; k \in \mathbb{N}\}$ is a positive Harris recurrent chain with invariant probability measure $\pi.$ 
Then there exists $0 < \epsilon < 1$ such that $\pi(\R^{\n} \setminus \overline{\mathbf{B}\pare{z, \epsilon} } ) > 0.$
\label{non-negligible}
\end{lemma}

\begin{proof}
A $\varphi$-irreducible Markov chain admits a maximal irreducibility measure $\psi$ dominating any other irreducibility measure \cite[Theorem 4.0.1]{meyn2012markov}. In other words, for a measurable set $A$, $\psi(A) = 0$ induces that $\varphi(A) = 0$ for any irreducibility measure $\varphi.$
\new{The measure}\del{Thanks to~\cite[Theorem 10.4.9]{meyn2012markov},} $\pi$ is equivalent to the maximal irreducibility measure $\psi$~\cite[Theorem~10.4.9]{meyn2012markov}.
Since $z$ is steadily attracting,\del{the}\del{then thanks to {\cite[Proposition 3.3]{chotard2019verifiable}} and {\cite[Proposition 4.2]{chotard2019verifiable}},}
$\textrm{supp } \psi = \overline{ A_{+}(z) } = \overline{ \bigcup_{k\in\mathbb{N} } \acco{S_{z}^{k}(\bf w) \, ; \bf w \in \text{$\mathcal{O}_{z}^{k}$ } } }$
\cite[Propositions 3.3 and 4.2]{chotard2019verifiable}.
We have $\text{\rm rank}\pare{C_{z}^{1}(w)} = n$, therefore the function $F(z, \cdot)$ is not constant. Along with the density of $\mathcal{O}_{z}^{1},$ we obtain that there exists $\epsilon > 0$ and a vector $v \in \textrm{supp } \psi$ such that $\|z - v\| = 2 \, \epsilon$. By definition of the support, it follows that every open neighborhood of $v$ has a positive measure with respect to $\pi$. Since $\R^{\n} \setminus \overline{\mathbf{B}\pare{z, \epsilon}}$ is an open neighborhood of $v$, the result of the lemma follows.  
\end{proof}

\section{Stability of the \normalized\ Markov chain $\{Z_k \, ; k \in \mathbb{N}\}$}
\label{normalized-chain-section}
%The goal of this section is to prove stability properties of the \normalized\ Markov chain associated to the step-size adaptative $\pare{ \mu/\mu_{w}, \lambda }$-ES defined in Proposition~\ref{prop:MC}, 
Assuming that $\f$ is \del{the}\new{a} strictly increasing transformation of either a $C^{1}$ scaling-invariant function with a unique global argmin or a nontrivial linear function, we prove that if Assumptions A1$-$A5 are satisfied and the expected logarithm of the step-size increases on nontrivial linear functions, then the \normalized\ Markov chain is a $\varphi$-irreducible aperiodic $T$-chain that is geometrically ergodic. In particular, it is positive and Harris recurrent.

\subsection{Irreducibility, aperiodicity and T-chain property of the \normalized\ Markov chain }
Prior to establishing Harris recurrence and positivity of the chain $\{Z_k \, ; k \in \mathbb{N}\}$, we need to establish the $\varphi$-irreducibility and aperiodicity as well as identify some small sets such that drift conditions can be used.
Since the step-size change is a deterministic function of the random input used to update the mean,
\new{we use the tools reminded in Section~\ref{deterministic-control-model} to establish these properties.}
%it looks difficult to establish those properties by hand.
%%Establishing those properties can be relatively immediate for some ES algorithms (see~\cite{auger2005convergence, auger2013linear}). Yet for the algorithms considered here, establishing $\varphi$-irreducibility turns out to be tricky because the step-size change is a deterministic function of the random input used to update the mean. 
%We hence use the tools developed by Chotard and Auger~\cite{chotard2019verifiable} and reminded in Section~\ref{deterministic-control-model}.
% using the underlying deterministic control model. 
The chain investigated satisfies $Z_{k+1} = F_{w} \pare{Z_{k}, \alpha_{\f}\pare{x^{\star} + Z_{k}, U_{k+1}}}$ 
 and therefore fits the model \eqref{nonlinear-alpha}. We prove next that the necessary assumptions needed to use the tools are satisfied if $f$ satisfies F1 or F2
\new{because}\del{This result relies on \cite[Proposition 5.2]{chotard2019verifiable} ensuring that}
if $\f$ is a continuous scaling-invariant function with Lebesgue negligible level sets, then for all $z \in \R^{\n}$, the random variable $\alpha_{\f}(x^{\star} + z, U_{1})$ admits a probability density function $p_{z}^{\f}$ such that $(z, u) \mapsto p_{z}^{\f}(u)$ is lower semi-continuous \cite[Proposition 5.2]{chotard2019verifiable}, i.e.\ B4 is satisfied.

\begin{proposition}
Let $\f$ be scaling-invariant with respect to $x^{\star}$ defined as $\varphi\circ g$ where $\varphi$ is strictly increasing and $g$ is a continuous scaling-invariant function with Lebesgue negligible level sets. Let $\{Z_k \, ; k \in \mathbb{N}\}$ be a \normalized\ Markov chain associated to the step-size adaptive $\pare{\mu/\mu_{w}, \lambda}$-ES defined as in Proposition~\ref{prop:MC} satisfying 
$
Z_{k+1} = F_{w} \pare{Z_{k}, \alpha_{\f}\pare{x^{\star} + Z_{k}, U_{k+1}}}.
$
Then model~\eqref{nonlinear-alpha} follows. In addition, if Assumption A1 is satisfied, then $F_{w}$ is $C^{1}$ and thus B5 is satisfied. If Assumption A5 is satisfied, then Assumptions B1$-$B4 are satisfied and the probability density function of the random variable $\alpha_{\f}(x^{\star} + z, U_{k+1})$ denoted by $p_{z}^{\f}$ and defined in~\eqref{mupositive} satisfies $(z, u) \mapsto p_{z}^{\f}(u)$ is lower semi-continuous.
\label{assumption-lsc}

In particular, if $\f$ satisfies F1 or F2, the assumption above on $f$ holds such that the conclusions above are valid.
\end{proposition}

\begin{proof}
 It follows from~\eqref{normalized-process} that $\{Z_k \, ; k \in \mathbb{N}\}$ is a homogeneous Markov chain following model~\eqref{nonlinear-alpha}. By~\eqref{homogeneousFunction}, $F_{w}$ is of class $C^{1}$ (B5 is satisfied) if A1 is satisfied ($\Gx: \foncfleche{ \R^{\n \mu} }{\R_{+}\backslash\acco{0}}$ is $C^{1}$).
If A5 is satisfied, then B1$-$B3 are also satisfied.

\del{With~\cite[Proposition 5.2]{chotard2019verifiable},}
For all $z \in \R^{\n},$ $\alpha_{g}(x^{\star} + z, U_{k+1})$ has a probability density function $p_{z}^{g}$ such that $(z, u) \mapsto p_{z}^{g}(u)$ is lower semi-continuous \cite[Proposition 5.2]{chotard2019verifiable}, and defined for all $z \in \R^{\n}$ and $u\in \R^{\n\mu}$ as in~\eqref{mupositive}. With Lemma~\ref{selection-function-increasing-transformation}, $\alpha_{f} = \alpha_{g}$ and then B4 holds.

A nontrivial linear function is a continuous scaling-invariant function with Lebesgue negligible level sets. Also\del{~\cite[Proposition 4.2]{scaling2021} implies that} $\f$ still has Lebesgue negligible level sets in the case where it is a $C^{1}$ scaling-invariant function with a unique global argmin \cite[Proposition 4.2]{scaling2021}.
\end{proof}

We show in the following lemma the density of \del{a}\new{the} control set\new{ in  $\R^{\n\mu}$} \del{for}\new{when the objective functions are} strictly increasing transformations of continuous scaling-invariant functions with Lebesgue negligible level sets, especially for functions $\f$ that satisfy F1 or F2. This is useful for Proposition~\ref{steadily-zero} and for the application of Lemma~\ref{rankCondition}.

\begin{lemma} Let $\f$ be a scaling-invariant function defined as $\varphi\circ g$ where $\varphi$ is strictly increasing and $g$ is a continuous scaling-invariant function with Lebesgue negligible level sets. Assume that $\{Z_k \, ; k \in \mathbb{N}\}$ is the \normalized\ Markov chain associated to a step-size adaptive $\pare{\mu/\mu_{w}, \lambda}$-ES as defined in Proposition~\ref{prop:MC} such that A5 is satisfied. Then for all $z \in \R^{\n},$ the control set $\mathcal{O}_{z}^{1} = \acco{ v \in \R^{\n\mu}  \, ; \,p_{z}^{\f}(v) > 0 }$ is dense in $\R^{\n\mu}.$

In particular, if $\f$ satisfies F1 or F2, the assumption above on $f$ holds and thus the conclusions above are valid.
\label{dense}
\end{lemma}
\begin{proof}
By Proposition~\ref{assumption-lsc}, we obtain that for all $z \in \R^{\n}$, $p_{z}^{\f}$ is defined as in~\eqref{mupositive}. 
\del{In addition, $\f$ has Lebesgue negligible level sets \cite[Proposition 4.2]{scaling2021}, \new{see} Lemma~\ref{selection-function-increasing-transformation}.} \new{In addition, $\alpha_{\f} = \alpha_{g}$ (see Lemma~\ref{selection-function-increasing-transformation}).} Therefore $p_{z}^{\f}\new{= p_{z}^{g}} > 0$ almost everywhere. Hence we have that $\mathcal{O}_{z}^{1}$ is dense in $\R^{\n\mu}$.
\end{proof}
Thanks to Theorem~\ref{condition-upstream}, to ensure that $\{Z_k \,; \,k \in \mathbb{N}\}$ is a $\varphi$-irreducible aperiodic $T$-chain, we prove that $0$ is a steadily attracting state and that there exists $w \in \mathcal{O}_{0}^{1}$ such that $\text{\rm rank}\pare{C_{0}^{1}( w )} = n $. We start with the steady attractivity in the next proposition.

\begin{proposition}
Let $\f$ be a scaling-invariant function defined as $\varphi\circ g$ where $\varphi$ is strictly increasing and $g$ is a continuous scaling-invariant function with Lebesgue negligible level sets.  Assume that $\{Z_k \, ; k \in \mathbb{N}\}$ is the \normalized\ Markov chain associated to a step-size adaptive $\pare{\mu/\mu_{w}, \lambda}$-ES as defined in Proposition~\ref{prop:MC} such that Assumptions A1 and A5 are satisfied. Then $0$ is a steadily attracting state of CM($F_{w}$).
Especially, if $\f$ satisfies F1 or F2, the assumption above on $f$ holds and thus the conclusions above are valid.
\label{steadily-zero}
\end{proposition}
\begin{proof} We fix $z \in \R^{\n}$ and prove that there exists a sequence $\acco{ z_{k} \in A_{+}^{k}(z)  \, ; k \in \mathbb{N} }$ that converges to $0.$ 
We construct the sequence recursively as follows.

We define $z_{0} = z$ and fix a natural integer $k.$ We define $z_{k+1}$ iteratively as follows.
We set $\tilde{v}_{k} = - \frac{1}{\|w\|^{2}} \pare{ w_{1}z_{k},\dots, w_{\mu}z_{k} },$ then $z_{k} + w^{\top}\tilde{v}_{k} = z_{k} - \frac{1}{\|w\|^{2} } \sum_{i=1}^{\mu} w_{i}^{2}z_{k} = 0.$ By continuity of $F_{w} $ and density of $\mathcal{O}_{z_{k}}^{1}$ thanks to Lemma~\ref{dense}, there exists $v_{k} \in \mathcal{O}_{z_{k}}^{1}$ such that $\| F_{w} (z_{k}, v_{k}) \| = \| F_{w} (z_{k}, v_{k}) - F_{w} (z_{k}, \tilde{v}_{k}) \| \leq \frac{1}{2^{k+1}}.$
 Define $z_{k+1} = F_{w} (z_{k}, v_{k}).$ Then the sequence $\pare{z_{k}}_{k\in\mathbb{N}}$ converges to $0.$ 
Now let us show that $z_{k} \in A_{+}^{k}(z)$ for all $k \in \mathbb{N}.$ 
Since $A_{+}^{0}(z) = \acco{z},$ then $z_{0} = z \in A_{+}^{0}(z).$ We fix again a natural integer $k$ and assume that $z_{k} \in A_{+}^{k}(z).$ It is then enough to prove that $z_{k+1} \in A_{+}^{k+1}(z).$
Recall that for all $\bf u \in$ $\R^{\n\mu(k+1)},$ $A_{+}^{k+1}(z) = \acco{ S_{z}^{k+1}(\text{$\bf u$})  \, ; \text{$\bf u$} \in \mathcal{O}_{z}^{k+1} }$,
$S_{z}^{k+1}(\textbf{u}) = F_{w} \pare{S_{z}^{k}\pare{u_{1}, \dots, u_{k}}, u_{k+1} }$,
$p_{z}^{\f, k+1}(\textbf{u}) = p_{z}^{\f, k}\pare{ u_{1},\dots, u_{k}} p_{S_{z}^{k}(u_{1},\dots, u_{k}) }^{\f}(u_{k+1})$,
$\mathcal{O}_{z}^{k+1} = \acco{ \bf{u}\text{ $\in \R^{\n\mu (k+1)}$ } \, ; \text{ $p_{z}^{f, k+1}(\bf{u})$ } \text{$ > 0$} }$.
Therefore by construction,
$p_{z}^{f, k+1}(v_{0}, \dots, v_{k}) = p_{z}^{\f, k}(v_{0},\dots,v_{k-1}) p_{z_{k}}^{\f}(v_{k}) > 0,$ hence $\pare{v_{0}, \dots, v_{k} }$ $\in \mathcal{O}_{z}^{k+1}.$
Finally, $z_{k+1} = F_{w} (z_{k}, v_{k}) = S_{z}^{k+1}(v_{0}, \dots, v_{k}) \in A_{+}^{k+1}(z).$
 
\end{proof}

The next proposition ensures that the steadily attracting state $0$ satisfies also the adequate full-rank condition on a controllability matrix of $0$.

\begin{proposition}
Let $\f$ be a scaling-invariant function defined as $\varphi\circ g$ where $\varphi$ is strictly increasing and $g$ is a continuous scaling-invariant function with Lebesgue negligible level sets. Assume that $\{Z_k \, ; k \in \mathbb{N}\}$ is the \normalized\ Markov chain associated to a step-size adaptive $\pare{\mu/\mu_{w}, \lambda}$-ES as defined in Proposition~\ref{prop:MC} such that Assumptions A1 and A5 are satisfied. Then there exists $w \in \mathcal{O}_{0}^{1}$ such that $\text{\rm rank}\pare{C_{0}^{1}( w )} = n$.

In particular, if $\f$ satisfies F1 or F2, the assumption above on $f$ holds and thus the conclusions above are valid.
\label{full-rank-zero}
\end{proposition}

\begin{proof}
Lemma~\ref{rankCondition} along with the density of the control set $\mathcal{O}_{0}^{1}$ in Lemma~\ref{dense} ensure that it is enough to prove the existence of $v \in \R^{\n\mu}$ such that $\text{\rm rank}\pare{C_{0}^{1}(v)} = n .$
Let us show that the matrix 
$
 C_{0}^{1}(0) =
 \frac{\partial S_{0}^{1} }{\partial v_{1} }(0) 
$
has a full rank, with $S_{0}^{1}: \foncfast{\R^{\n\mu}}{v}{F_{w} (0, v)}{\R^{\n}}$. This is equivalent to showing that the differential $ D S_{0}^{1}(0): \R^{\n\mu} \to \R^{\n}$ of $S_{0}^{1}$ at $0$ is surjective.
Denote by $l$ the linear function $\foncfast{\R^{\n\mu}}{h}{\sum_{i=1}^{\mu}w_{i} h_{i} }{\R^{\n}}$. Then $S_{0}^{1} = l / \Gx$ and then $ D S_{0}^{1}(h) = D l(h) \frac{1}{\Gx(h)} + l(h) D (\frac{1}{\Gx}) (h).$ 
Since $ l(0) = 0,$ it follows that $ D S_{0}^{1}(0) = \frac{l}{\Gx(0)}$ and finally we obtain that $ D S_{0}^{1}(0)$ is surjective.
\end{proof}

By applying Propositions~\ref{assumption-lsc},~\ref{steadily-zero} and~\ref{full-rank-zero} along with Theorem~\ref{condition-upstream}, we directly deduce that the \normalized\ Markov chain associated to a step-size adaptive $\pare{\mu/\mu_{w}, \lambda}$-ES is a $\varphi$-irreducible aperiodic T-chain. More formally, the next proposition holds.

\begin{proposition}
Let $\f$ be a scaling-invariant function defined as $\varphi\circ g$ where $\varphi$ is strictly increasing and $g$ is a continuous scaling-invariant function with Lebesgue negligible level sets. Assume that $\{Z_k \, ; k \in \mathbb{N}\}$ is the \normalized\ Markov chain associated to a step-size adaptive $\pare{\mu/\mu_{w}, \lambda}$-ES as defined in Proposition~\ref{prop:MC} such that Assumptions A1 and A5 are satisfied. Then $\{Z_k \, ; k \in \mathbb{N}\}$ is a $\varphi$-irreducible aperiodic $T$-chain, and every compact set is a small set.

In particular, if $\f$ satisfies F1 or F2, the assumption above on $f$ holds and thus the conclusions above are valid.
\label{normal-tchain}
\end{proposition}

\subsection{Convergence in distribution of the step-size multiplicative factor}

In order to prove that $\{Z_k \, ; k \in \mathbb{N}\}$ satisfies a geometric drift condition, we investigate the distribution of $\{Z_k \, ; k \in \mathbb{N}\}$ outside of a compact set (small set). Intuitively, when $Z_k$ is very large, i.e.\  $X_k - x^{\star}$ large compared to the step-size $\sigma_k$, the algorithm sees the function $f$ in a small neighborhood from $X_k - x^{\star}$ where $f$ resembles a linear function (this holds under regularity conditions on the level sets of $f$). Formally we prove that for all $\new{k \in \mathbb{N}}$, \del{$z\in \R^{\n},$}the step-size multiplicative factor $\Gx\pare{\alpha_{\f}(x^{\star} + z, U_{\new{k+}1})}$ converges in distribution\footnote{Recall that a sequence of real-valued random variables $\acco{Y_{k}}_{k \in \mathbb{N}}$ converges in distribution to a random variable $Y$ if 
$\lim_{k\to\infty} F_{Y_{k}}(x) = F_{Y}(x)$
for all continuity point $x$ of $F_{Y}$, where $F_{Y_{k}}$ and $F_{Y}$ are respectively the cumulative distribution functions of $Y_{k}$ and $Y.$

The Portmanteau lemma~\cite{billingsley1999convergence} ensures that  $\acco{Y_{k}}_{k \in \mathbb{N}}$ converges in distribution to $Y$ if and only if for all bounded and continuous function $\varphi$,
%\begin{align}
$ \lim_{k\to\infty} \mathbb{E}\croc{ \varphi(Y_{k}) } = \mathbb{E}\croc{ \varphi(Y) }$.  
%\label{Portmanteau}
%\end{align}
} 
towards the step-size change on nontrivial linear functions $\Gl$ defined in~\eqref{step-size-notation}\new{,  when $\|z\|$ goes to $\infty$}.

To do so we derive in Proposition~\ref{dynamic-linear} an intermediate result that requires to introduce a specific nontrivial linear function $l_{z}^{f}$ defined as follows.

We consider a scaling-invariant function $\f$ with respect to its unique global argmin $x^{\star}$. Then the function $\tilde{\f}: x \mapsto f(x^{\star} + x) - \f(x^{\star})$ is $C^{1}$ scaling-invariant with respect to $0$ which is the unique global argmin. There exists a vector in the closed unit ball $z_{0}^{\f} \in \overline{\mathbf{B}\pare{0, 1}}$ whose $\tilde{f}$-level set is included in the closed unit ball, that is $\level_{\tilde{\f}, z_{0}^{\f} } \subset \overline{\mathbf{B}\pare{0, 1}}$ and such that for all $z \in \level_{\tilde{\f}, z_{0}^{\f} }$, the scalar product between $z$ and the gradient of $f$ at $x^{\star} + z$ satisfies $z^{\top}\nabla \f(x^{\star} + z) > 0$ \cite[Corollary 4.1 and Proposition 4.10]{scaling2021}. In addition, any half-line of origin $0$ intersects the level set $\level_{\tilde{\f}, z_{0}^{\f} }$ at a unique point.  We denote for all $z \neq 0$ by $ t_{z}^{\f} $ the unique scalar of $(0, 1]$ such that 
$
t_{z}^{f}\frac{z}{\| z \|}$ belongs to the level set $  \mathcal{L}_{\tilde{f},z_{0}^{f}} \subset \overline{\mathbf{B}\pare{0, 1}}.
$
We finally define for all $z \neq 0$, the nontrivial linear function $l_{z}^{f}$ for all $w \in \R^{\n}$ as 
\begin{align}
l_{z}^{f}(w) = w^{\top}\, \nabla f\pare{x^{\star} + t_{z}^{f} \frac{z}{\| z \|}}.
\label{intermediate-linear}
\end{align}
We state below the intermediate result that when $\|z\|$ goes to $\infty$, the selection random vector $\alpha_{\f}(x^{\star} + z, U_{1})$ has asymptotically the distribution of the selection random vector on the linear function $l_{z}^{\f}$. According to Lemma~\ref{alpha-linearity}, the latter does not depend on the current location and is equal to the distribution of $\alpha_{ l_{z}^{\f} }(0, U_{1})$.

\begin{proposition}
Let $f$ be a $C^{1}$ scaling-invariant function with a unique global argmin. For all $\varphi: \R^{n\mu} \to \R$ continuous and bounded,
$
\lim_{\| z \| \to \infty} \int \varphi(u) \pare{p_{z}^{f}(u) - p_{z}^{ l_{z}^{f} }(u)}\mathrm{d}u  = 0
$
 where $l_{z}^{\f}$ is defined as in~\eqref{intermediate-linear}. In other words, the selection random vectors $\alpha_{\f}(x^{\star} + z, U_{1})$ and  $\alpha_{ l_{z}^{\f} }(0, U_{1})$  have asymptotically the same distribution when $\|z\|$ goes to $\infty$.
\label{dynamic-linear}
\end{proposition}
\begin{proof}[Proof idea]
We sketch the proof idea and refer to Appendix~\ref{proof-dynamic-linear} for the full proof.
Note beforehand that $\alpha_{\f}(x^{\star} + z, U_{1}) = \alpha_{\tilde{\f}}(z, U_{1})$ so that we assume without loss of generality that $x^{\star} = 0$ and $\f(0) = 0$.
If $f$ is a $C^{1}$ scaling-invariant function with a unique global argmin, we can construct\del{ thanks to~\cite[Proposition 4.11]{scaling2021}} a positive number $\delta_{\f}$ such that for all element $z$ of the compact set $\mathcal{L}_{f,z_{0}^{\f} }  + \overline{\mathbb{B}(0, 2 \delta_{f})}$, $z^{\top}\nabla \f(z) > 0$ \cite[Proposition 4.11]{scaling2021}.
In particular, this result produces a compact neighborhood of the level set $\mathcal{L}_{f,z_{0}^{\f} }$ where $\nabla f$ does not vanish.
This helps to establish the limit of $\mathbb{E}\croc{ \varphi( \alpha_{\f}(z, U_{1}) ) }$ when $\|z \|$ goes to $\infty$.
We \del{do}\new{prove} it by exploiting the uniform continuity of a function that we obtain thanks to its continuity on the compact set $\pare{\mathcal{L}_{f,z_{0}^{\f}} + \overline{\mathbb{B}(0, \delta_{f})} }\times [0, \delta_{f}]$ \cite{auger2013linear}. 
\end{proof}

Thanks to Proposition~\ref{dynamic-linear} and Proposition~\ref{linear-invariance}, we can finally state in the next theorem the convergence in distribution of the step-size multiplicative factor for $\f$ satisfying F1 towards $\Gl$ defined in~\eqref{step-size-notation}.
%%% THEOREM
\begin{theorem}
\label{scale-linear}
Let $\f$ be a scaling-invariant function satisfying F1. Assume that $\{U_{k+1} \,; k \in \mathbb{N} \}$ satisfies Assumption A5, $\Gx$ is continuous and satisfies Assumption A2, i.e.\  $\Gx$ is invariant under rotation. Then for all natural integer $k$, $\Gx\pare{\alpha_{\f}(x^{\star} + z, U_{k+1} ) }$ converges in distribution to $\Gl$ defined in~\eqref{step-size-notation}, when $\| z \| \to \infty.$
\end{theorem}

\begin{proof}
Let $\varphi: \Gamma(\R^{n\mu}) \to \R$ be continuous and bounded. It is enough to prove that $\dsp\lim_{\|z\| \to \infty}\mathbb{E}_{U_{1}\sim\N_{\n\lambda}}\croc{ \varphi\pare{\Gx\pare{\alpha_{f}(x^{\star}+z, U_{1})}} } = \mathbb{E}\croc{ \varphi\pare{\Gl} } $
 and apply the Portmanteau lemma.
By Proposition~\ref{dynamic-linear}, $\dsp\lim_{\| z \| \to \infty} \int \varphi\pare{ \Gx(u)} \, \pare{p_{z}^{f}(u) - p_{z}^{ l_{z}^{\f} }(u)}\mathrm{d}u = 0$. Then $\dsp\lim_{\| z \| \to \infty} \mathbb{E}_{U_{1}\sim\N_{\n\lambda}}\croc{ \varphi\pare{\Gx\pare{\alpha_{f}(x^{\star}+z, U_{1})}} } - \mathbb{E}_{U_{1}\sim\N_{\n\lambda}}\croc{ \varphi\pare{\Gx\pare{\alpha_{ l_{z}^{f} }(x^{\star}+z, U_{1}) }}} = 0$.
With Proposition~\ref{linear-invariance}, $\mathbb{E}_{U_{1}\sim\N_{\n\lambda}}\croc{ \varphi\pare{\Gx\pare{\alpha_{ l_{z}^{f} }(x^{\star}+z, U_{1}) }}} = \mathbb{E}\croc{ \varphi\pare{\Gl} }$.
\end{proof}

\subsection{Geometric ergodicity of the \normalized\ Markov chain}
\label{geometric-dirft-condition-for-main}

The convergence in distribution of the step-size multiplicative factor while optimizing a function $\f$ that satisfies F1 proven in Theorem~\ref{scale-linear}  allows us to control the behavior of the \normalized\ chain when its norm goes to $\infty$.
More specifically, we use it to show the geometric ergodicity of $\{Z_k \, ; k \in \mathbb{N}\}$ defined as in Proposition~\ref{prop:MC} for $\f$ satisfying F1 or F2. Beforehand, let us show the following proposition, which is a first step towards the construction of a geometric drift function.
\begin{proposition} 
Let $\f$ be a scaling-invariant function that satisfies F1 or F2 and $\{Z_k \, ; k \in \mathbb{N}\}$ be the \normalized\ Markov chain associated to a step-size adaptive $\pare{\mu/\mu_{w}, \lambda}$-ES defined as in Proposition~\ref{prop:MC}. We assume that $\Gx$ is continuous and Assumptions A2, A3 and A5 are satisfied. Then for all $\alpha > 0,$
$
\lim_{\| z \|\to\infty }\frac{\mathbb{E}\croc{ \|Z_{1}\|^{\alpha} | Z_{0} = z} }{ \| z \|^{\alpha} } = \mathbb{E}\croc{  \frac{1}{[\Gl]^{\alpha}} }
$
where $\Gl$ is the random variable defined in~\eqref{step-size-notation} that represents the step-size change on any nontrivial linear function. 
\label{geometric-asymptotic}
\end{proposition}

\begin{proof}
Let $z \neq 0$. Since $Z_{1} = F_{w} \pare{Z_{0}, \alpha_{\f}\pare{x^{\star}+Z_{0}, U_{1}}} = \frac{Z_{0} +  w^{\top} \alpha_{\f}\pare{x^{\star}+Z_{0}, U_{1}}   }{ \Gx \pare{\alpha_{\f}\pare{x^{\star}+Z_{0}, U_{1}} } },$ then
$
\mathbb{E}\croc{ \| Z_{1} \|^{\alpha} | Z_{0} = z} / \| z \|^{\alpha} - \mathbb{E}\croc{1 / \Gx \pare{ \alpha_{f}(x^{\star}+z, U_{1}) }^{\alpha}  } =  \mathbb{E}\croc{ \frac{ \Big\| \frac{z}{\| z \|} + \frac{w^{\top} \alpha_{f}(x^{\star}+z, U_{1})}{\| z \|}  \Big\|^{\alpha} - 1} {\Gx \pare{ \alpha_{f}(x^{\star}+z, U_{1}) }^{\alpha} }  }.
$
The function $\Gx$ is lower bounded by $m_{\Gx} > 0$ thanks to  Assumption A3. In addition,  $\| w^{\top} \alpha_{\f}\pare{x^{\star}+z, U_{1}} \| \leq \| w \| \, \| U_{1} \| .$ Then the term
\begin{equation}
 \left\lvert \left\lVert \frac{z}{\| z \|} + \frac{1}{\| z \|} w^{\top} \alpha_{f}(x^{\star}+z, U_{1}) \right\lVert^{\alpha} - 1 \right \rvert / \Gx \pare{ \alpha_{f}(x^{\star}+z, U_{1}) }^{\alpha} 
\end{equation}
converges almost surely towards 0 when $\|z\|$ goes to $\infty,$ and is bounded (when $\| z \| \geq 1$) by the integrable random variable 
$\frac{ 1 +  \pare{ 1+ \| w \| \,\| U_{1} \|  }^{\alpha} }{m_{\Gx}^{\alpha}}.$ Then it follows by the dominated convergence theorem that

\begin{equation}\label{eq:inproof}
\lim_{\| z \| \to \infty} \mathbb{E}\croc{\|  Z_{1} \|^{\alpha} | Z_{0} = z} / \| z \|^{\alpha} - \mathbb{E}\croc{ 1 / \Gx \pare{ \alpha_{f}(x^{\star}+z, U_{1}) }^{\alpha}  } = 0.
\end{equation}

Since $\dsp x \mapsto 1/x^{\alpha}$ is continuous and bounded on $\Gamma\pare{\R^{n\mu}} \subset [m_{\Gamma}, \infty$), then for $f$ satisfying F1, Theorem \ref{scale-linear} implies that $\dsp\lim_{\| z \| \to \infty}\mathbb{E}\croc{1 / \Gx \pare{ \alpha_{f}(x^{\star}+z, U_{1}) }^{\alpha}   }$ exists and is equal to $\mathbb{E}\croc{ 1 / [\Gl]^\alpha }$. Starting from \eqref{eq:inproof} and using Proposition~\ref{linear-invariance} to replace $\mathbb{E}\croc{\frac{1}{\Gx \pare{ \alpha_{f}(x^{\star}+z, U_{1}) }^{\alpha} }}  $ by $\mathbb{E}\croc{\frac{1}{[\Gl]^\alpha}}$, the same conclusion holds for $\f$ satisfying F2. 
Thereby 
$ \lim_{\| z \| \to \infty} \mathbb{E}\croc{\| Z_{1} \|^{\alpha} | Z_{0} = z} / \| z \|^{\alpha} = \mathbb{E}\croc{  \frac{1}{[\Gl]^{\alpha}} } .$  
\end{proof}

We introduce the next two lemmas, that allow to go from Proposition~\ref{geometric-asymptotic} to a formulation with the multiplicative $\log$-step-size factor.

\begin{lemma} Let $\f$ be a continuous scaling-invariant function with respect to $x^{\star}$ with Lebesgue negligible level sets, let $z \in \R^{\n}.$  Assume that $\Gx$ satisfies Assumption~A4. Then
$u \mapsto \log\pare{ \Gx \pare{ \alpha_{\f}(x^{\star}+z, u)} }$ is $\N_{\n\lambda}$-integrable with 
\begin{align}
 \mathbb{E}_{U_{1}\sim\N_{\n\lambda}}\croc{ \abs{ \log\pare{ \Gx \pare{ \alpha_{\f}(x^{\star}+z, U_{1})} } } } \leq  \frac{\lambda ! \, \, \mathbb{E}_{W_{1}\sim\N_{\n\mu}}\croc{ \abs{ \log\circ \,\Gx}\pare{W_{1}} }}{(\lambda - \mu)!} .
\label{log-gamma-integrable}
\end{align}
\label{log-gamma-lemma} 
\end{lemma}
\begin{proof}
With~\eqref{mupositive}, we have $\frac{(\lambda - \mu)!}{\lambda !} \mathbb{E}\croc{ \abs{ \log\pare{ \Gx \pare{ \alpha_{\f}(x^{\star}+z, U_{1})} }} }$ $\leq$ $\dsp\int_{\R^{\n}} \abs{ \log\circ \,\Gx}(v)$ 
$\prod_{i=1}^{\mu} p_{\N_{\n}}(v^{i}) \diff v =  \mathbb{E}\croc{ \abs{ \log\circ \,\Gx}\pare{ \mathcal{N}_{\n\mu} } },$
and A4 says that $\log\circ \,\Gx$ is $\N_{\n\mu}$-integrable.  
\end{proof}

The next lemma states that if the expected logarithm of the step-size change is positive, then we can find $\alpha > 0$ such that the limit in Proposition~\ref{geometric-asymptotic} is strictly smaller than $1$. This is the key lemma to have the condition in the main results expressed as $\mathbb{E}\croc{ \log\pare{\Gl} } > 0$, instead of $\mathbb{E}\croc{1 / [\Gl]^\alpha}  < 1$ for a positive $\alpha$~\cite{auger2013linear}.
\begin{lemma}  Assume that $\Gx$ satisfies Assumptions A3 and A4. If $\mathbb{E}\croc{ \log\pare{\Gl} } > 0$, then there exists $0 < \alpha < 1$ such that $\mathbb{E}\croc{  \frac{1}{[\Gl]^\alpha} }  < 1$, where $\Gl$ is defined in~\eqref{step-size-notation}.
\label{lyapunov-coefficient}
\end{lemma}

\begin{proof}
Lemma \ref{log-gamma-lemma} ensures that $\log\pare{\Gl}$ is integrable.
For $\alpha > 0,$
$
 \frac{1}{[\Gl]^\alpha } = \exp\croc{ -\alpha\log\pare{\Gl } } = 1 - \alpha \log\pare{\Gl } + o(\alpha).
$
 Then the random variable $ A(\alpha) = \left(\frac{1}{ [\Gl]^\alpha } - 1 + \alpha \log\pare{\Gl } \right) / \alpha$ depending on the parameter $\alpha$ converges almost surely towards $0$ when $\alpha$ goes to $0$.

Let $u \in \R^{\n\mu}$ and $\alpha \in (0, 1)$. Define $\varphi_{u}: c \mapsto \dsp\frac{1}{ \Gx(u)^{c} } = \exp( -c \log(\Gx(u)) )$ on $[0, \alpha]$. By the mean value theorem, there exists $c_{u,\alpha}\in (0, \alpha)$ such that $\pare{\frac{1}{ \Gx(u)^\alpha } - 1} / \alpha = \varphi_{u}^{\prime}(c_{u,\alpha}) = - \log(\Gx(u)) \frac{1}{ \Gx(u)^{c_{u,\alpha}} }$. In addition, $\frac{1}{ \Gx(u)^{c_{u,\alpha}} } \leq \frac{1}{m_{\Gx}^{c_{u,\alpha}}}$ thanks to Assumption A3, and $ \frac{1}{m_{\Gx}^{c_{u,\alpha}}} = \exp\pare{  - c_{u,\alpha} \log(m_{\Gx}) } \leq \exp\pare{ \abs{\log(m_{\Gx})} }$. Therefore $\abs{A(\alpha)} \leq \pare{1 + \exp\pare{ \abs{\log(m_{\Gx})} }} \abs{\log\pare{\Gl }}$. The latter is integrable thanks to Assumption A4, and does not depend on $\alpha$. Then by the dominated convergence theorem, $\mathbb{E}\croc{A(\alpha)}$ converges to $0$ when $\alpha$ goes to $0$ or equivalently
$
\mathbb{E}\croc{ \frac{1}{[\Gl]^\alpha } } = 1 - \alpha \mathbb{E}\croc{ \log\pare{\Gl} } + o(\alpha).
 $
  Hence there exists $0 < \alpha < 1$ small enough such that $\mathbb{E}\croc{ \frac{1}{[\Gl]^\alpha } }  < 1.$   
\end{proof}

We now have enough material to state and prove the desired geometric ergodicity of the \normalized\ Markov chain in the following theorem.
\begin{theorem} (Geometric ergodicity)
 \label{geometric-drift}
Let $\f$ be a scaling-invariant function that satisfies F1 or F2. Let $\{Z_k \, ; k \in \mathbb{N}\}$ be the \normalized\ Markov chain associated to a step-size adaptive $\pare{\mu/\mu_{w}, \lambda}$-ES defined as in Proposition~\ref{prop:MC} such that Assumptions A1$-$A5 are satisfied.
Assume that $\mathbb{E}\croc{ \log\pare{\Gl} } > 0$ where $\Gl$ is defined in~\eqref{step-size-notation}.

Then there exists $0 < \alpha < 1 $ such that the function $V: z \mapsto 1+\| z \|^{\alpha}$ is a geometric drift  function for the Markov chain $\{Z_k \, ; k \in \mathbb{N}\}$.  Therefore $\{Z_k \, ; k \in \mathbb{N}\}$ is $V$-geometrically ergodic, admits an invariant probability measure $\pi$ and is Harris recurrent.

\end{theorem}

\begin{proof}
Proposition~\ref{normal-tchain} shows that $\{Z_k \, ; k \in \mathbb{N}\}$ is a $\varphi$-irreducible aperiodic T-chain. With~\cite[Theorem 5.5.7 and Theorem 6.2.5]{meyn2012markov}, every compact set is a small set.
Since $\mathbb{E}\croc{ \log\pare{\Gl} } > 0$, by Lemma \ref{lyapunov-coefficient} there exists $0 < \alpha < 1$ such that $\mathbb{E}\croc{  \frac{1}{[\Gl]^\alpha} } < 1$. Define $V: z \mapsto 1+\| z \|^{\alpha}$.
By Proposition~\ref{geometric-asymptotic},
$\dsp\lim_{\| z \|\to\infty } \mathbb{E}\croc{\| Z_{1} \|^{\alpha} | Z_{0} = z} / \| z \|^{\alpha}  = \mathbb{E}\croc{ 1 / [\Gl]^\alpha }$. 
Since $ {\mathbb{E}\croc{V(Z_{1}) | Z_{0} = z}}/{V(z)} =\pare{1+\mathbb{E}\croc{\| Z_{1} \|^{\alpha} | Z_{0} = z} } / \pare{ 1+\| z \|^{\alpha} }$, $\lim_{\| z \|\to\infty } \mathbb{E}\croc{V(Z_{1}) | Z_{0} = z} / V(z) = \mathbb{E}\croc{ 1 / [\Gl]^\alpha }.$
Let $\gamma = \frac12 \left(  1 +  \mathbb{E}\croc{  \frac{1}{[\Gl]^\alpha} } \right) < 1.$
There exists $r > 0$ such that for all $\| z \| > r$
 \begin{align}
  \mathbb{E}\croc{V(Z_{1}) | Z_{0} = z} / V(z) < \gamma.
 \label{rate}
 \end{align}
 In addition, since $\| z + w^{\top}\alpha_{\f}\pare{x^{\star}+z, U_{1}} \| \leq \| z \| +  \| w \| \| U_{1} \|$ then $\mathbb{E}\croc{V(Z_{1}) | Z_{0} = z}$ $\leq$ \,
 $\dsp{ \new{1 + } \mathbb{E}\croc{\pare{ \| z \| + \| w \|  \| U_{1} \|  }^{\alpha} }}/{m_{\Gx}^{\alpha}}.$ Since $z \mapsto \dsp \nnew{1+}\mathbb{E}\croc{\pare{ \| z \| + \| w \| \| U_{1} \|  }^{\alpha} } / m_{\Gx}^{\alpha} - \gamma V(z)$ is continuous on the compact $\overline{\mathbf{B}\pare{0, r}},$ it is bounded on that compact. Denote by $b\in \R_{+}$ an upper bound. We have proven that for all $z \in \overline{\mathbf{B}\pare{0, r}},$ $\mathbb{E}\croc{V(Z_{1}) | Z_{0} = z} \leq \gamma V(z) + b.$ This result, along with~\eqref{rate}, show that for all $z\in \R^{\n},$
$\mathbb{E}\croc{ V(Z_{1}) | Z_{0} = z} \leq \gamma V(z) + b \mathds{1}_{ \overline{\mathbf{B}\pare{0, r}} }(z)$. Therefore $\{Z_k \, ; k \in \mathbb{N}\}$ is $V$-geometrically ergodic. Then thanks to~\cite[Theorem 15.0.1]{meyn2012markov}, $\{Z_k ; k \in \mathbb{N}\}$ is positive and Harris recurrent with invariant probability measure $\pi$.
 \end{proof}

\section{\del{Proofs of the main results}\new{Main results: linear behavior as a consequence of the stability and integrability}}
\label{consequences-section}
We \new{are now almost ready to establish the main results of the paper.}\del{ present in this section the proofs of Theorems~\ref{theo:main-CV},~\ref{clt-cbsaes} and Proposition~\ref{prop:weight-condition} presented 
%the main results stated 
in Section~\ref{main-results} the main results of the paper. }
%namely Theorems~\ref{theo:main-CV},~\ref{clt-cbsaes} and Proposition~\ref{prop:weight-condition}. 
\new{Yet, we first prove in the next section}
\del{Before establishing those main results, we  prove}the integrability of $z \mapsto \log \|z\|$ and $\mathcal{R}_{\f}$ defined in~\eqref{expected-step-size-f}, with respect to the invariant probability measure of the Markov chain $\{Z_k \, ; k \in \mathbb{N}\}$ whose existence is proven in Theorem~\ref{geometric-drift}.
\new{We state and prove in Section~\ref{sec:linear-behavior} the linear behavior of the studied class of algorithms for an abstract step-size update satisfying A1-A4 on scaling invariant functions. We provide in Section~\ref{CLT} a Central Limit Theorem for approximating the convergence rate.
 We investigate in Section~\ref{sec:SC-linear-behavior-CSA-xNES} how the CSA-ES and xNES satisfy the required conditions for a linear behavior providing sufficient conditions expressed in terms of parameters of the algorithms.}

\subsection{Integrabilities with respect to the invariant probability measure}
\label{consequences-lin-conv}

For a scaling-invariant function $\f$ that satisfies F1 or F2, the limit in Theorem~\ref{theo:main-CV} is expressed as $\mathbb{E}_{\pi}(\mathcal{R}_{\f})$ where the function $\mathcal{R}_{\f}$ is defined as in~\eqref{expected-step-size-f} and $\pi$ is a probability measure. Therefore the $\pi$-integrability of the function $z \mapsto \mathcal{R}_{\f}(z)$ is necessary to obtain Theorem~\ref{theo:main-CV}. In the following, we present a result stronger than its $\pi$-integrability, that is the boundedness of $\mathcal{R}_{\f}$ under some assumptions.

\begin{proposition}
Let $\f$ be a continuous scaling-invariant function with Lebesgue negligible level sets. Let $\{(X_k,\sigma_k) \, ; k \in \mathbb{N} \}$ be the  sequence defined in~\eqref{incumbent} and~\eqref{step-size} such that Assumptions A4 and A5 are satisfied. Then the function $z \mapsto \abs{  \mathcal{R}_{\f}}(z)$ is bounded by $\frac{\lambda !}{(\lambda - \mu)!} \mathbb{E}_{W\sim\N_{\n\mu}}\croc{ \abs{ \log\circ \,\Gx}(W) }$, where the function $z \mapsto \mathcal{R}_{\f}(z)$ is defined as in~\eqref{expected-step-size-f}.

If in addition \new{the following holds: (i)} $\f$ satisfies F1 or F2, \new{(ii)}  Assumptions A1$-$A3 are satisfied and \new{(iii)} the expected log step-size change satisfies $\mathbb{E}\croc{ \log\pare{\Gl} } > 0$ where $\Gl$ is defined in~\eqref{step-size-notation},
then 
$
 \mathbb{E}_{\pi}(\abs{  \mathcal{R}_{\f} }) = \int \abs{\mathcal{R}_{\f}(z)} \pi(dz) <  \infty
$
that is $z \mapsto \mathcal{R}_{\f}(z)$ is $\pi$-integrable where $\pi$ is the invariant probability measure of $\{Z_k \, ; k \in \mathbb{N}\}$ defined as in Proposition~\ref{prop:MC}.
\label{R-integrable}
\end{proposition}

\begin{proof}
 Lemma~\ref{log-gamma-lemma} shows that for all $z \in \R^{\n},$ $z \mapsto \log\pare{ \Gx \pare{ \alpha_{f}(x^{\star}+z, u )} }$ is $\N_{\n\lambda}$-integrable with 
$\dsp \mathbb{E}\croc{ \abs{ \log\pare{ \Gx \pare{ \alpha_{f}(x^{\star}+z, U_{1})} } } } \leq  \frac{\lambda !}{(\lambda - \mu)!} \mathbb{E}_{W\sim\N_{\n\mu}}\croc{ \abs{ \log\circ \,\Gx}(W) }.$
Then $\abs{\mathcal{R}_{\f}}$ is bounded since $\abs{\mathcal{R}_{\f}(z)} \leq  \frac{\lambda !}{(\lambda - \mu)!} \mathbb{E}_{W\sim\N_{\n\mu}}\croc{ \abs{ \log\circ \,\Gx}(W) }$ for all $z \in \R^{\n}$.
If in addition Assumptions A1$-$A3 are satisfied and $\mathbb{E}\croc{ \log\pare{\Gl} } > 0$, Theorem~\ref{geometric-drift} ensures that $\{Z_k \, ; k \in \mathbb{N}\}$ is a positive Harris recurrent chain with invariant probability measure $\pi.$
Hence the integrability with respect to $\pi.$  
\end{proof}

We prove in the following the $\pi$-integrability of $z \mapsto \log \| z \|$, where $\pi$ is the invariant probability measure of the \normalized\ chain, under some assumptions.

\begin{proposition}
Let $\f$ satisfy F1 or F2 and $\{Z_k \, ; k \in \mathbb{N}\}$ be the Markov chain defined as in Proposition~\ref{prop:MC} such that Assumptions A1$-$A5 are satisfied. 
Assume that $\mathbb{E}\croc{ \log\pare{\Gl} } > 0$ where $\Gl$ is defined in~\eqref{step-size-notation}. Then $\{Z_k \, ; k \in \mathbb{N}\}$ has an invariant probability measure $\pi$ and $z \mapsto \log{\| z \|}$ is $\pi$-integrable.
\label{log-integrable}
\end{proposition}
\begin{proof}
Theorem~\ref{geometric-drift} ensures that $\{Z_k \, ; k \in \mathbb{N}\}$ is $V$-geometrically ergodic with invariant probability measure $\pi$, where $V:\foncfast{\R^{\n}}{z}{ 1+\| z \|^{\alpha} }{\R_{+}}$. We define for all $z \in \R^{\n},$ 
$g(z) =  \frac{ (\lambda - \mu)!}{2 \lambda !} \abs{\log \| z \|}.$
The $\pi$-integrability of $g$ is 
obtained if there exist a set $A$ with $\pi(A) > 0$ such that $\dsp\int_{A} g(z)\pi(\mathrm{d}z) < \infty,$ and a measurable function $h$ with $h\mathds{1}_{A^{c}} \geq g\mathds{1}_{A^{c}}$ such that (i) $\dsp\int_{A^{c}} P(z, \diff y) h(y) < h(z) - g(z) \, , \forall z \in A^{c}$ and (ii) $\dsp \sup_{z \in A} \int_{A^{c}} P(z, \diff y) h(y) < \infty$ \cite[Theorem 1]{tweedie1983existence}.
\new{For $z \in  \R^{\n}$} and $v \in \R^{\n\mu}$, denote $\varphi(z, v)$ as
$\varphi(z, v) = p_{\N_{\n\mu}} \pare{ v - \frac{1}{\|w\|^{2}} \pare{ w_{1}z,\dots, w_{\mu}z} } \mathds{1}_{ \| w^{\top} v \| \leq 1 }  \enspace$.
We prove in a first time that 
$\dsp\lim_{ \| z \|\to 0 } \dsp \int \abs{ \log\dsp \| w^{\top} v\| } \varphi(z, v) \diff v < \infty$.
We have $(2\pi)^{n\mu/2}  \varphi(z, v)$ which is equal to $\exp\left( \frac12\left( -\|v\|^{2} -  \frac{\| w\|^{2}\|z\|^{2} }{ \|w\|^{4}}  + \frac{2 (w^{\top}v)^{\top}z}{ \|w\|^{2}}  \right) \right) \mathds{1}_{ \| w^{\top} v \| \leq 1 }$ which is smaller than $\exp \left( \frac12 \left(-\|v\|^{2} -  \frac{\| w\|^{2}\|z\|^{2}}{ \|w\|^{4} } + \frac{2 \|w^{\top}v \| \|z\|}{ \|w\|^{2}}  \right) \right) \mathds{1}_{ \| w^{\top} v \| \leq 1 }$. \new{Then for $z \in  \R^{\n}$ and $v \in \R^{\n\mu}$, 
\begin{align}
(2\pi)^{n\mu/2}  \varphi(z, v) \leq   \exp\left( \frac12 \left( -\|v\|^{2} - \frac{\| w\|^{2}\|z\|^{2}}{ \|w\|^{4}}  + \frac{2 \|z\|}{  \|w\|^{2}}  \right) \right) \mathds{1}_{ \| w^{\top} v \| \leq 1 }.
\label{domination-varprhi}
\end{align}
Therefore for $z \in  \overline{\mathbf{B}\pare{0, 1} }$ and $v \in \R^{\n\mu}$, $\varphi(z, v) \leq \exp\left(\frac{1}{\|w\|^{2}}  \right) \varphi(0, v)$.}
\del{ which is smaller than
%(2\pi)^{n\mu/2}  \varphi(z, v) &= 
%&\exp\left( \frac12\left( -\|v\|^{2} -  \frac{\| w\|^{2}\|z\|^{2} }{ \|w\|^{4}}  + \frac{2 (w^{\top}v)^{\top}z}{ \|w\|^{2}}  \right) \right) \mathds{1}_{ \| w^{\top} v \| \leq 1 }  \nonumber\\
 \begin{multline}
  \exp\left( \frac12 \left( -\|v\|^{2} - \frac{\| w\|^{2}\|z\|^{2}}{ \|w\|^{4}}  + \frac{2 \|z\|}{  \|w\|^{2}}  \right) \right) \mathds{1}_{ \| w^{\top} v \| \leq 1 } \\
 \leq (2\pi)^{\frac{n\mu}{2}} \exp\left(\frac{1}{\|w\|^{2}}  \right) \varphi(0, v). \label{domination-varprhi}
 \end{multline}
 }

Since $v \mapsto \abs{ \log\dsp \| w^{\top} v\| } \varphi(0, v)$ is Lebesgue integrable, it follows by the dominated convergence theorem that $z \mapsto \dsp\int  \abs{ \log\dsp \| w^{\top} v\| } \varphi(z, v) \diff v$ is continuous on $\overline{\mathbf{B}\pare{0, 1} }$ and $\dsp\lim_{ \| z \|\to 0 } \dsp \int \abs{ \log\dsp \| w^{\top} v\| } \varphi(z, v) \diff v < \infty$.
In addition, $\dsp\lim_{ \| z \|\to 0 } g(z) = \infty$. Then there exists $\epsilon_{1} \in \pare{0, 1}$ such that for $z \in  \overline{\mathbf{B}\pare{0, \epsilon_{1}}}:$
\begin{align}
 \int  \abs{\log\dsp \| w^{\top} v \|} \varphi(z, v) \diff v \, + 2\, \mathbb{E}_{W\sim\N_{\n\mu}}\croc{ \abs{\log\circ\Gx}(W)} \leq g(z).
\label{majoration}
\end{align}
We define $\epsilon_{2}$ from Lemma~\ref{non-negligible} and denote $\epsilon = \min(\epsilon_{1}, \epsilon_{2}).$
Define $A = \R^{\n} \setminus \overline{\mathbf{B}\pare{0, \epsilon}}.$ Then from Lemma \ref{non-negligible} it follows that $\pi(A) > 0.$ Note also that $A^c = \overline{\mathbf{B}\pare{0, \epsilon}}$. In addition, $g$ is dominated by the $\pi$-integrable function V around $\infty,$ then $\dsp\int_{A} g(z)\pi(\mathrm{d}z) < \infty.$
We define now the function $h$ for all $z \in \R^{\n}$ as $h(z) = 2 g(z) \mathds{1}_{A^{c}}(z).$ Then $h\mathds{1}_{A^{c}} \geq g\mathds{1}_{A^{c}}$.
 It remains to verify the items \del{$1$}\new{(i)} and \del{$2$}\new{(ii)} from above to obtain the $\pi$-integrability of $g$. We give in the following an upper bound of $K(z) = \dsp\int_{A^{c}} P(z, \diff y) h(y) = \dsp-\frac{ (\lambda - \mu)!}{\lambda !} \mathbb{E}_{z}\croc{\dsp\mathds{1}_{  \overline{\mathbf{B}\pare{0, \epsilon}} }(Z_{1}) \log\| Z_{1} \| }$. We have
$K(z) \leq \dsp - \frac{ (\lambda - \mu)!}{\lambda !}  \int_{ \|  z + w^{\top} v \| \leq \Gx(v) } \log\dsp\frac{ \|  z + w^{\top} v\|  }{ \Gx(v)} p_{z}^{f}(v)\diff v.$ With~\eqref{mupositive}, $\dsp \frac{(\lambda - \mu)!}{\lambda !} p_{z}^{\f} \leq p_{\N_{\n\mu}} $. Then
$\dsp K(z) \leq \dsp\int \abs{\log(\Gx(v))} p_{\N_{\n\mu}}(v) \diff v + \dsp\int_{ \| z + w^{\top} v\|  \leq \Gx(v) } \abs{\log\| z + w^{\top} v \|  } p_{\N_{\n\mu}}(v)\diff v.$
We split the latter integral between the events $\acco{\| z + w^{\top} v \| \leq \min(1, \Gx(v)) }$ and the events $\acco{1 < \| z + w^{\top} v \| \leq \Gx(v)}$. Then
$K(z) \leq \dsp \int_{ \| z + w^{\top} v \| \leq \min(1,\Gx(v)) } \abs{\log\dsp\| z + w^{\top} v \| }$ \\ $p_{\N_{\n\mu}}(v) \diff v \,+ \dsp \int_{\Gx(v) \geq 1} \log(\Gx(v)) p_{\N_{\n\mu}}(v) \diff v \,+ \dsp \int \abs{\log(\Gx(v))} p_{\N_{\n\mu}}(v) \diff v.$ Hence
 $\dsp K(z) \leq 2\, \mathbb{E}_{W\sim\N_{\n\mu}}\croc{ \abs{\log\circ\Gx}(W)}
- \dsp  \int_{ \| z + w^{\top} v \| \leq 1 } \log\dsp\| z + w^{\top} v \|p_{\N_{\n\mu}}(v) \diff v .$
With a translation $v \to v - \frac{1}{\|w\|^{2}} \pare{ w_{1}z,\dots, w_{\mu}z}$ within the last integrand, we obtain:
\begin{align}
K(z)  \leq 2 \, \mathbb{E}_{W\sim\N_{\n\mu}}\croc{ \abs{\log\circ\Gx}(W)} +  \dsp\int \abs{\log\dsp\| w^{\top} v \|} \varphi(z, v) \diff v. \label{simidrift}
\end{align}

Equations \eqref{majoration} and \eqref{simidrift} show that for $z \in  A^{c} = \overline{\mathbf{B}\pare{0, \epsilon}}$, 
$\dsp\int_{A^{c}} P(z, \diff y) h(y)  \leq g(z) = h(z) - g(z)$. Therefore the item \del{$1$}\new{(i)} follows.
With~\eqref{domination-varprhi}, it follows that there exist $c_{1} > 0$ and $c_{2} > 0$ such that for $\| z\| \geq c_{1}$ and $v \in \R^{\n}$, $\varphi(z, v) \leq c_{2} \varphi(0, v)$. Thanks to the dominated convergence theorem, 
$\dsp\lim_{\| z \| \mapsto \infty} \dsp\int \abs{\log\dsp\| w^{\top} v \|} \varphi(z, v) \diff v = 0$. Therefore that integral is bounded outside of a compact. In addition, $z \mapsto \dsp\int \abs{\log\dsp\| w^{\top} v \|} \varphi(z, v) \diff v$ is continuous and is bounded on any compact included in $\overline{A}$. Then along with~\eqref{simidrift} it follows that $\dsp \sup_{z \in A} \int_{A^{c}} P(z, \diff y) h(y) < \infty$. Hence the item \del{$2$}\new{(ii)} is also satisfied, which ends the integrability proof of $z \mapsto \log \| z \|$.
\end{proof}

\subsection{Linear behavior \new{for an abstract step-size update}}
\label{sec:linear-behavior}
%Proof of Theorem~\ref{almostSureAndLogProgress}}
%We start by proving~\eqref{almost-sure-linear-convergence-main}.
\new{We are now ready to establish the linear behavior of the \muslashmu-ES.}
Our condition for the linear behavior \new{stemming from the drift condition for geometric ergodicity established in Theorem~\ref{geometric-drift}} is that the expected logarithm of the step-size change function $\Gx$ on a nontrivial linear function is positive. 
%\del{More precisely, let us denote the expected change of the logarithm of the step-size for any state $z \in \R^{\n}$ of the \normalized\ chain as
%\begin{align}
%\mathcal{R}_{\f}(z) = \mathbb{E}_{U_{1}\sim\N_{\n\lambda}}\croc{ \log\pare{\Gx \pare{ \alpha_{f}(x^{\star} + z, U_{1})} } } \enspace.
%\label{expected-step-size-f}
%\end{align}
%}
By Proposition~\ref{linear-invariance}, when $\f$ satisfies F2, the expected change of the logarithm of the step-size is constant and for all $z$, 
$
\mathcal{R}_{\f}(z) = \mathcal{R}_{\f}(-x^{\star}) =  \mathbb{E}\croc{\log\pare{\Gl}}
$ where $\Gl$ is defined in~\eqref{step-size-notation}. Our main result states that if the expected logarithm of the step-size increases on nontrivial linear functions, i.e.\ if $\mathbb{E}\croc{\log\pare{\Gl}} > 0$, then almost sure linear behavior holds on functions satisfying F1 or F2. If $f$ satisfies F2, then almost sure linear divergence holds with a divergence rate of $\mathbb{E}\croc{\log\pare{\Gl} }$.
\new{More precisely the following results hold.}
\begin{theorem}
\label{theo:main-CV} 
Let $\f$ be a scaling-invariant function with respect to $x^{\star}$. Assume that $\f$ satisfies F1 (in which case $x^{\star}$ is the global optimum) or F2. Let $\{(X_k,\sigma_k) \, ; k \in \mathbb{N}  \}$ be the  sequence defined in~\eqref{incumbent} and~\eqref{step-size} such that Assumptions A1$-$A5 are satisfied. Let $\left\{Z_k = (X_k - x^{\star})/\sigma_k \, ; k \in \mathbb{N} \right\}$ be the \normalized\ Markov chain (Proposition~\ref{prop:MC}).
If the expected logarithm of the step-size increases on nontrivial linear functions, i.e.\  if $\mathbb{E}\croc{\log\pare{\Gl}} > 0$  where $\Gl$ is defined in~\eqref{step-size-notation}, then $\{ Z_k \, ; k \in \mathbb{N}\}$ admits an invariant probability measure $\pi$ such that $\mathcal{R}_{\f}$ defined in~\eqref{expected-step-size-f} is $\pi$-integrable. And for all $\pare{X_{0}, \sigma_{0}} \in \left(\R^{\n}\setminus\acco{x^{\star}}\right) \times \pare{0,\infty},$ linear behavior of $X_k$ and $\sigma_k$ as in \eqref{eq:linear-behavior} holds almost surely with
\begin{align}
\lim_{k\to\infty} \frac1k \log\frac{ \| X_{k} - x^{\star}\| }{ \| X_{0} - x^{\star} \| } = \lim_{k\to\infty} \frac1k \log\frac{ \sigma_{k} }{ \sigma_{0} } =   \mathbb{E}_{\pi}(\mathcal{R}_{\f})
\enspace.
\label{almost-sure-linear-convergence-main}
\end{align}
In addition, for all initial conditions $\pare{X_{0}, \sigma_{0} } = (x, \sigma)  \in \R^{\n} \times \pare{0,\infty},$ \del{we have linear behavior of} the expected log-progress \new{behaves linearly} with 
\begin{align}
\lim_{k\to\infty} \mathbb{E}_{\frac{x-x^{\star}}{\sigma}}\croc{ \log\frac{ \| X_{k+1} - x^{\star}\|   }{ \| X_{k} - x^{\star} \|   } } = \lim_{k\to\infty} \mathbb{E}_{\frac{x-x^{\star}}{\sigma}} \croc{  \log\frac{ \sigma_{k+1} }{ \sigma_{k} }}  = \mathbb{E}_{\pi}(\mathcal{R}_{\f}) \enspace .
\label{expected-log-progress-main}
\end{align}

If $\f$ satisfies F2, then $\mathcal{R}_{\f} $ is constant equal to $\mathbb{E}_{\pi}(\mathcal{R}_{\f}) = \mathbb{E}\croc{\log\pare{\Gl}} > 0$, and then both $X_k$ and $\sigma_k$ diverge to infinity with a divergence rate of $\mathbb{E}\croc{\log\pare{\Gl}}$.

If  $\mathbb{E}_{\pi}(\mathcal{R}_{\f}) < 0$, 
then $X_k$ converges linearly to the global optimum $x^{\star}$ with a convergence rate of $-\mathbb{E}_{\pi}(\mathcal{R}_{\f})$ and the step-size converges \new{linearly} to zero. 
\label{almostSureAndLogProgress}
\end{theorem}
\begin{proof}
Theorem~\ref{geometric-drift} ensures that $\{Z_k \, ; k \in \mathbb{N}\}$ is a positive Harris recurrent chain with invariant probability measure $\pi.$
%With~\eqref{one-step-relation}, we have that
% \begin{align*}
% \frac1k\log \frac{\|X_{k} - x^{\star}\|}{\|X_{0} - x^{\star}\|} &= \frac1k \log \frac{\|Z_{k}\|}{\|Z_{0}\|} + \frac1k\log\pare{ \frac{\sigma_{k}}{\sigma_{0}} } \\
%  &= \frac1k\sum_{t=0}^{k-1}\log \frac{\|Z_{t+1}\|}{\|Z_{t}\|} + \frac1k\sum_{t=0}^{k-1}\log(\Gx(\alpha_{\f}(x^{\star} + Z_{t}, U_{t+1}) )) \enspace ,
% \end{align*}
%where $\Gx$ and $\alpha_\f$ are defined in~\eqref{step-size} and in~\eqref{alphaDefinition}. 
We start from \eqref{apply-lln}. Since $z \mapsto \log\|z\|$ is $\pi$-integrable, Theorem~\ref{lln-origin} ensures that the LLN holds with $ \lim_{k \to \infty} \frac1k\sum_{t=0}^{k-1}\log \frac{\|Z_{t+1}\|}{\|Z_{t}\|} = \dsp\int \log \pare{\|z\|} \pi(\diff z) -  \dsp\int \log \pare{\|z\|} \pi(\diff z) = 0$.

Let us consider the chain $\{W_k = (Z_{k}, U_{k+2}) \, ; k \in \mathbb{N}\}$. Then thanks to Proposition~\ref{general-stability}, $\{W_k = (Z_{k}, U_{k+2}) \, ; k \in \mathbb{N}\}$ is geometrically ergodic with invariant probability measure $\pi\times \N_{\n\lambda}$. Define the function $g$ for $\pare{ (z_{1}, u_{3}), \,(z_{2}, u_{4}) } \in  \pare{\R^{\n}\times\R^{\n\lambda} }^{2}$ as $g\pare{ (z_{1}, u_{3}), (z_{2}, u_{4}) } = \log(\Gx(\alpha_{\f}(x^{\star}+z_{2}, u_{3} )))$.
We have by Proposition~\ref{R-integrable} that for all natural integer $t$, 
$\mathbb{E}_{\pi\times\N_{\n\lambda}}(\abs{g(W_{t}, W_{t+1})} ) \leq  \frac{\lambda !}{(\lambda - \mu)!} \mathbb{E}_{Y\sim\N_{\n\mu}}\croc{ \abs{ \log\circ \,\Gx}(Y) } < \infty .$ By Theorem~\ref{general-lln} or Corollary~\ref{coro-lln}, for any initial distribution, $\frac1k\sum_{t=0}^{k-1}\log(\Gx(\alpha_{\f}(x^{\star}+Z_{t}, U_{t+1}) ))$ converges almost surely towards $ \mathbb{E}_{\pi \times \N_{\n\lambda}}(g(W_{1}, W_{2})) = \mathbb{E}_{\pi}(\mathcal{R}_{\f})$.
 
 Let us prove now~\eqref{expected-log-progress-main}. Equation~\eqref{one-step-relation} \new{implies that}\del{ and~\eqref{expected-step-plus} show that} for all $z \in \R^{\n}$
\begin{align*}
 \mathbb{E}_{z} \croc{\log \frac{\|X_{k+1}- x^{\star} \|}{\|X_{k} - x^{\star} \|}} &= \mathbb{E}_{z}\croc{\log \frac{\|Z_{k+1}\|}{\|Z_{k}\|} } +  \mathbb{E}_{z}\croc{\log\frac{\sigma_{k+1}}{\sigma_{k}} } \\
 &=   \dsp\int P^{k+1}(z, \diff y) \log(\|y\|) - \dsp\int P^{k}(z, \diff y) \log(\|y\|) \nonumber\\&+   \dsp\int P^{k}(z, \diff y) \mathcal{R}_{\f}(y) .
 \end{align*}

Define $h$ on $\R^{\n}$ as $h(z) = 1 + \abs{\log\| z \|}$ for all $z \in \R^{\n}$ which is $\pi$-integrable thanks to Proposition~\ref{log-integrable}.
Then $z \mapsto \log\|z\|$ is $\pi$-integrable, and\del{ by~\cite[Theorem 14.0.1]{meyn2012markov},} for $z \in \acco{y \in \R^{\n} \, ; V(y) < \infty} = \R^{\n}$, $\dsp \lim_{k \to \infty} \| P^{k}(z, \cdot) - \pi \|_{h} = 0$ \cite[Theorem 14.0.1]{meyn2012markov}.
 Then $\dsp\lim_{k \to \infty} \int P^{k+1}(z, \diff y) \log(\|y\|) = \dsp\lim_{k \to \infty} \int P^{k}(z, \diff y) \log(\|y\|) = \dsp\int \log(\|y\|) \pi(\diff y)$.
In addition, $\abs{\mathcal{R}_{\f}} / h$ is bounded, then $\dsp\lim_{k \to \infty} \int P^{k}(z, \diff y) \mathcal{R}_{\f}(y) = \dsp\int \mathcal{R}_{\f}(y) \pi(\diff y) = \mathbb{E}_{\pi}(\mathcal{R}_{\f})$, and finally~\eqref{expected-log-progress-main} follows.
% $$
%\dsp\lim_{k\to\infty} \mathbb{E}_{\frac{x-x^{\star}}{\sigma}}\croc{ \log\frac{ \| X_{k+1} - x^{\star} \|   }{ \| X_{k} - x^{\star} \|   } } = \dsp\lim_{k\to\infty} \mathbb{E}_{\frac{x-x^{\star}}{\sigma}} \croc{  \log\frac{ \sigma_{k+1} }{ \sigma_{k} }}  = \mathbb{E}_{\pi}(\mathcal{R}_{\f}) .
% $$
 %
We also note that if $\f$ satisfies F2, then thanks to Proposition~\ref{linear-invariance}, for all $z \in \R^{\n},$ $\mathcal{R}_{\f}(z) = \mathbb{E}\croc{ \log\pare{\Gl} } $, hence $\mathcal{R}_{\f}$ is constant. Then $\mathbb{E}_{\pi} \pare{ \mathcal{R}_{\f} } =  \dsp\int \mathcal{R}_{\f}(z) \pi(\diff z) = \mathbb{E}\croc{ \log\pare{\Gl} } $.
If in addition $\mathbb{E}\croc{ \log\pare{\Gl} } > 0$, we obtain that $\| X_{k} \|$ and $\sigma_{k}$ both diverge to $\infty$ when $k$ goes to $\infty$.  
\end{proof}

The result that both the step-size and log distance converge (resp.\ diverge) to the optimum (resp.\ to $\infty$) at the same rate is noteworthy and directly follows from our theory. In addition, we provide the exact expression of the rate. Yet it is expressed using the stationary distribution of the Markov chain $\{Z_k \, ; k \in \mathbb{N}\}$ for which we know little information. From a practical perspective, while we never know the optimum of a function on a real problem, \eqref{almost-sure-linear-convergence-main} suggests that we can track the evolution of the step-size to define a termination criterion based on the tolerance of the x-values.
\del{The almost sure linear behavior result  in \eqref{almost-sure-linear-convergence-main} derives from applying a Law of Large Numbers (LLN) to $\{Z_k \, ; k \in \mathbb{N}\}$ and a generalized law of large numbers to the chain $\{(Z_k,\, U_{k+2}) \, ; k \in \mathbb{N}\}$. Our proof mainly relies on showing that $\{Z_k \, ; k \in \mathbb{N}\}$ satisfies an LLN. Yet we prove stronger properties than what is needed for an LLN and imply from there a central limit theorem related to the expected logarithm of the step-size change.}

\subsection{Central Limit Theorem}
\label{CLT}

The rate of convergence (or divergence) of a step-size adaptive $(\mu/\mu_{w},\lambda)$-ES given in \eqref{almost-sure-linear-convergence-main} is expressed as $\abs{\mathbb{E}_{\pi}(\mathcal{R}_{\f})}$ where $\pi$ is the invariant probability measure of the $\sigma$-normalized Markov chain and $\mathcal{R}_{\f}$ is defined in~\eqref{expected-step-size-f}. Yet we do not have an explicit expression for $\pi$ and thus of $\mathbb{E}_{\pi}(\mathcal{R}_{\f})$.
However, we can approximate $\mathbb{E}_{\pi}(\mathcal{R}_{\f})$ with Monte Carlo simulations. We present a central limit theorem for the approximation of $\mathbb{E}_{\pi}(\mathcal{R}_{\f})$  as $\frac1t \sum_{k=0}^{t-1} \mathcal{R}_{\f}(Z_k)$ where $\{ Z_{k} ; k \in \mathbb{N}\}$ is the homogeneous Markov chain defined in Proposition~\ref{prop:MC}.

\begin{theorem} (Central limit theorem for the expected logarithm of the step-size)
\label{clt-cbsaes}
Let $\f$ be a scaling-invariant function with respect to $x^{\star}$ that satisfies F1 or F2. Let $\{(X_k,\sigma_k) \, ; k \in \mathbb{N}\}$ be the sequence defined in~\eqref{incumbent} and \eqref{step-size} such that Assumptions A1$-$A5 are satisfied.
If the expected logarithm of the step-size increases on nontrivial linear functions, i.e.\  if $\mathbb{E}\croc{\log\pare{\Gl}} > 0$  where $\Gl$ is defined in~\eqref{step-size-notation}, then $\acco{Z_k = (X_k - x^{\star})/\sigma_k\, ; k \in \mathbb{N} }$ is a Markov chain admitting an invariant probability measure $\pi$. Define $\mathcal{R}_{\f}$ as in~\eqref{expected-step-size-f} and for all positive integer $t,$ define $S_{t}(\mathcal{R}_{\f}) = \sum_{k=0}^{t-1} \mathcal{R}_{\f}(Z_{k})$. Then the constant $\gamma^{2}$ defined as\\
%\begin{align*}
$\mathbb{E}_{ \pi } \croc{   \pare{ \mathcal{R}_{\f}(Z_{0}) - \mathbb{E}_{\pi}(\mathcal{R}_{\f}) }^{2}   } + 2 \dsp \sum_{k=1}^{\infty}  \mathbb{E}_{ \pi } \croc{ \pare{ \mathcal{R}_{\f}(Z_{0}) - \mathbb{E}_{\pi}(\mathcal{R}_{\f}) }\pare{ \mathcal{R}_{\f}(Z_{k}) - \mathbb{E}_{\pi}(\mathcal{R}_{\f}) }  }$
%\end{align*}
is well defined, non-negative, finite and 
$\dsp\lim_{t\to\infty}\frac{1}{t}  \mathbb{E}_{ \pi } \croc{ \pare{ S_{t}(\mathcal{R}_{\f}) - t \,\mathbb{E}_{\pi}(\mathcal{R}_{\f}) }^{2} } = \gamma^{2}$.

If $\gamma^{2} > 0,$ then the central limit theorem holds in the sense that for any initial condition $z_{0},$ 
$\dsp \sqrt{\frac{t}{\gamma^{2}} } \pare{ \frac1t S_{t}(\mathcal{R}_{\f}) -  \mathbb{E}_{\pi}(\mathcal{R}_{\f}) }$  converges in distribution to $\N(0, 1).$ 
If $\gamma^{2} = 0,$ then 
$\dsp\lim_{t \to \infty} \frac{ S_{t}(\mathcal{R}_{\f}) - t \mathbb{E}_{\pi}(\mathcal{R}_{\f}) }{ \sqrt{t} }  = 0 $ a.s.
\end{theorem}

\begin{proof}
Thanks to Proposition~\ref{R-integrable}, $\abs{  \mathcal{R}_{\f} }$ is bounded. And then there exists a positive constant $K$ large enough such that $\mathcal{R}_{\f}^{2} \leq K \, V$ where $V$ is the geometric drift function of $\acco{ Z_{k}\, ; k \in \mathbb{N} }$ given by Theorem~\ref{geometric-drift}. Then $K \, V$ remains a geometric drift function. Thanks to Theorem~\ref{clt-generic}, the constant $\gamma$ defined as
$
\mathbb{E}_{ \pi } \croc{   \pare{ \mathcal{R}_{\f}(Z_{0}) -  \mathbb{E}_{\pi}(\mathcal{R}_{\f}) }^{2}   } + 2  \sum_{k=1}^{\infty}  \mathbb{E}_{ \pi } \croc{ \pare{ \mathcal{R}_{\f}(Z_{0}) -  \mathbb{E}_{\pi}(\mathcal{R}_{\f}) }\pare{ \mathcal{R}_{\f}(Z_{k}) -  \mathbb{E}_{\pi}(\mathcal{R}_{\f}) }  }
$
 is well defined, non-negative, finite and  
$\dsp\lim_{t\to\infty}\frac{1}{t}  \mathbb{E}_{ \pi } \croc{ \pare{ S_{t}(\mathcal{R}_{\f}) - t  \mathbb{E}_{\pi}(\mathcal{R}_{\f}) }^{2} } = \gamma^{2}.$
Moreover if $\gamma^{2} > 0,$ then the CLT holds for any $z_{0}$ as follows
%\begin{align*}
$
\lim_{t \to \infty} P_{z_{0}} \pare{ (t \gamma^{2})^{-\frac{1}{2}} \pare{S_{t}(\mathcal{R}_{\f}) - t  \mathbb{E}_{\pi}(\mathcal{R}_{\f})}  \leq z} = \dsp\int_{-\infty}^{z} \frac{1}{\sqrt{2\pi}} e^{-u^{2}/2} \diff u.
$
%\end{align*}
Which can be rephrased as 
%\begin{align}
$\frac{1}{ \sqrt{t \gamma^{2}} } \pare{S_{t}(\mathcal{R}_{\f}) - t \mathbb{E}_{\pi}(\mathcal{R}_{\f})}$   converges in distribution to $\N(0, 1)$ when $t \to \infty.$
%\end{align}
And if $\gamma = 0,$ then 
%\begin{align}
$\lim_{t \to \infty} (S_{t}(\mathcal{R}_{\f}) - t \,\mathbb{E}_{\pi}(\mathcal{R}_{\f})) / \sqrt{t}  = 0 \textit{ a.s.}
$%\end{align}
\end{proof}

\subsection{Sufficient conditions for the linear behavior of \nnew{the} $(\mu/\mu_w, \lambda)$-CSA1-ES and \nnew{the} $(\mu/\mu_w, \lambda)$-xNES}
\label{sec:SC-linear-behavior-CSA-xNES}

Theorems~\ref{theo:main-CV} and~\ref{clt-cbsaes} hold for an abstract step-size update function $\Gamma$ that satisfies Assumptions A1$-$A4. For the step-size update functions of the $\pare{\mu/\mu_w, \lambda}$-CSA1-ES and the $\pare{\mu/\mu_w, \lambda}$-xNES defined in \eqref{step-size-csa} and \eqref{step-size-xnes}, sufficient and necessary conditions to obtain a step-size increase on linear functions are presented in the next proposition. They are expressed using the weights and the $\mu$ best order statistics $\N^{1:\lambda}, \dots, \N^{\mu:\lambda}$ of a sample of $\lambda$ standard normal distributions $\N^{1}, \dots, \N^{\lambda}$
defined such as $\N^{1:\lambda} \leq \N^{2:\lambda} \leq \dots \leq \N^{\lambda:\lambda}$. 
\begin{proposition}[\textbf{Necessary and sufficient condition for step-size increase on nontrivial linear functions}]%
\label{prop:weight-condition}
For the $\pare{\mu/\mu_w, \lambda}$-CSA-ES algorithm without cumulation, $ \mathbb{E}\left[ \log\pare{ (\Gx_{ \rm{CSA1}})_{{\rm linear}}^{\star} } \right] = \frac{1}{2 d_{\sigma}\n}\left(\mathbb{E}\left[\left(\sum_{i=1}^{\mu} \frac{w_{i}}{\| w\|  } \N^{i:\lambda} \right)^{2} \right] - 1\right)$. Therefore, the expected logarithm of the step-size increases on nontrivial linear functions if and only if
%\begin{align}
$\mathbb{E}\croc{\pare{\sum_{i=1}^{\mu} \frac{w_{i}}{\| w\|  } \N^{i:\lambda} }^{2} }  > 1.$
%\enspace. \label{condition-csa}
%\end{align}

For the $\pare{\mu/\mu_w, \lambda}$-xNES without covariance matrix adaptation, if  $w_{i} \geq 0$ for all $i=1,\dots, \mu$, $\mathbb{E}\croc{ \log\pare{ (\Gx_{ \rm{xNES}})_{\rm linear}^{\star} } } = \frac{1}{2 d_{\sigma}\n}\pare{\sum_{i=1}^{\mu} \frac{w_{i}}{\sum_{j=1}^{\mu} w_{j}} \mathbb{E}\croc{\pare{ \N^{i:\lambda} }^{2} } - 1}$. Therefore, the expected logarithm of the step-size increases on nontrivial linear functions if and only if
%\begin{align}
 $\sum_{i=1}^{\mu} \frac{w_{i}}{\sum_{j=1}^{\mu} w_{j}} \mathbb{E}\croc{\pare{ \N^{i:\lambda} }^{2} }  > 1$.
 % \enspace. 
%\label{condition-xnes}
%\end{align}
In addition, this latter equation is satisfied if $\lambda, \mu$ and $w$ are set such that $\lambda \geq 3,$ $\mu < \frac{\lambda}{2}$ and $w_{1} \geq w_{2} \geq \cdots \geq w_{\mu} \geq 0$.
\end{proposition}
\begin{proof}
We first prove the statement related to the $\pare{\mu/\mu_w, \lambda}$-CSA1-ES. Then we show the condition regarding \new{the} $\pare{\mu/\mu_w, \lambda}$-xNES. Finally we prove the general practical condition that allows to obtain the condition regarding the xNES algorithm.

 If $m$ is a positive integer and $u = \pare{u^{1}, \dots, u^{m}} \in \R^{\n m},$ we denote $u_{1} = \pare{u^{1}_{1}, \dots, u^{m}_{1} }$ and $u_{-1} = \pare{u^{1}_{-1}, \dots, u^{m}_{-1} }$ where $u^{i}_{-1} = \pare{u^{i}_{2},\dots,u^{i}_{\n}}$ for $i  = 1,\dots, m.$ Define the nontrivial linear function $\l^{\star}$ such that $l^{\star}(x) = x_{1}$ for $x \in \R^{\n}$, and denote by $e_{1}$ the unit vector $\pare{1, \dots, 0}$.

\textit{Part 1.}
We prove that $\mathbb{E}_{U_{1}\sim\N_{\n\lambda}}\croc{ \log\pare{\Gx_{\rm CSA1} \pare{ \alpha_{l^{\star}}(e_{1}, U_{1})} } }$ has the same sign than $\mathbb{E}\croc{\pare{\sum_{i=1}^{\mu} \frac{w_{i}}{\| w\| } \N^{i:\lambda} }^{2} } - 1,$ and apply Theorem~\ref{almostSureAndLogProgress}. 
%\end{align}
% \begin{align}
We have\\
 $2d_{\sigma}\|w\|^{2}\n \mathbb{E}_{U_{1}\sim\N_{\n\lambda}}\croc{ \log\pare{\Gx_{\rm CSA1} \pare{ \alpha_{l^{\star}}(e_{1}, U_{1})}} } =$\\$ \left( \mathbb{E}_{U_{1}\sim\N_{\n\lambda}}\croc{\|\sum_{i=1}^{\mu}
w_{i} \croc{\alpha_{l^{\star}}(e_{1}, U_{1})}_{i}  \|^{2} }  - \|w\|^{2}\n  \right)  .
$ Therefore it is enough to show that 
%\begin{align}
$\mathbb{E}\croc{ \left\Vert \sum_{i=1}^{\mu} \frac{w_{i}}{\|w\|} \croc{\alpha_{l^{\star}}(e_{1}, U_{1})}_{i}  \right\rVert^{2} }  - \n = \mathbb{E}\croc{\pare{\sum_{i=1}^{\mu} \frac{w_{i}}{\|w\|} \N^{i:\lambda} }^{2} } - 1.$
%\label{progress-to-order-stats-csa}
%\end{align}
Recall that the probability density function of $\alpha_{l^{\star}}(e_{1}, U_{1})$ is $p_{e_{1}}^{l^{\star}}$ defined for all $u\in \R^{\n\mu}$ as 
%\begin{align}
 $p_{e_{1}}^{ l^{\star} }(u) = \frac{\lambda !}{(\lambda - \mu)!} (1-Q_{e_{1}}^{ l^{\star} }(u^{\mu}))^{\lambda-\mu} \dsp\prod_{i=1}^{\mu-1} \mathds{1}_{\acco{l^{\star}(u^i) < l^{\star}(u^{i+1})}} \dsp\prod_{i=1}^{\mu} p_{\N_{\n}}(u^{i}).
 $%\end{align}
 
Denote $A = \mathbb{E}_{U_{1}\sim\N_{\n\lambda}}\croc{\lVert \sum_{i=1}^{\mu} \frac{w_{i}}{\|w\|} \croc{\alpha_{l^{\star}}(e_{1}, U_{1})}_{i}  \rVert^{2} }.$ It follows that\\
%\begin{align*}
$A = \frac{\lambda !}{(\lambda - \mu)!} \! \dsp\bigintsss \left\lVert\sum_{i=1}^{\mu} \frac{w_{i}}{\|w\|} u^{i} \right\rVert^{2}  (1-Q_{e_{1}}^{ l^{\star} }(u^{\mu}))^{\lambda-\mu}
\prod_{j=1}^{\mu-1} \mathds{1}_{\acco{l^{\star}(u^j) < l^{\star}(u^{j+1})}} \dsp\prod_{j=1}^{\mu} p_{\N_{\n}}(u^{j}) \diff u $ \\ 
$= \dsp \bigintsss \!\! \pare{\left\lVert\sum_{i=1}^{\mu} \frac{w_{i}}{\|w\|} u_{1}^{i} \right\rVert^{2} \! +  \left\| \sum_{i=1}^{\mu} \frac{w_{i}}{\|w\|} u_{-1}^{i} \right\|^{2 } } \!\! P\pare{\N > u_{1}^{\mu}}^{\lambda-\mu} \dsp\prod_{j=1}^{\mu-1} \mathds{1}_{\acco{u_{1}^j < u_{1}^{j+1} } } \dsp\prod_{j=1}^{\mu} p_{\N}(u_{1}^{j}) $\\$ \dsp\prod_{j=1}^{\mu} p_{\N_{\n-1}}(u_{-1}^{j}) \,\diff u.$
%\end{align*}
We expand the integrand, the first term is $ \mathbb{E}\croc{\pare{\sum_{i=1}^{\mu} \frac{w_{i}}{\|w\|} \N^{i:\lambda} }^{2} }.$ 
%that is $\mathbb{E}_{U_{1}\sim\N_{\n\lambda}}\croc{\left\lVert\sum_{i=1}^{\mu} \frac{w_{i}}{\|w\|} \pare{\alpha_{l^{\star}}(e_{1}, U_{1})}_{i}  \right\rVert^{2} }.$ 

Denote $B = \mathbb{E}\croc{\pare{\sum_{i=1}^{\mu} \frac{w_{i}}{\|w\|} \N^{i:\lambda} }^{2} }$ and $C = B-A.$ Then
%\begin{eqnarray*}
$\frac{(\lambda - \mu)!}{\lambda !} C$ equals $\bigintss \left\lVert\sum_{i=1}^{\mu} \frac{w_{i}}{\|w\|} u_{-1}^{i} \right\rVert^{2 } P\pare{\N > u_{1}^{\mu}}^{\lambda-\mu}\prod_{j=1}^{\mu-1} \mathds{1}_{\acco{u_{1}^j < u_{1}^{j+1} } } \dsp\prod_{j=1}^{\mu} p_{\N}(u_{1}^{j}) p_{\N_{\n-1}}(u_{-1}^{j}) \diff u$.\\
Then $C = \bigintss_{\R^{\mu}}  \frac{\lambda !}{(\lambda - \mu)!} P\pare{\N > u_{1}^{\mu}}^{\lambda-\mu}  \prod_{j=1}^{\mu-1} \mathds{1}_{\acco{u_{1}^j < u_{1}^{j+1} } } \prod_{j=1}^{\mu} p_{\N}(u_{1}^{j}) \diff u_{1}$ \\
$\bigintss_{\R^{(\n-1)\mu}}  \left\lVert\sum_{i=1}^{\mu} \frac{w_{i}}{\|w\|} u_{-1}^{i} \right\rVert^{2 } \prod_{j=1}^{\mu} p_{\N_{\n-1}}(u_{-1}^{j}) \diff u_{-1}$.
The first integral equals $1$ as it is the integral of a probability density function. The second integral is equal to $\mathbb{E}\croc{\|\sum_{i=1}^{\mu} \frac{w_{i}}{\|w\|} W_{i} \|^{2} }$ where $W_{1}, \dots, W_{\mu}$ are i.i.d. random variables of law $\N_{\n-1}.$ Then the law of $\sum_{i=1}^{\mu} \frac{w_{i}}{\|w\| } W_{i}$ is $\N_{\n-1}$. Then $\mathbb{E}\croc{\left\lVert\sum_{i=1}^{\mu} \frac{w_{i}}{\|w\|} W_{i} \right\rVert^{2} } = \n - 1.$
Hence\\
$
\mathbb{E}_{U_{1}\sim\N_{\n\lambda}}\croc{\left\lVert\sum_{i=1}^{\mu} \frac{w_{i}}{\|w\|} \croc{\alpha_{l^{\star}}(e_{1}, U_{1})}_{i}  \right\rVert^{2} }$ $-$ $\mathbb{E}\croc{\pare{\sum_{i=1}^{\mu} \frac{w_{i}}{\|w\|} \N^{i:\lambda} }^{2} } = \n- 1,
$
which ends this part.

%%%%%%%%%%%%%%%%%%%%%%%%%%%%%%%%%%%%%%

\textit{Part 2.} For the second item, we show that $\mathbb{E}_{U_{1}\sim\N_{\n\lambda}}\croc{ \log\pare{\Gx_{\rm xNES} \pare{ \alpha_{l^{\star}}(e_{1}, U_{1})}} }$ has the same sign than $\sum_{i=1}^{\mu} \frac{w_{i}}{\sum_{j=1}^{\mu} w_{j} } \mathbb{E}\croc{\pare{ \N^{i:\lambda} }^{2} } - 1, $ and apply Theorem~\ref{almostSureAndLogProgress}.
We have
$\mathbb{E}_{U_{1}\sim\N_{\n\lambda}}\croc{ \log\pare{\Gx_{\rm xNES} \pare{ \alpha_{l^{\star}}(e_{1}, U_{1})}} }$ $=$\\ $ \frac{1}{2 d_{\sigma} n \sum_{i=1}^{\mu} w_{i}} \sum_{i=1}^{\mu}
w_{i} \pare{\mathbb{E}_{U_{1}\sim\N_{\n\lambda}}\croc{\| \croc{ \alpha_{l^{\star}}(e_{1}, U_{1}) }_{i} \|^{2} }  - n}.$
Then it is enough to show: 
$\dsp \sum_{i=1}^{\mu} w_{i} \pare{\mathbb{E}_{U_{1}\sim\N_{\n\lambda}}\croc{\| \croc{ \alpha_{l^{\star}}(e_{1}, U_{1}) }_{i} \|^{2} }  - n} = \dsp\sum_{i=1}^{\mu}w_{i} \mathbb{E}\croc{\pare{ \N^{i:\lambda} }^{2} } - \sum_{i=1}^{\mu} w_{i}.$
Denote $A =  \sum_{i=1}^{\mu} w_{i} \mathbb{E}_{U_{1}\sim\N_{\n\lambda}}\croc{\| \croc{ \alpha_{l^{\star}}(e_{1}, U_{1}) }_{i} \|^{2} }.$ It follows
%\begin{align*}
%A &= \frac{\lambda !}{ (\lambda - \mu)! } \int \sum_{i=1}^{\mu} w_{i}  \| u^{i} \|^{2}  (1-Q_{e_{1}}^{ l^{\star} }(u^{\mu}))^{\lambda-\mu}
%\prod_{j=1}^{\mu-1} \mathds{1}_{\acco{l^{\star}(u^j) < l^{\star}(u^{j+1})}} \\
%&\prod_{j=1}^{\mu} p_{\N_{\n}}(u^{j}) \diff u
%= \frac{\lambda !}{ (\lambda - \mu)! }\int \pare{ \sum_{i=1}^{\mu} w_{i}  \| u^{i}_{1} \|^{2} + \sum_{i=1}^{\mu} w_{i}  \| u^{i}_{-1} \|^{2} } \\ &P\Big(\N > u_{1}^{\mu}\Big)^{\lambda-\mu} 
%\prod_{j=1}^{\mu-1} \mathds{1}_{\acco{u_{1}^j < u_{1}^{j+1} } } \prod_{j=1}^{\mu} p_{\N}(u_{1}^{j})  \dsp\prod_{j=1}^{\mu} p_{\N_{\n-1}}(u_{-1}^{j}) \,\diff u.
%\end{align*}
$ A = \frac{\lambda !}{ (\lambda - \mu)! } \int \sum_{i=1}^{\mu} w_{i}  \| u^{i} \|^{2}  (1-Q_{e_{1}}^{ l^{\star} }(u^{\mu}))^{\lambda-\mu}
\prod_{j=1}^{\mu-1} \mathds{1}_{\acco{l^{\star}(u^j) < l^{\star}(u^{j+1})}} 
\prod_{j=1}^{\mu} p_{\N_{\n}}(u^{j}) \diff u $ which is equal to \\
$
 \frac{\lambda !}{ (\lambda - \mu)! }\int \pare{ \sum_{i=1}^{\mu} w_{i}  \| u^{i}_{1} \|^{2} + \sum_{i=1}^{\mu} w_{i}  \| u^{i}_{-1} \|^{2} }  P\Big(\N > u_{1}^{\mu}\Big)^{\lambda-\mu} 
\prod_{j=1}^{\mu-1} \mathds{1}_{\acco{u_{1}^j < u_{1}^{j+1} } } $\\
$\prod_{j=1}^{\mu} p_{\N}(u_{1}^{j})  \dsp\prod_{j=1}^{\mu} p_{\N_{\n-1}}(u_{-1}^{j}) \,\diff u$. Then after expansion, the integral of the first term of the integrand equals $\dsp \frac{(\lambda - \mu)!}{\lambda !} \sum_{i=1}^{\mu}w_{i} \mathbb{E}\croc{\pare{ \N^{i:\lambda} }^{2} }.$ 
Denote $B =  \dsp\sum_{i=1}^{\mu} w_{i} \mathbb{E}\croc{\pare{ \N^{i:\lambda} }^{2} }$ and $C = A-B.$ Then $ \dsp \frac {(\lambda - \mu)!} {\lambda !} C = \dsp\int \sum_{i=1}^{\mu} w_{i} \| u_{-1}^{i} \|^{2 } \pare{ P\pare{\N > u_{1}^{\mu}} }^{\lambda-\mu} $\\
$\dsp\prod_{j=1}^{\mu-1} \mathds{1}_{\acco{u_{1}^j < u_{1}^{j+1} } } \dsp\prod_{j=1}^{\mu} p_{\N}(u_{1}^{j}) p_{\N_{\n-1}}(u_{-1}^{j}) \diff u$. Then $C = \dsp\int_{\R^{\mu}}  \dsp \frac {\lambda !} {(\lambda - \mu)!} \dsp\prod_{j=1}^{\mu-1} \mathds{1}_{\acco{u_{1}^j < u_{1}^{j+1} } }$ \\$P\pare{\N > u_{1}^{\mu}}^{\lambda-\mu} \dsp\prod_{j=1}^{\mu} p_{\N}(u_{1}^{j}) \diff u_{1} \dsp\int_{\R^{(\n-1)\mu}}  \sum_{i=1}^{\mu} w_{i}\|u_{-1}^{i} \|^{2 } \dsp\prod_{j=1}^{\mu} p_{\N_{\n-1}}(u_{-1}^{j}) \diff u_{-1}$.
The first integral equals $1$ as it is the integral of a probability density function. The second one equals $\sum_{i=1}^{\mu} w_{i} \mathbb{E}\croc{\| \N_{\n-1} \|^{2} }  = \pare{\n-1} \sum_{i=1}^{\mu} w_{i}.$
We finally have that \\
$\sum_{i=1}^{\mu} w_{i} \mathbb{E}_{U_{1}\sim\N_{\n\lambda}}\croc{\| \croc{ \alpha_{l^{\star}}(e_{1}, U_{1}) }_{i} \|^{2} } - \sum_{i=1}^{\mu}w_{i} \mathbb{E}\croc{\pare{ \N^{i:\lambda} }^{2} }= \pare{\n-1} \sum_{i=1}^{\mu} w_{i} .$

%%%%%%%%%%%%%%%%%%%%%%%%%%%%%%%%%%%%%%%%%%%%%%%%%%%

\textit{Part 3.}
If $\pare{X_{1}, \dots, X_{\lambda}}$ is distributed according to $\pare{\N^{1:\lambda}, \dots, \N^{\lambda:\lambda} },$ then $X_{1} \leq \dots \leq X_{\lambda}$ and then $-X_{\lambda} \leq \dots \leq -X_{1} .$ Therefore $\pare{-X_{\lambda}, \dots, -X_{1}}$ is also distributed according to $\pare{\N^{1:\lambda}, \dots, \N^{\lambda:\lambda} }.$
Assume that $\lambda \geq 3$ and $\mu > \frac{\lambda}{2}.$ We show the results in two parts.

\textit{Part 3.1.} First we assume that $w_{1} = \dots = w_{\mu} = \frac{1}{\mu}.$
In this case, we have to prove that:
$1 < \sum_{i=1}^{\mu} \frac{w_{i}}{\sum_{j=1}^{\mu} w_{j}} \mathbb{E}\croc{\pare{ \N^{i:\lambda} }^{2} } = \frac{1}{\mu}\sum_{i=1}^{\mu} \mathbb{E}\croc{ \pare{\N^{i:\lambda}}^{2}}.
$
Since $\N^{1:\lambda} \leq \dots \leq \N^{\lambda:\lambda}$ is equivalent to $-\N^{\lambda:\lambda} \leq \dots \leq -\N^{1:\lambda},$ then $\pare{\N^{1:\lambda}, \dots, \N^{\lambda:\lambda} }$ has the distribution of $\pare{-\N^{\lambda:\lambda}, \dots, -\N^{1:\lambda} }$. And then for $i = 1, \dots, \lambda,$
$\pare{\N^{i:\lambda}}^{2}$ has the distribution of $\pare{\N^{\lambda-i+1:\lambda}}^{2}$. It follows that
%\begin{align}
$
\sum_{i=1}^{\lambda} \mathbb{E}\croc{\pare{ \N^{i:\lambda} }^{2} } = 2 \sum_{i=1}^{\mu} \mathbb{E}\croc{\pare{ \N^{i:\lambda} }^{2} } +
 \sum_{i=\mu+1}^{\lambda-\mu} \mathbb{E}\croc{\pare{ \N^{i:\lambda} }^{2} }$.
% \label{selection-mu}
%\end{align}
Moreover, 
%\begin{align}
$
\sum_{i=1}^{\lambda} \mathbb{E}\croc{\pare{ \N^{i:\lambda} }^{2} } = \sum_{i=1}^{\lambda} \mathbb{E}\croc{\pare{ \N^{i} }^{2} } = \lambda,
$
%\label{selection-shuffle}
%\end{align}
meaning that we lose the selection effect of the order statistics when we do the above summation. Both equations above
%Equations~\eqref{selection-mu} and~\eqref{selection-shuffle}
ensure that
\begin{align}
2 \sum_{i=1}^{\mu} \mathbb{E}\croc{\pare{ \N^{i:\lambda} }^{2} } +
 \sum_{i=\mu+1}^{\lambda-\mu} \mathbb{E}\croc{\pare{ \N^{i:\lambda} }^{2} } = \lambda.
 \label{selection-equality}
\end{align}
For any $j \in \acco{\mu+1,\dots,\lambda-\mu}$ and any $i \in \acco{ 1\dots,\mu},$ 
$\N^{i:\lambda} \leq \N^{j:\lambda} \leq  \N^{\lambda+1-i:\lambda}.$ Therefore if $ \N^{j:\lambda} \geq 0,$ $ \pare{\N^{j:\lambda}}^{2} \leq \pare{  \N^{\lambda+1-i:\lambda}  }^{2},$ and if $ \N^{j:\lambda} \leq 0,$ $ \pare{\N^{j:\lambda}}^{2} \leq \pare{\N^{i:\lambda} }^{2}.$
Since $\pare{  \N^{\lambda+1-i:\lambda}  }^{2}$ has the distribution of $\pare{  \N^{i:\lambda}  }^{2},$ it follows that for all $j \in \acco{\mu+1,\dots,\lambda-\mu}$ and $i \in \acco{ 1\dots,\mu}:$
$
\pare{  \N^{j:\lambda}  }^{2} \leq \pare{  \N^{i:\lambda}  }^{2},
$
and it is straightforward to see that the we do not have almost sure equality. It then follows that for all $j \in \acco{\mu+1,\dots,\lambda-\mu}\footnote{Note that the set $\acco{\mu+1,\dots,\lambda-\mu}$ is not empty since $1 \leq \mu < \frac{\lambda}{2}$.}$ and $i \in \acco{ 1\dots,\mu}:$
$
 \mathbb{E}\croc{\pare{ \N^{j:\lambda} }^{2} } <  \mathbb{E}\croc{\pare{ \N^{i:\lambda} }^{2} }. 
$
Therefore for all $j \in \acco{\mu+1,\dots,\lambda-\mu}$
\begin{align}
\mathbb{E}\croc{\pare{ \N^{j:\lambda} }^{2} } < \frac{1}{\mu}\sum_{i=1}^{\mu} \mathbb{E}\croc{\pare{ \N^{i:\lambda} }^{2} }. 
\label{selection-average}
\end{align}
With~\eqref{selection-average} and~\eqref{selection-equality}, we have
$\lambda =  2 \sum_{i=1}^{\mu} \mathbb{E}\croc{\pare{ \N^{i:\lambda} }^{2} } +
 \sum_{i=\mu+1}^{\lambda-\mu} \mathbb{E}\croc{\pare{ \N^{i:\lambda} }^{2} } 
 <  2 \sum_{i=1}^{\mu} \mathbb{E}\croc{\pare{ \N^{i:\lambda} }^{2} } + \frac{ \lambda-2\mu }{\mu}\sum_{i=1}^{\mu} \mathbb{E}\croc{\pare{ \N^{i:\lambda} }^{2} } 
 = \frac{ \lambda }{\mu}\sum_{i=1}^{\mu} \mathbb{E}\croc{\pare{ \N^{i:\lambda} }^{2} }.
$
Finally it follows that 
\begin{align}
\frac{1}{\mu}\sum_{i=1}^{\mu} \mathbb{E}\croc{\pare{ \N^{i:\lambda} }^{2} } > 1.
\label{particular-case-equal-weights}
\end{align}

\textit{Part 3.2.} Now we fall back to the general assumption where $w_{1} \geq \dots \geq w_{\mu}.$
Let us prove beforehand that: 
\begin{align}
\mathbb{E}\croc{\pare{ \N^{1:\lambda} }^{2} } \geq \mathbb{E}\croc{\pare{ \N^{\,2:\lambda} }^{2} } \geq \dots \geq \mathbb{E}\croc{\pare{ \N^{\mu:\lambda} }^{2} }.
\label{order-stats-inequalities}
\end{align}
Let $i \in \acco{1,\dots, \mu-1}.$ We have that
$
\N^{i:\lambda} \leq \N^{i+1:\lambda} \leq \N^{\lambda+1-i\nnew{:\lambda}}
$.
Then if $\N^{i+1:\lambda} \geq 0,$ $\pare{\N^{i+1:\lambda}}^{2} \leq \pare{\N^{\lambda+1-i:\lambda}}^{2}$ and if $\N^{i+1:\lambda} \leq 0,$ $\pare{\N^{i+1:\lambda}}^{2} \leq \pare{\N^{i:\lambda}}^{2}.$ Since $\pare{\N^{\lambda+1-i:\lambda}}^{2}$ and $\pare{\N^{i:\lambda}}^{2}$ have the same distribution, it follows that 
$\pare{\N^{i+1:\lambda}}^{2} \leq \pare{\N^{i:\lambda}}^{2}$. Therefore~\eqref{order-stats-inequalities} holds.

To prove the general case, we use the Chebyshev's sum inequality which states that if
$a_{1} \geq a_{2} \geq \dots \geq a_{\mu}$ and $b_{1} \geq b_{2} \geq \dots \geq b_{\mu},$ then
%\begin{align}
$
\frac{1}{\mu}\sum_{k=1}^{\mu} a_{k} b_{k} \geq \pare{ \frac{1}{\mu}\sum_{k=1}^{\mu} a_{k} } \pare{ \frac{1}{\mu}\sum_{k=1}^{\mu} b_{k} }.$
%\label{chebyshev}
%\end{align}
By applying Chebyshev's sum inequality on $w_{1} \geq \dots \geq w_{\mu}$ and $\mathbb{E}\croc{\pare{ \N^{1:\lambda} }^{2} } \geq \mathbb{E}\croc{\pare{ \N^{\,2:\lambda} }^{2} } \geq \dots \geq \mathbb{E}\croc{\pare{ \N^{\mu:\lambda} }^{2} }$, it follows that
%\begin{align*}
$
\frac{1}{\mu}\sum_{i=1}^{\mu} w_{i} \mathbb{E}\croc{\pare{ \N^{i:\lambda} }^{2} }
 \geq \pare{ \frac{1}{\mu}\sum_{j=1}^{\mu} w_{j} } \pare{ \frac{1}{\mu}\sum_{i=1}^{\mu} \mathbb{E}\croc{\pare{ \N^{i:\lambda} }^{2} } }.
$
%\end{align*}
Therefore,
%\begin{align}
$ \sum_{i=1}^{\mu} \frac{ w_{i} }{  \sum_{j=1}^{\mu} w_{j} } \mathbb{E}\croc{\pare{ \N^{i:\lambda} }^{2} }
 \geq \frac{1}{\mu}\sum_{i=1}^{\mu} \mathbb{E}\croc{\pare{ \N^{i:\lambda} }^{2} }$. And the first case in~\eqref{particular-case-equal-weights} ensures that  
%\begin{align}
$ \sum_{i=1}^{\mu} \frac{ w_{i} }{  \sum_{j=1}^{\mu} w_{j} } \mathbb{E}\croc{\pare{ \N^{i:\lambda} }^{2} }
 \geq \frac{1}{\mu}\sum_{i=1}^{\mu} \mathbb{E}\croc{\pare{ \N^{i:\lambda} }^{2} }  >1.
$%\end{align}

\end{proof}
The positivity of $\mathbb{E}\croc{\log\pare{\Gl}}$ is the main assumption for our main results. In this context, Proposition~\ref{prop:weight-condition} gives more practical and concrete ways to obtain the conclusion of Theorems~\ref{almostSureAndLogProgress} and~\ref{clt-cbsaes} for the $\pare{\mu/\mu_w, \lambda}$-CSA1-ES and $\pare{\mu/\mu_w, \lambda}$-xNES.
In the case where $\mu = 1$, the two conditions given in the previous proposition for CSA and xNES 
%~\eqref{condition-csa} and~\eqref{condition-xnes} 
are equivalent and yield the equation $\mathbb{E}\croc{\pare{ \N^{1:\lambda} }^{2} }  > 1$. The latter is satisfied if $\lambda \geq 3$ and $\mu = 1$, which is the linear divergence condition on linear functions of the $\pare{1, \lambda}$-CSA1-ES~\cite{chotard2012cumulative}.
Conditions similar to the one given for CSA in the previous proposition
%~\eqref{condition-csa} 
had already been derived for the so-called mutative self-adaptation of the step-size \cite{hansen2006analysis}.

\section{Conclusion and discussion}

We have proven the asymptotic linear behavior of step-size adaptive \muslashmu-ESs on composites of strictly increasing functions with continuously differentiable scaling-invariant functions. The step-size update has been modeled as an abstract function of the random input multiplied by the current step-size. Two well-known step-size adaptation mechanisms are included in this model, namely derived from the Exponential Natural Evolution Strategy (xNES)~\cite{glasmachers2010exponential} and the Cumulative Step-size Adaptation (CSA)~\cite{hansen2016cma} without cumulation.

Our main condition for the linear behavior proven in Theorem~\ref{theo:main-CV} reads ``the logarithm of the step-size increases on linear functions'', formally, stated as $\mathbb{E}\croc{ \log\pare{\Gl} } > 0$ where $\Gl$ is the step-size change on nontrivial linear functions. This condition is equivalent to the geometric divergence of the step-size on nontrivial linear functions, as shown by the next lemma.
\begin{lemma}
Let $f$ be an increasing transformation of a nontrivial linear function, i.e. satisfy F2. Let $\{(X_k,\sigma_k) \, ; k \in \mathbb{N}  \}$ be the  sequence defined in~\eqref{incumbent} and~\eqref{step-size}. Assume that $\{U_{k+1} \,; k \in \mathbb{N} \}$ satisfies Assumption A5 and that $\Gx$ satisfies Assumptions A2 and A4, i.e.\  $\Gx$ is invariant under rotation and $\log\circ\,\Gamma$ is $\N_{n\mu}$-integrable. Then $\dsp \lim_{k\to\infty} \frac1k \log\frac{ \sigma_{k} }{ \sigma_{0} } = \mathbb{E}\croc{\log\pare{\Gl}} $.
\end{lemma}
\begin{proof}
We have $ \frac1k \log\frac{ \sigma_{k} }{ \sigma_{0} } = \frac1k \sum_{t=0}^{k-1} \log\frac{ \sigma_{t+1} }{ \sigma_{t} }.$ With~\eqref{step-size-notation} and Proposition~\ref{linear-invariance}, $\sigma_{t+1} = \sigma_{t} \, \Gx \pare{ \alpha_{l^{\star}}(0, U_{t+1})}$ where $l^{\star}$ is the linear function defined as $l^{\star}(x) = x_{1}$ for $x \in \R^{n}$.
Therefore $ \frac1k \log\frac{ \sigma_{k} }{ \sigma_{0} } = \frac1k \sum_{t=0}^{k-1} \pare{\log\circ \Gx\circ\alpha_{l^{\star} } } \pare{ 0, U_{t+1} }.$ Using Assumption A3 and Lemma~\ref{log-gamma-lemma}, we have that the function $u \mapsto \pare{\log\circ \Gx\circ\alpha_{l^{\star} } }\pare{ 0, u}$ is $\N_{n\lambda}$-integrable. Then by the LLN applied to the i.i.d. sequence $\{U_{k+1} \,; k \in \mathbb{N} \}$, $ \frac1k \log\frac{ \sigma_{k} }{ \sigma_{0} } $ converges almost surely to $\mathbb{E}\croc{\log\pare{\Gl}} $.
\end{proof}
\new{Geometric divergence of the step-size on a linear function is also the main condition when analyzing the deterministic flow of the IGO algorithm \cite{akimoto2012convergence}.}
\new{For the $\pare{1\!+\!1}$-ES and the $\pare{1, \lambda}$ self-adaptive ES,
a different condition \nnew{than $\mathbb{E}\croc{ \log\pare{\Gl} } > 0$} has been used to characterize the step-size increase on linear functions:
there exists $\beta > 0$ such that $\mathbb{E}\croc{ \frac{1}{\Gl^{\beta} }} < 1$
\cite{auger2013linear, auger2005convergence}.}\del{
We find in the literature a different condition \nnew{than $\mathbb{E}\croc{ \log\pare{\Gl} } > 0$} for the $\pare{1+1}$-ES~\cite{auger2013linear} and the $\pare{1, \lambda}$ self-adaptive ES~\cite{auger2005convergence}, that \del{is}\nnew{characterizes the step-size increase on linear functions.}\del{ ``the step-size increases on linear functions''}.\todo{}\niko{do we need the ambiguous phrases when we have the unambiguous conditions?}\anne{I took it out} This condition is formally stated as the existence of $\beta > 0$ such that $\mathbb{E}\croc{ \frac{1}{\Gl^{\beta} }} < 1$.}
With the concavity of the logarithm \nnew{and}\del{thanks to} Jensen's inequality, we have that $\log \pare{\mathbb{E}\croc{\frac{1}{\Gl^\beta } } } \geq
 \mathbb{E}\croc{\log \pare{\frac{1}{\Gl^\beta} } } = -\beta\, \mathbb{E}\croc{\log\pare{\Gl } }$. Therefore\del{ if} $\mathbb{E}\croc{ \frac{1}{\Gl^{\beta} }} < 1$ \new{implies}\del{, then} $ \mathbb{E}\croc{\log\pare{\Gl } } > 0$ \nnew{and}\del{. Thereby} our condition \nnew{that} ``the logarithm of the step-size increases on linear functions'' is tighter than \del{traditional}\new{the previously used.}\del{ condition ``the step-size increases on linear functions''.}

\nnew{Our main condition for the linear behavior of} the $\pare{\mu/\mu_w, \lambda}$-CSA-ES algorithm without cumulation\del{, our main condition in Proposition~\ref{prop:weight-condition} for the linear behavior} is formulated based on $\lambda$, $\mu$, the weights $w$ and the order statistics of the standard normal distribution \nnew{as}\del{. It reads} $\mathbb{E}\Bigl[\pare{\sum_{i=1}^{\mu} \frac{w_{i}}{\| w\|  } \N^{i:\lambda} }^{\!2} \Big]  > 1$\new{, see Proposition~\ref{prop:weight-condition}.} For $\mu = 1$, this condition is satisfied when $\lambda \geq 3$.\del{, thanks to Proposition~\ref{prop:weight-condition}.}

\del{In~\cite{chotard2012cumulative}\done{},}
The linear divergence of both the incumbent and the step-size \new{was proven for}\del{is obtained in} a $\pare{1, \lambda}$\nnew{-ES}\del{ scenario} without cumulation on linear functions whenever $\lambda \geq 3$ with a divergence rate equal to $\frac{ \mathbb{E}\croc{ (\N^{1:\lambda})^{2} } - 1 }{2 d_{\sigma} n} $ \cite{chotard2012cumulative}. 
\new{Proposition~\ref{prop:weight-condition} extends this result to values of $\mu>1$.}\del{This result is therefore incorporated in Proposition~\ref{prop:weight-condition}.}
Note that \del{we have simultaneously}\new{our results cover both,} linear divergence on strictly increasing transformations of nontrivial linear functions and linear behavior on strictly increasing transformations of $C^{1}$ scaling-invariant functions with a unique global argmin. 

\del{While our framework does not include cumulation by a path for the step-size update via CSA~\cite{hansen2001completely},\del{ cumulation is encompassed in \cite{chotard2012cumulative}\done{} and} linear divergence of the step-size holds on linear functions for the $\pare{1,\lambda}$-CSA-ES \new{also with a cumulation path \cite{chotard2012cumulative}}.
The key aspect consists in applying an LLN to the cumulation path. Linear divergence is only proven for the step-size as it requires the application of the LLN to a more complex Markov chain to prove it for the mean~\cite{chotard2012cumulative}.
}

Our methodology leans on investigating the stability of the \normalized\ homogeneous Markov chain to be able to apply an LLN and obtain the limit of the log-distance to the optimum divided by the iteration index. Then we obtain an exact expression of the rate of convergence or divergence as an expectation with respect to the stationary distribution of the \normalized\ chain. This is an elegant feature of our analysis. Other approaches \cite{jagerskupper2003analysis,jagerskupper2007algorithmic,jagerskupper2005rigorous,jagerskupper20061+,akimoto2018drift,akimoto2020} provide bounds on the convergence rate but not its exact expression \new{with the advantage that the} \del{.B}\new{b}ounds are often expressed depending on dimension or population size \del{which are relevant parameters in practice}\new{and thus describe the scaling of the algorithm with respect to relevant parameters}.

The class of scaling-invariant functions is, as far as we can see, the largest class to which our methodology can conceivably be applied---because on any wider class of functions, a selection function for the \normalized\ Markov chain can not anymore reflect the selection operation in the underlying chain. We require additionally that the objective function is a strictly increasing transformation of either a continuously differentiable function with a unique global argmin or a nontrivial linear function. Many non-convex functions with non-convex sublevel sets are included.

The implied requirement of smooth level sets is instrumental for our analysis.
We believe that there exist unimodal functions with non-smooth level sets on which scale-invariant ESs can not converge to the global optimum with probability one
\new{independently of the initial conditions},
for example $x\mapsto\sum_{i=1}^{n} \sqrt{|x_i|}$.
However,\del{ we also believe that} smooth level sets are not a necessary condition for convergence---we consistently observe convergence on $x\mapsto \sum_{i=1}^{n} |x_i|$ for smaller values of $n$ and understand the reason why ESs succeed on the one-norm but fail on the $\nicefrac12$-norm function.
Capturing this distinction in a rigorous analysis of the Markov chain remains an open challenge.

\del{In contrast, the approach used in~\cite{akimoto2020}\done{} allows to handle functions that are not scaling-invariant.}
\new{A broader function class has been analyzed by requiring}\del{.
This approach requires}
a drift condition to hold on the whole state-space \cite{akimoto2020} while our methodology requires\del{ that} the drift condition \new{to} only hold outside of a small set \del{which means here }(when the step-size is much smaller than the distance to the optimum).
Hence in our approach, it suffices to control the behavior in the limit when the step-size normalized by the distance to the optimum approaches zero.

A major limitation of our current analysis is the omission of cumulation that is used in the $\muslashmu$-CSA-ES to adapt the step-size (we have set the cumulation parameter to 1, see Section~\ref{sec:alg-encompassed}).
In case of a parent population of size $\mu = 1$, Chotard et al.~\cite{chotard2012cumulative} obtain linear divergence of the step-size on linear functions also with cumulation.
However, no proof of linear behavior exists, to our knowledge, on functions whose level sets are not affine subspaces.
While we consider cumulation a crucial component in practice, proving the drift condition for the stability of the Markov chain is much harder when the state space is extended with the cumulative evolution path and this remains an open challenge.

Technically, our results rely on proving $\varphi$-irreducibility, positivity and Harris-recurrence of the \normalized\ Markov chain. The $\varphi$-irreducibility is difficult to prove directly for the class of algorithms studied in this paper while it is relatively easy to prove for the $(1,\lambda)$-ES with self-adaptation~\cite{auger2005convergence} or for the (1+1)-ES with one-fifth success rule~\cite{auger2013linear}. \new{We circumvented the problem by looking at the stability of an underlying deterministic control model and exploit its connection to the stability of Markov chains
%and stability of deterministic control models
\cite{chotard2019verifiable}.}\del{With \new{better}\del{the} tools\niko{tools is a little strange and pretty generic word}\del{ developed in} \cite{chotard2019verifiable},\done{} proving $\varphi$-irreducibility, aperiodicity and a T-chain property is much easier, illustrating how useful the connection between stability of Markov chains with stability of deterministic control models can be.}
Positivity and Harris-recurrence are proven using Foster-Lyapunov drift conditions~\cite{meyn2012markov}. We prove a drift condition for geometric ergodicity that implies positivity and Harris-recurrence. \new{The main ingredient for obtaining the drift condition is}\del{ It relies on} the convergence in distribution of the step-size change towards the step-size change on a linear function when $Z_k=z$ goes to infinity. \new{It implies that the drift condition holds for $Z_k=z$ outside a compact set}. We also prove in Lemma~\ref{non-negligible} the existence of non-negligible sets with respect to the invariant probability measure $\pi$, outside of a neighborhood of a steadily attracting state. This is used in Proposition~\ref{log-integrable} to obtain the $\pi$-integrability of the function $z \mapsto \log \|z\|$.

We have developed generic results to facilitate further studies of similar Markov chains. More specifically, applying an LLN to the \normalized\ chain is not enough to conclude linear convergence. We introduce the technique to apply the generalized LLN to an abstract chain $\{ \pare{Z_k, U_{k+2}} ; k \in \mathbb{N}\}$ and prove that stability properties from $\{Z_k; k \geq 0\}$ are transferred to $\{ \pare{Z_k, U_{k+2}} ; k \in \mathbb{N}\}$.

\section*{Acknowledgements}
Part of this research has been conducted in the context of a research collaboration between Storengy and Inria. We particularly thank F. Huguet and A. Lange from Storengy for their strong support.

\bibliographystyle{spmpsci}
\bibliography{saes}

\appendix
\renewcommand*{\thesection}{\Alph{section}}

\section{ Proof of Proposition~\ref{linear-invariance}}
\label{proof-linear-invariance}
With Lemma~\ref{selection-function-increasing-transformation}, we assume without loss of generality that $\f$ is a nontrivial linear function.
Let us remark beforehand that the random variable $\alpha_{f}(z, U_{1})$ does not depend on $z$ thanks to Lemma~\ref{alpha-linearity}. 
Let $\varphi: \Gamma(\R^{n\mu}) \to \R$ be a continuous and bounded function, it is then enough to prove that
$\mathbb{E}_{U_{1} \sim \N_{\n \lambda}}\croc{ \varphi\pare{\Gx( \alpha_{f}(z, U_{1}) )} }= \mathbb{E}_{U_{1} \sim \N_{\n \lambda}}\croc{ \varphi\pare{\Gx( \alpha_{l^{\star}}(0, U_{1}) ) }}.$ Denote by $e_{1}$ the unit vector $\pare{1,0,\dots,0}$, then for all $x \in \R^{\n}$, $l^{\star}(x) = e_{1}^{\top}x$. 
Denote by $\tilde{e}_{1}$ the \normalized\ gradient of $f$ at some point. Then there exists $K > 0$ such that
for all  $x \in \R^{n},$ $f(x) = K \tilde{e}_{1}^{\top} x$. And by the Gram-Schmidt process, there exist $\pare{e_{2},\dots, e_{n}}$ and $\pare{\tilde{e}_{2},\dots, \tilde{e}_{n}}$ such that 
$\pare{e_{1}, e_{2},\dots, e_{n}}$ and $\pare{\tilde{e}_{1}, \tilde{e}_{2},\dots, \tilde{e}_{n}}$  are orthonormal bases.
Denote by $T$ the linear function defined as
$T(e_{i}) = \tilde{e}_{i} \text{ for $i = 1,\dots, n$}.
$ Then $T$ is an orthogonal matrix. For all $x\in \R^{n}$, 
%\begin{align}
$\tilde{e}_{1}^{\top} T(x) = e_{1}^{\top}  x, \text{ and } \|T(x)\| = \|x\|.$
%\label{rotation}
%\end{align}
Denote $A = \mathbb{E}_{U_{1} \sim \N_{\n \lambda}}\croc{ \varphi\pare{\Gx( \alpha_{f}(z, U_{1}) )} } .$ 
We do a change of variable $u \mapsto \pare{T(u^{1}),\dots,T(u^{\mu})}$. Then
%\begin{align*}
$\frac{(\lambda - \mu)!}{\lambda !} A =
 \int \varphi\pare{ \Gx(u)} \mathds{1}_{\tilde{e}_{1}^{\top} (u^{2}-u^{1}) \, > \, 0, \dots, \tilde{e}_{1}^{\top} (u^{\mu}-u^{\mu-1}) \,> \,0 }$ \\
 $P\pare{\tilde{e}_{1}^{\top} \N_{\n}  > \tilde{e}_{1}^{\top} u^{\mu}  }^{\lambda-\mu} p_{\N_{\n}}(u^{1}) \dots p_{\N_{\n}}(u^{\mu}) \mathrm{d}u^{1}\dots\mathrm{d}u^{\mu}$
$ = \int  \varphi\pare{ \Gx\pare{ T(u^{1}), \dots, T(u^{\mu})} }$ \\$\mathds{1}_{e_{1}^{\top} (u^{2}-u^{1}) \,> \,0, \dots, e_{1}^{\top} (u^{\mu}-u^{\mu-1}) \,> \,0 } P\pare{e_{1}^{\top} \N_{\n}  > e_{1}^{\top} u^{\mu}  }^{\lambda-\mu}p_{\N_{\n}}(T(u^{1})) \dots p_{\N_{\n}}(T(u^{\mu})) \mathrm{d}u^{1}\dots$ \\$\mathrm{d}u^{\mu}$,
%\end{align*}
thanks to the fact that $e_{1}^{\top} \N_{\n} \sim \tilde{e}_{1}^{\top} \N_{\n} \sim \N(0, 1).$ Since $\Gx$ and $p_{\N_{\n}}$ are invariant under rotation, $\mathbb{E}_{U_{1} \sim \N_{\n \lambda}}\croc{ \varphi\pare{\Gx( \alpha_{f}(z, U_{1}) )} } = \mathbb{E}_{U_{1} \sim \N_{\n \lambda}}\croc{ \varphi\pare{\Gx( \alpha_{l^{\star}}(0, U_{1}) )} }$. \quad  

\section{ Proof of Proposition~\ref{general-stability} }
\label{proof-general-stability}

We have $Z_{k+1} = G(Z_{k}, U_{k+1} )$ and $U_{k+3}$ is independent from $\{ W_t\, ; t \leq k \}$, then $\{W_k\, ; k \in \mathbb{N}\}$ is a Markov chain on $\mathcal{B}(\ZZ) \otimes \mathcal{B}(\R^{m})$. 
Let $\pare{A, B} \in \mathcal{B}(\ZZ) \times \mathcal{B}(\R^{m})$ and $(z, u) \in \ZZ \times \R^{m}$. Then by independence 
%\begin{align*}
$P\pare{ (Z_{t+1}, U_{t+3} )\in A\times B | (Z_t,U_{t+2})=(z,u) } = P\pare{ Z_{t+1}\in A | Z_{t}=z} P\pare{ U_{t+3} \in B }.$
%\end{align*}
For $\pare{A, B} \in \mathcal{B}(\ZZ) \times \mathcal{B}(\R^{m})$, for $(z, u) \in \ZZ \times \R^{m}$, $\sum_{k=1}^{\infty} P^{k}((z,u), A\times B) =  \Psi(B)\, \sum_{k=1}^{\infty} P^{k}(z, A).$ Therefore $\sum_{k=1}^{\infty} P^{k}((z,u), \cdot)$ is a product measure. 

Let $\varphi$ be an irreducible measure of $\{Z_k \, ; k \in \mathbb{N}\}$ and let $E \in \mathcal{B}(\ZZ) \otimes \mathcal{B}(\R^{m})$. By definition of a product measure, $\pare{\varphi \times \Psi }(E) = \dsp \int \varphi(E^{v}) \Psi(\diff v)$ and thus $\sum_{k=1}^{\infty} P^{k}((z,u), E)  =  \dsp \int \sum_{k=1}^{\infty} P^{k}(z, E^{v}) \Psi(\diff v) $
%\begin{align*}
%\pare{\varphi \times \Psi }(E) &= \dsp \int \varphi(E^{v}) \Psi(\diff v)\\
%\sum_{k=1}^{\infty} P^{k}((z,u), E) &=  \dsp \int \sum_{k=1}^{\infty} P^{k}(z, E^{v}) \Psi(\diff v)
%\end{align*}
 where $E^{v} = \acco{z \in \ZZ\, ; (z, v) \in E}$. 
 If $\sum_{k=1}^{\infty} P^{k}((z,u), E) = 0$, then $0 = \sum_{k=1}^{\infty} P^{k}(z, E^{v})$ for almost all $v$ and then $\varphi(E^{v}) = 0$ for almost all $v$. Then $\pare{\varphi \times \Psi }(E) = \dsp \int \varphi(E^{v}) \Psi(\diff v) = 0$, hence the $\pare{\varphi \times \Psi }$-irreducibility of $\{W_k \,; k \in \mathbb{N}\}$.
 
Let us show that $\pi \times \Psi$ is an invariant probability measure of $\{W_k \, ; k \in \mathbb{N}\}$ when $\pi$ is an invariant measure of $\{Z_k \, ; k \in \mathbb{N}\}$.
Assume that $\pare{A, B} \in \mathcal{B}(\ZZ) \times \mathcal{B}(\R^{m})$. Then 
%\begin{align*}
%&\dsp\int P\pare{ (Z_{1}, U_{3} )\in A\times B | (Z_0,U_2)=(z,u)} (\pi\times\Psi )\pare{\diff (z,u)} = \\
%& \dsp\int P_{z}\pare{Z_{1} \in A }\Psi(B) \pi(\diff z) \Psi(\diff u) = \Psi(B) \pi(A) = (\pi\times \Psi )(A\times B).
%\end{align*}
$\int P\pare{ (Z_{1}, U_{3} )\in A\times B | (Z_0,U_2)=(z,u)} (\pi\times\Psi )\pare{\diff (z,u)} = 
\dsp\int P_{z}\pare{Z_{1} \in A }\Psi(B) \pi(\diff z) \Psi(\diff u) = \Psi(B) \pi(A) = (\pi\times \Psi )(A\times B).
$
Hence $\pi \times \Psi$ is an invariant probability of $\{W_k \,; k \in \mathbb{N}\}$.
Assume that $\{W_k \, ; k \in \mathbb{N}\}$ has a $d$-cycle $\pare{D_{i}}_{i=1,\dots,d} \in \pare{\mathcal{B}(\ZZ)\otimes\mathcal{B}(\R^{m} )}^{d}$. Define for $i=1,\dots,d$, $\widetilde{D}_{i} = \acco{ z \in \ZZ | \exists \, u\in \R^{m}\,; (z,u) \in D_{i}}$ and let us prove that $\pare{  \widetilde{D} _{i}}_{i=1,\dots,d}$ is a $d$-cycle of $\{Z_k \,; k \in \mathbb{N}\}$. 

Let $z \in \widetilde{D}_{i}$ and $i = 0, \dots, d-1$ (mod  $d$). There exists $u \in \R^{m}$ such that $(z, u) \in D_{i}$. Then $1 = P((z,u), D_{i+1}) = P\pare{ \pare{Z_{1}, U_{3} } \in D_{i+1} | Z_{0}=z} \leq P\pare{ Z_{1} \in \widetilde{D}_{i+1} | Z_{0}=z}$.  
Thereforer $P\pare{ Z_{1} \in \widetilde{D}_{i+1} | Z_{0}=z} = 1$.  

 If $\Lambda$ is an irreducible measure of $\{Z_k \, ; k \in \mathbb{N}\}$, then we have proven above that $\Lambda \times \Psi$ is an irreducible measure of $\{W_k\,; k \in \mathbb{N}\}$. Then $0 = \pare{\Lambda\times \Psi } \pare{ \pare{ \bigcup_{i=1}^{d} D_{i} }^{c} }$. For $i = 1,\dots, d$, 
 $\pare{\Lambda\times \Psi } \pare{ D_{i} } = \dsp\int \Lambda(D_{i}^{v}) \Psi (\diff v) \leq \dsp\int \Lambda(\widetilde{D}_{i}) \Psi(\diff v) = \Lambda(\widetilde{D}_{i})$. Then $\Lambda \pare{ \bigcup_{i=1}^{d} \widetilde{D}_{i} } = \sum_{i=1}^{d} \Lambda( \widetilde{D}_{i}) \geq \pare{\Lambda\times\Psi } \pare{\bigcup_{i=1}^{d} D_{i} } $. Hence $\Lambda\pare{ \pare{ \bigcup_{i=1}^{d} \widetilde{D}_{i} }^{c} }$ $=$ $0$ and finally we have a $d$-cycle for $\{Z_k \, ; k \in \mathbb{N}\}$. 
Similarly we can show that if $\{Z_k \, ; k \in \mathbb{N}\}$ has a $d$-cycle, then $\{W_k \, ; k \in \mathbb{N}\}$ also has a $d$-cycle.
Now assume that $C$ is a small set of $\{Z_k\, ; k \in \mathbb{N}\}$. Then there exists a positive integer $k$ and a nontrivial measure $\nu_{k}$ on $\mathcal{B}(\ZZ)$ such that $P^{k}(z, A) \geq \nu_{k}(A)$ for all $z \in C, \, A \in  \mathcal{B}(\ZZ)$. If $(z, u) \in C \times \R^{m}$ and $E \in \mathcal{B}(\ZZ) \otimes \mathcal{B}( \R^{m} )$, $P^{k}((z,u), E) \geq \pare{\nu_{k}\times\Psi}(E)$ and therefore
$C \times \R^{m}$ is a small set of $\{W_k \, ; k \in \mathbb{N} \}$.
The drift condition for $\{W_k\,; k \in \mathbb{N}\}$ follows directly from the drift condition for $\{Z_k  \, ; k \in \mathbb{N}\}$. 

\section{ Proof of Proposition~\ref{dynamic-linear} }
\label{proof-dynamic-linear}

To prove the convergence in distribution of the step-size multiplicative factor for a function $\f$ that satisfies F1 or F2, we use the intermediate result given by Proposition~\ref{dynamic-linear}, that asymptotically links $\Gx\pare{\alpha_{\f}(x^{\star}+z, U_{1})}$ to the random variable $\Gx\pare{\alpha_{l_{z}^{\f}}(z, U_{1})}$ where the nontrivial linear function $l_{z}^{\f}$ depends on $z$, $\nabla \f$, and is introduced in~\eqref{intermediate-linear}. 
Since $\alpha_{\f}(x^{\star} + z, U_{1}) = \alpha_{\tilde{\f}}(z, U_{1})$, we assume without loss of generality that $x^{\star} = 0$ and $\f(0) = 0$.

The next lemma is our fist step towards understanding the asymptotic behavior of $\alpha_{f}\pare{z, U_{1}}$ for a $C^{1}$ scaling-invariant function $f$ with a unique global argmin. For $\varphi: \R^{n\mu} \to \R$ continuous and bounded, we approximate $\mathbb{E}\croc{ \varphi( \alpha_{\f}(z, U_{1}) ) }$ by using the explicit definition of $p_{z}^{\f}$ in~\eqref{mupositive}, and observing the integrals in the balls $\overline{\mathbf{B}\pare{0, \sqrt{\|z\|}  }  }  $, such that the $f$-values we consider are relatively close to the $f$-values of $\frac{t_{z}^{f}}{\| z \|}z \in \level_{\f, z_{0}}$.

\begin{lemma}
Let $f$ be a $C^{1}$ scaling-invariant function with a unique global argmin \new{assumed to be in $0$ such that $f(0)=0$}. For $(z, w, v) \in (\R^{n})^{3}$, define the function $h: (z, w, v) \mapsto  \mathds{1}_{\acco{ f\pare{\frac{t_{z}^{f} }{\| z \|}z + \frac{t_{z}^{f} }{\| z \|} w } > f\pare{\frac{t_{z}^{f} }{\| z \|} z + \frac{t_{z}^{f}}{\| z \|} v } }}$.
Then for all $\varphi: \R^{n\mu} \to \R$ continuous and bounded:\\
$\lim_{\| z \| \to \infty} \int_{ \|u\|  \leq \sqrt{\| z \|} }   \left( \int_{ \|w\|  \leq \sqrt{\| z \|} } h(z, w, u^{\mu}) p_{\N_{\n}}(w) \mathrm{d}w \right)^{\lambda-\mu} $\\
$\frac{\lambda !} { (\lambda - \mu)! } \varphi(u) \prod_{i=1}^{\mu-1} h(z, u^{i+1}, u^{i}) \prod_{i=1}^{\mu} p_{\N_{\n}}(u^{i}) \mathrm{d}u  - \int \varphi(u) p_{z}^{f}(u) \mathrm{d}u 
= 0. $
\label{bounded} 
\end{lemma}
\begin{proof}%[Proof of Lemma~\ref{bounded}]
 For $z \in \R^{\n}$ and $u \in \R^{n\mu}$, define $A(z) = \biggr\lvert \int \varphi(u) p_{z}^{f}(u) \mathrm{d}u  -  \int_{\|u\|\leq \sqrt{\| z \|} } \varphi(u) p_{z}^{f}(u) \mathrm{d}u \biggr\rvert$. 
Then $A(z) \leq \frac{\lambda !} { (\lambda - \mu)! } \| \varphi \|_{\infty} \int_{\| u\|  > \sqrt{\| z \|} } \prod_{i=1}^{\mu} p_{\N_{\n}}(u^{i}) \mathrm{d}u = \frac{\lambda !} { (\lambda - \mu)! }  \|\varphi\|_{\infty} \int_{\|u\|  > \sqrt{\| z \|} } p_{\N_{\n \mu}}(u) \mathrm{d}u 
= \frac{\lambda !} { (\lambda - \mu)! }  \| \varphi \|_{\infty} \pare{1 - P\pare{ \| \N_{\n \mu} \| \leq \sqrt{\| z \|} } }$.

Then by scaling-invariance with a multiplication by $t_{z}^{f} / \| z \|$,
%\begin{align}
$\lim_{\| z \| \to \infty} \int_{ \|u\|  \leq \sqrt{\| z \|} } \varphi(u)$ \\$\pare{ \mathbb{E}\croc{h(z, \N_{n}, u^{\mu})} }^{\lambda-\mu} 
 \dsp \prod_{i=1}^{\mu-1} h(z, u^{i+1}, u^{i}) \prod_{i=1}^{\mu} p_{\N_{\n}}(u^{i}) \mathrm{d}u - \frac{ (\lambda - \mu)! }{\lambda !}  \int \varphi(u) p_{z}^{f}(u) \mathrm{d}u = 0.$ 
%\label{scaling-probas}
%\end{align}
Also, $\mathbb{E}\croc{ h(z, \N_{n}, u^{\mu}) } - \int_{\|w\|  \leq \sqrt{\| z \|} }  h(z, w, u^{\mu}) p_{\N_{\n}}(w) \diff w =$
 $ \int_{ \|w\|  > \sqrt{\| z \|} } 
h(z, w, u^{\mu})
% \mathds{1}_{ f\pare{\frac{t_{z}^{f}}{\| z \|}z + \frac{t_{z}^{f}}{\| z \|} w } > f\pare{\frac{t_{z}^{f}}{\| z \|} z + \frac{t_{z}^{f}}{\| z \|} u^{\mu}} }
 p_{\N_{\n}}(w) \diff w \leq 1 - P\pare{ \| \N_{\n}\| \leq \sqrt{\| z \|} }$.
 %\end{align*}
 Hence along with the dominated convergence theorem, the lemma is proven.  
\end{proof}

%\begin{proof}[Proof of Proposition~\ref{dynamic-linear}]
%\todo{isn't it that we need $x^\star$ in some formula}
\new{\noindent We are now ready to prove the proposition.}\\
Let $\varphi: \R^{n\mu} \to \R$ be continuous and bounded. Using Lemma \ref{bounded}, it is enough to prove that
$ \lim_{ \| z \| \to  \infty }\int_{ \|u\|  \leq \sqrt{\| z \|} } \pare{ \int_{ \|w\|  \leq \sqrt{\| z \|} } 
h(z, w, u^{\mu})  p_{\N_{\n}}(w) \mathrm{d}w }^{\lambda-\mu}  \varphi(u)$ \\$\prod_{i=1}^{\mu-1} h(z, u^{i+1}, u^{i})  \prod_{i=1}^{\mu} p_{\N_{\n}}(u^{i}) \mathrm{d}u 
-  \frac{ (\lambda - \mu)! }{\lambda !} \int \varphi(u) p_{z}^{ l_{z}^{f} }(u)\mathrm{d}u = 0.
$
We define the function $g$ on the compact  $\pare{\mathcal{L}_{f,z_{0}^{\f}} + \overline{\mathbb{B}(0, \delta_{f})} }\times [0, \delta_{f}]$ as,  for $(x, \rho) \in \pare{\mathcal{L}_{f,z_{0}^{f}} + \overline{\mathbb{B}(0, \delta_{f})} }\times (0, \delta_{f}],$ $g(x, \rho) =$
$
   \int_{ \|u\| \leq \frac{1}{\sqrt{\rho} } } \pare{\int_{ \|w\| \leq \frac{1}{\sqrt{\rho} }   } 1_{\theta(w,u^\mu,x)  \,> \,0 }  \,p_{\N_{\n}}(w) \mathrm{d}w}^{\lambda-\mu} \!\!\!\!\!\!
   \varphi(u) \prod_{i=1}^{\mu-1} \mathds{1}_{
   \theta(u^{i+1},u^{i},x) \,> \,0} 
p_{\N_{\n \mu}(u)}
\mathrm{d}u
$, 
with $\theta(w,v,x) = (w-v)^{\top} \nabla f\pare{ x + t_{x}^{f} \rho (v+ \tau_{x}^{\rho}(v,w) (w-v))  } $ and  $\tau_{x}^{\rho}(v^{1}, v^{2})\in\pare{0,1}$ defined thanks to the mean value theorem by 
$
f(x + t_{x}^{f}\rho v^{2}) - f(x+ t_{x}^{f}\rho v^{1}) =  t_{x}^{f}
\theta(v^2,v^1,x)
$.
For $x \in \mathcal{L}_{f, z_{0}^{f}} + \overline{\mathbb{B}(0, \delta_{f})},$\del{ $g(x, 0)$ equals} 
$g(x, 0) = \int \varphi(u) P\Big( (\N_{\n}-u^{\mu})^{\top} \nabla f(x) > 0 \Big) ^{\lambda-\mu} 
\prod_{i=1}^{\mu-1} \mathds{1}_{ (u^{i+1}-u^{i})^{\top} \nabla f(x) \,> \,0}
p_{\N_{\n\mu}}(u)
\mathrm{d}u.$
Note that
$ g\pare{ t_{z}^{f}\frac{z}{\| z \|},  0 } =  \frac{ (\lambda - \mu)! }{\lambda !}  \dsp\int \varphi(u) p_{z}^{l_{z}^{f}}(u)  \mathrm{d}u$. With Lemma \ref{bounded},
$ \lim_{ \| z \| \to \infty } g\pare{ t_{z}^{f}\frac{z}{\| z \|}, \frac{ 1 }{\| z \|}} -  \frac{ (\lambda - \mu)! }{\lambda !} \dsp\int \varphi(u) p_{z}^{f}(u)  \mathrm{d}u = 0.$ Therefore it is enough to prove that $g$ is uniformly continuous in order to obtain that 
$ \frac{(\lambda-\mu) !}{\lambda !} \pare{ \lim_{ \| z \| \to \infty }  \int \varphi(u) p_{z}^{f}(u)  \mathrm{d}u
-   \int \varphi(u) p_{z}^{l_{z}^{f}}(u)   \mathrm{d}u}$ is equal to $\lim_{ \| z \| \to \infty } g\pare{ t_{z}^{f}\frac{z}{\| z \|}, \frac{ 1 }{\| z \|}} -  g\pare{ t_{z}^{f}\frac{z}{\| z \|},  0 }$ which is equal to $0.$

For $(x, \rho) \in \pare{\mathcal{L}_{f, z_{0}^{f}} + \overline{\mathbb{B}(0, \delta_{f})} }\times (0, \delta_{f}],$ for $u \in  \overline{\mathbb{B}(0, 1 / \sqrt{\rho} )},$ $w \in  \overline{\mathbb{B}(0, 1 / \sqrt{\rho} )},$
$\nabla f\Big( x+ t_{x}^{f}\rho (u^{\mu}+ \tau_{x}^{\rho}(u^{\mu}, w) (w-u^{\mu})) \Big) \neq 0$
since $x+ t_{x}^{f}\rho (u^{\mu}+ \tau_{x}^{\rho}(u^{\mu}, w) (w-u^{\mu})) \in \mathcal{L}_{f, z_{0}^{f}} + \overline{\mathbb{B}(0, 2 \delta_{f})}.$
Then the set 
$\acco{w \in \R^{\n}; 
\theta(w,u^\mu,x)
= 0}$ is Lebesgue negligible. In addition, the function $y \mapsto \mathds{1}_{y \,> \,0}$ is continuous on $\mathbb{R} \backslash \acco{0}$, therefore  for almost all $w,$ 
$(x, \rho, u^{\mu}) \mapsto \mathds{1}_{ \|w\| \leq \frac{1}{\sqrt{\rho} }} \mathds{1}_{  
\theta(w,u^\mu,x)
}   p_{\N_{\n}}(w) 
$
 is continuous and bounded by the integrable function $ p_{\N_{\n}}$. Then by domination, for almost all  $u,$ $ (x, \rho) \mapsto$
 $
 \mathds{1}_{ \|u\| \leq \frac{1}{\sqrt{\rho} }} \pare{\int_{ \|w\| \leq \frac{1}{\sqrt{\rho} }   } \mathds{1}_
 {
 \theta(w,u^\mu,x)
 > \,0 }  
 \,p_{\N_{\n}}(w) \mathrm{d}w}^{\lambda-\mu}
 $
is continuous.
Similarly $(x, \rho) \mapsto  \mathds{1}_{ \|u\| \leq \frac{1}{\sqrt{\rho} }} \prod_{i=1}^{\mu-1} \mathds{1}_{ 
\theta(u^{i+1},u^i,x)
\, > \,0}
$
 is continuous for almost all $u$.
Therefore $g$ is continuous on $\pare{\mathcal{L}_{f,z_{0}^{f}} + \overline{\mathbb{B}(0, \delta_{f})} }\times (0, \delta_{f}],$ and for all $x \in \mathcal{L}_{f, z_{0}^{f}} + \overline{\mathbb{B}(0, \delta_{f})},$ $\lim_{\rho \to 0} g(x, \rho)$ exists and equals
\begin{align*}
 &  \int \lim_{\rho \to 0}  \mathds{1}_{ \|u\| \leq \frac{1}{\sqrt{\rho} } } 
 \varphi(u) \prod_{i=1}^{\mu-1} \mathds{1}_{ 
 \theta(u^{i+1},u^i,x)
 > 0} 
 \pare{\int_{ \|w\| \leq \frac{1}{\sqrt{\rho} }   } 1_{ 
\theta(w,u^\mu,x)
> 0 }  \,p_{\N_{\n}}(w) \mathrm{d}w}^{\lambda-\mu} \!\!\!\!\!
p_{\N_{\n \mu}(u)}
\mathrm{d}u 
\end{align*}
which is equal to $g(x, 0)$.
Finally $g$ is continuous on the compact $\pare{\mathcal{L}_{f, z_{0}^{f}} + \overline{\mathbb{B}(0, \delta_{f})} }\times [0, \delta_{f}]$; it is thereby uniformly continuous on that compact.
%\end{proof}

\end{document}